\documentclass[11pt]{amsart}

\usepackage{amssymb,amsmath,amsfonts,amsthm,mathrsfs,fullpage}
\usepackage{moreverb,enumitem}
\usepackage{color}
\numberwithin{equation}{section}

\newcommand{\CC}{\mathbb C}

\newcommand{\FF}{\mathbb F}
\newcommand{\GG}{\mathbb G}

\newcommand{\QQ}{\mathbb Q}
\newcommand{\RR}{\mathbb R}

\newcommand{\ZZ}{\mathbb Z}

\newcommand{\calB}{\mathcal B}

\newcommand{\calG}{\mathcal G}

\newcommand{\OO}{\mathcal O}
\newcommand{\calP}{\mathcal P}

\newcommand{\calS}{\mathcal S}
\newcommand{\calW}{\mathcal W}

\newcommand{\calV}{\mathcal V}

\newcommand{\scrS}{\mathscr S}

\newcommand{\q}{\mathfrak q}
\newcommand{\p}{\mathfrak p}
\renewcommand{\P}{\mathfrak P}

\def\cl{\operatorname{cl}}

\def\ST{\operatorname{ST}}

\newcommand\commentr[1]{{\color{red}[#1]}}  
  
\newcommand\commentinv[1]{{}}

\newcommand{\ang}[1]{ \langle #1 \rangle  }

\def\Spec{\operatorname{Spec}} 

 \def\Gal{\operatorname{Gal}}
\def\End{\operatorname{End}}

\def\cl{\operatorname{cl}}

\def\conn{{\operatorname{conn}}}
\def\der{{\operatorname{der}}}

\def\Spec{\operatorname{Spec}}

\def\Gal{\operatorname{Gal}}

\def \Hg {\operatorname{Hg}}  
\def \MT {\operatorname{MT}}

\def \GL{\operatorname{GL}}

\def \Sp  {\operatorname{Sp}} 
\def \SO {\operatorname{SO}}

\def \GSp  {\operatorname{GSp}}

\def\Aut{\operatorname{Aut}} 
\def\Out{\operatorname{Out}}
\def\Inn{\operatorname{Inn}}
\def\inn{\operatorname{inn}}

\def\End{\operatorname{End}}
\def\Frob{\operatorname{Frob}}

\def\rank{\operatorname{rank}}
\def\MT{\operatorname{MT}}

\def\Res{\operatorname{Res}}

\newcommand{\Hom}{\operatorname{Hom}}

 \def \Aut {\operatorname{Aut}}

\def\bbar#1{\setbox0=\hbox{$#1$}\dimen0=.2\ht0 \kern\dimen0 
\overline{\kern-\dimen0 #1}}
\newcommand{\Qbar}{{\overline{\mathbb Q}}} 
\newcommand{\Kbar}{\bbar{K}}

\newcommand{\kbar}{\bbar{k}}

\newcommand{\defi}[1]{\textsf{#1}}  
\newtheorem{thm}{Theorem}[section]
\newtheorem{lemma}[thm]{Lemma}
\newtheorem{cor}[thm]{Corollary}
\newtheorem{prop}[thm]{Proposition}
\newtheorem{algorithm}[thm]{Algorithm}

\newtheorem{defn}[thm]{Definition}

\newtheorem{conj}[thm]{Conjecture}

\theoremstyle{remark}
\newtheorem{remark}[thm]{Remark}

\newenvironment{alphenum}{\hfill \begin{enumerate}[label=({\alph*})] 
}
{\end{enumerate}}

\newenvironment{romanenum}{\hfill \begin{enumerate}[label=({\roman*})] 
}
{\end{enumerate}}

\definecolor{webcolor}{rgb}{0.8,0,0.2}
\definecolor{webbrown}{rgb}{.6,0,0}
\usepackage[
        colorlinks,
        linkcolor=webbrown,  filecolor=webcolor,  citecolor=webbrown, 
        backref,
        pdfauthor={David Zywina}, 
       pdftitle={Determining monodromy groups of abelian varieties},
]{hyperref}
\usepackage[alphabetic,backrefs,lite]{amsrefs} 

\usepackage[euler-digits,small]{eulervm}

\begin{document}
\title[Determining monodromy groups of abelian varieties]{Determining monodromy groups of abelian varieties}
\author{David Zywina}
\address{Department of Mathematics, Cornell University, Ithaca, NY 14853, USA}
\email{zywina@math.cornell.edu}
\urladdr{http://www.math.cornell.edu/~zywina}


\begin{abstract} 
Associated to an abelian variety over a number field are several interesting and related groups: the motivic Galois group, the Mumford--Tate group, $\ell$-adic monodromy groups, and the Sato--Tate group.       Assuming the Mumford--Tate conjecture, we show that from two well chosen Frobenius polynomials of our abelian variety, we can recover the identity component of these groups (or at least an inner form), up to isomorphism, along with their natural representations.    We also obtain a practical probabilistic algorithm to compute these groups by considering more and more Frobenius polynomials;  the groups are connected and reductive and thus can be expressed in terms of root datum.  These groups are conjecturally linked with algebraic cycles and in particular we obtain a probabilistic algorithm to compute the dimension of the Hodge classes of our abelian variety for any fixed degree.
\end{abstract}

\maketitle

\section{Introduction}  \label{S:introduction}

Throughout we fix a nonzero abelian variety $A$ of dimension $g$ defined over a number field $K$.   Fix an embedding $K\subseteq \CC$.  Let $\Kbar$ be the algebraic closure of $K$ in $\CC$ and define $\Gal_K:=\Gal(\Kbar/K)$.  Define the homology group $V_A:=H_1(A(\CC),\QQ)$, where $A(\CC)$ is viewed with the usual analytic topology.   \\

We now describe several of the algebraic groups we are interested in computing.   We denote by $G_A$ the \defi{motivic Galois group} of $A$ with respect to the category of motives in the sense of Andr\'e \cite{MR1423019}; the article \cite{CFC} gives a nice overview of the relevant groups and their connections.  We can view $G_A$ as an algebraic subgroup of $\GL_{V_A}$ since each $H^i(A(\CC),\QQ)$ is naturally isomorphic to the $i$-th exterior power of the dual of $V_A$; for a definition of $\GL_{V_A}$ see \S\ref{SS:notation}.  The \defi{Mumford--Tate group} $\MT_A \subseteq \GL_{V_A}$ of $A$, which is defined in \S\ref{SS:MT and Hodge groups} using the Hodge decomposition of $H_1(A(\CC),\CC)$, agrees with the identity component $G_A^\circ$ of $G_A$, cf.~\cite{CFC}*{p.~4}.   

Now take any rational prime $\ell$.  Let $V_\ell(A)$ be the $\ell$-adic Tate module; it is a $2g$-dimensional $\QQ_\ell$-vector space with an action of $\Gal_K$ which is discussed in \S\ref{SS:ell-adic monodromy group definition}.  We express this Galois action in terms of a continuous representation
\[
\rho_{A,\ell} \colon \Gal_K \to \GL_{V_\ell(A)}(\QQ_\ell).
\]
The \defi{$\ell$-adic monodromy group} of $A$ is the Zariski closure $G_{A,\ell}$ of the image of $\rho_{A,\ell}$ in $\GL_{V_\ell(A)}$; it is an algebraic group defined over $\QQ_\ell$.   Using the comparison isomorphism $V_\ell(A)=V_A\otimes_\QQ \QQ_\ell$, we can view $G_{A,\ell}$ as a subgroup of $(\GL_{V_A})_{\QQ_\ell}$.   Moreover, we always have an inclusion $G_{A,\ell} \subseteq (G_A)_{\QQ_\ell}$ from Artin's comparison theorem, cf.~\cite{CFC}*{p.~2}.    The \defi{Mumford--Tate conjecture} for $A$ says that in fact $G_{A,\ell}=(G_A)_{\QQ_\ell}$;  the original Mumford--Tate conjecture (Conjecture~\ref{C:MT}) only predicts that $G_{A,\ell}^\circ$ and $(G_A)_{\QQ_\ell}^\circ =(\MT_A)_{\QQ_\ell}$ agree, but these two formulations are equivalent by the main theorem of \cite{CFC}.

Take any nonzero prime ideal $\p$ of the ring of integers $\OO_K$ for which $A$ has good reduction.   If $\p\nmid \ell$, then $\rho_{A,\ell}$ is unramified at $\p$ and we have
\[
P_{A,\p}(x)= \det(xI-\rho_{A,\ell}(\Frob_\p)) \in \QQ_\ell[x],
\]
where $P_{A,\p}(x)$ is the \defi{Frobenius polynomial} of $A$ at $\p$.  The polynomial $P_{A,\p}(x)$ is monic of degree $2g$ with integer coefficients and does not depend on $\ell$.   In practice, these polynomials are computable when $A$ is given explicitly.   

There is a stronger conjectural version of this Frobenius compatibility.  Fix any prime $\ell$ satisfying $\p\nmid \ell$ and choose any embedding $\iota\colon \QQ_\ell \hookrightarrow \CC$.  Then by using $\iota$, we may identify $\rho_{A,\ell}(\Frob_\p) \in G_A(\QQ_\ell)$ with an element of $G_A(\CC)$.    It is conjectured that the conjugacy class of $\rho_{A,\ell}(\Frob_\p)$ in $G_A(\CC)$ does not depend on the initial choice of prime $\ell$ or embedding $\iota$.   We will make use of a weaker version of this conjecture, cf.~Conjecture \ref{C:Frob conj}.

Let $K_A^\conn\subseteq \Kbar$ be the minimal extension of $K$ for which the Zariski closure of $\rho_{A,\ell}(\Gal(\Kbar/K_A^\conn))$ is $G_{A,\ell}^\circ$.   The extension $K_A^\conn/K$ does not depend on $\ell$ and its degree can be bounded in terms of $g$, cf.~Proposition~\ref{P:connected}.    The reader is encouraged to focus on the case where $K_A^\conn=K$ since this assumption does not change the strength of the theorems and the group $G_A$ is then connected as well, cf.~\cite{CFC}*{p.~4}. 
\\

The group $G_A^\circ=\MT_A$ is  reductive.  So in particular, the base extension of $G_A^\circ$ to $\Qbar$ will be completely determined by its \defi{root datum}.   For some background on reductive groups and root data, see \S\ref{S:reductive groups}.  Fix a maximal torus $T$ of $G_A^\circ$ that is defined over $\QQ$.   The root datum $\Psi(G_A^\circ,T)$ of $G_A^\circ$ with respect to $T$ consists of the group $X(T)$ of characters $T_{\Qbar}\to (\GG_{m})_{\Qbar}$, a finite set of \defi{roots}, along with the group of cocharacters and coroots that satisfy certain conditions.   The action of $\Gal_\QQ$ on $X(T)$ induces an action on $\Psi(G_A^\circ,T)$.   Suppressing the choice of torus $T$, we obtain an abstract root datum $\Psi(G_A^\circ)$, that is uniquely determined up to an automorphism from its Weyl group, with a homomorphism
\[
\mu_{G_A^\circ} \colon \Gal_\QQ \to \Out(\Psi(G_A^\circ))
\]
arising from the Galois action.  The representation $V_A$ of $G_A^\circ$ also gives rise to a set of weights in $X(T)$ with multiplicities.\\
 
Our main result says that the root datum of $G_A^\circ$ and the homomorphism $\mu_{G_A^\circ}$ can be recovered from the Frobenius polynomials $P_{A,\p}(x)$ of two ``random'' primes $\p$ that split completely in $K_A^\conn$.

\begin{thm} \label{T:main}
Assume that Conjectures~\ref{C:MT} and \ref{C:Frob conj} hold for $A$.  Then for almost all prime ideals $\q$ and $\p$ of $\OO_K$ that split completely in $K_A^\conn$,   the polynomials $P_{A,\q}(x)$ and $P_{A,\p}(x)$ determine the following:
\begin{itemize}
\item the root datum $\Psi(G_A^\circ)$ up to isomorphism, 
\item the homomorphism $\mu_{G_A^\circ} \colon \Gal_\QQ \to \Out(\Psi(G_A^\circ))$,
\item the set of weights of the representation $G_A^\circ \subseteq \GL_{V_A}$ and their multiplicities.
\end{itemize}
By ``almost all'', we mean that the above holds with $\q \notin S$ and $\p\notin S'$, where $S$ and $S'$ are sets of prime ideals of $\OO_K$ of density $0$ with the second set $S'$ depending on $\q$.
\end{thm}

\begin{remark} \label{R:main remarks}
\begin{romanenum}
\item \label{R:main remarks ii}
The root datum $\Psi(G_A^\circ)$ and the homomorphism $\mu_{G_A^\circ}$ do not determine the group $G_A^\circ$, but they do  characterize $G_A^0$ up to an inner form.   In the proof, it will be convenient to work with the \emph{quasi-split inner form} of $G_A^\circ$.
\item \label{R:main remarks i}
Mark Kisin and Rong Zhou have recently announced a proof of Conjecture~ \ref{C:Frob conj} for primes not dividing $2$.  Once their work is available, it will make Theorem~\ref{T:main} and our other results conditional only on the Mumford--Tate conjecture for $A$ (Conjecture~\ref{C:MT}).
\end{romanenum}
\end{remark}

By considering the Frobenius polynomial $P_{A,\p}(x)$ for two ``random'' primes $\p$ that split completely in $K_A^\conn$, we can also recover the $\ell$-adic monodromy group $G_{A,\ell}^\circ$ and its representation $V_\ell(A)$, up to isomorphism, for all but finitely many $\ell$.
 
\begin{cor} \label{C:main}
Assume that Conjectures~\ref{C:MT} and \ref{C:Frob conj} hold for $A$.  Let $S$ be the (finite) set of primes $\ell$ for which $(G_A^\circ)_{\QQ_\ell}$ is not quasi-split.     Then for almost all prime ideals $\q$ and $\p$ of $\OO_K$ that split completely in $K_A^\conn$,   the polynomials $P_{A,\q}(x)$ and $P_{A,\p}(x)$ determine the group $G_{A,\ell}^\circ$ and the representation $V_\ell(A)$ of $G_{A,\ell}^\circ$ up to isomorphism for all $\ell\notin S$.
\end{cor}
\begin{proof}
From Theorem~\ref{T:main}, we may assume that we know the root datum $\Psi(G_A^\circ)$, the homomorphism $\mu_{G_A^\circ}$, and the weights of $G_A^\circ \subseteq \GL_{V_A}$ with multiplicities.   Take any prime $\ell \notin S$ and set $G:=(G_A^\circ)_{\QQ_\ell}=G_{A,\ell}^\circ$, where the equality uses the Mumford--Tate conjecture assumption.  

By choosing an embedding $\Qbar\hookrightarrow \Qbar_\ell$, we obtain the root datum $\Psi(G)$ and the associated homomorphism $\mu_G\colon \Gal(\Qbar_\ell/\QQ_\ell) \to \Out(\Psi(G))$.   Since $G$ is quasi-split, this information determines the algebraic group $G=G_{A,\ell}^\circ$ up to isomorphism, cf.~\S\ref{SS:quasi-split}.   Since we know the weights of the representation $(G_A^\circ)_{\QQ_\ell}=G_{A,\ell}^\circ$ on $V_A\otimes_{\QQ} \QQ_\ell = V_\ell(A)$ with multiplicities, we also have enough information to determine the representation $V_\ell(A)$ of $G_{A,\ell}^\circ$ up to isomorphism.  Finally, the set $S$ is finite, cf.~Theorem~6.7 of \cite{MR1278263}.  
\end{proof}

In \S\ref{SS:algorithms}, we will observe that given the polynomials $P_{A,\p}(x)$ and $P_{A,\q}(x)$ for appropriate prime ideals $\q$ and $\p$, the root datum with Galois action and weights as in Theorem~\ref{T:main} can be explicitly computed.  

By choosing prime ideals ``randomly'', this will give a probabilistic algorithm to compute the information from Theorem~\ref{T:main}; we can view it as a Monte Carlo algorithm where the probability that an incorrect result is returned decays exponentially in the number of primes considered.  

In order to prove Theorem~\ref{T:main}, it is easy to reduce to the case where $K_A^\conn=K$ (by replacing $A$ by its base extension by $K_A^\conn$).    We do not want to impose this condition since our algorithm does not require knowledge of the field $K_A^\conn$.

\subsection{Galois groups of Frobenius polynomials} \label{SS:intro Galois groups}
We now describe the Galois group of the Frobenius polynomial for almost all $\p$ that split completely in $K_A^\conn$; this will be a key ingredient in our proof of Theorem~\ref{T:main} and is of independent interest.  Let $\calW_{A,\p} \subseteq \Qbar$ be the set of roots of $P_{A,\p}(x)$.

Let us first define the groups that will generically occur as Galois groups.  Let $W(G_A^\circ)$ be the \defi{Weyl group} of $G_A^\circ$.   Let $\Gamma(G_A^\circ)\subseteq \Aut(\Psi(G_A^\circ))$ be the group containing $W(G_A^\circ)$ whose image in $\Out(\Psi(G_A^\circ)) = \Aut(\Psi(G_A^\circ))/W(G_A^\circ)$ agrees with the image of $\mu_{G_A^\circ}$.   Let $k$ be the fixed field in $\Qbar$ of the kernel of $\mu_{G_A^\circ}$.

\begin{thm} \label{T:Frobenius Galois groups new}
Assume that Conjectures~\ref{C:MT} and \ref{C:Frob conj} hold for $A$.   Fix a  number field $k \subseteq L \subseteq \Qbar$.   There is a set $S$ of prime ideals of $\OO_K$ with density $0$ such that for all nonzero prime ideals $\p\notin S$ of $\OO_K$ that split completely in $K_A^\conn$,
we have
\[
\Gal(L(\calW_{A,\p})/L) \cong W(G_A^\circ) \quad \text{ and } \quad \Gal(\QQ(\calW_{A,\p})/\QQ) \cong \Gamma(G_A^\circ).
\]
\end{thm}

\begin{remark}
A variant of Theorem~\ref{T:Frobenius Galois groups new} has been proved by the author  in the special case where $A$ is geometrically simple but without assuming Conjecture~\ref{C:Frob conj}, cf.~Theorem~1.5 of \cite{MR3264675}.   Conjecture~\ref{C:Frob conj} makes it much easier to identify $\Gal(\QQ(\calW_{A,\p})/\QQ)$ with a subgroup of $\Gamma(G_A^\circ)$.
\end{remark}

\subsection{Sato--Tate groups} \label{SS:ST}

The motivic Galois group $G_A$ comes with a character to $(\GG_m)_{\QQ}$ arising from the weight structure on the category of motives; its kernel we denote by $G_A^1$.    The neutral component of $G_A^1$ agrees with the \defi{Hodge group}  $\Hg_A$ of $A$ which is defined in \S\ref {SS:MT and Hodge groups}
 
We define the \defi{Sato--Tate group} of $A$ to be a maximal compact subgroup $\ST_A$ of $G_A^1(\CC)$ with respect to the analytic topology; it is a compact Lie group that is unique up to conjugation in $G_A^1(\CC)$.    The identity component $\ST_A^\circ$ of $\ST_A$ is a maximal compact subgroup of $\Hg_A(\CC)$.  

By considering the Frobenius polynomial of two ``random'' primes ideals that split completely in $K_A^\conn$, we can also recover $\ST_A^\circ$; it will be a straightforward consequence of Theorem~\ref{T:main}.  
 
\begin{thm} \label{T:ST}
Assume that Conjectures~\ref{C:MT} and \ref{C:Frob conj} hold for $A$.  Then for almost all prime ideals $\q$ and $\p$ of $\OO_K$ that split completely in $K_A^\conn$,   
 the polynomials $P_{A,\q}(x)$ and $P_{A,\p}(x)$ determine the root datum of $\Hg_A$ along with the weights and multiplicities of the action of $\Hg_A$ on $V_A$.  In particular, such polynomials  determine the Lie group $\ST_A^\circ$ and the representation $V_A\otimes_\QQ \CC$ of $\ST_A^\circ$ up to isomorphism.   
\end{thm}

For context, we now recall the Sato--Tate conjecture, cf.~\S13 of \cite{MR1265537}.  Fix a rational prime $\ell$ and an embedding $\iota\colon \QQ_\ell \hookrightarrow \CC$.  Take any nonzero prime ideal $\p\nmid \ell$ of $\OO_K$ for which $A$ has good reduction.   Using the embedding $\iota$, we may view $\rho_{A,\ell}(\Frob_\p)$ as a (semisimple) element of $G_A(\CC)$ and  $g_{\p,\ell}:=\rho_{A,\ell}(\Frob_\p)/\sqrt{N(\p)}$ as an element of $G_A^1(\CC)$.   From Weil, we know that the complex roots of $P_{A,\p}(x)$ all have absolute value $\sqrt{N(\p)}$ and hence the eigenvalues of $g_{\p,\ell}$ all have absolute value $1$.   So there is an element $\vartheta_\p \in \ST_A$, unique up to conjugacy in $\ST_A$, such that $g_{\p,\ell}$ and $\vartheta_\p$ are conjugate in $G_A^1(\CC)$.   

\begin{conj}[Sato--Tate]  \label{C:ST} 
The elements $\{\vartheta_\p\}_\p$ are equidistributed in the conjugacy classes of $\ST_A$ with respect to the Haar measure.   Equivalently, for any continuous central function $f\colon \ST_A \to \CC$, we have
\[
\lim_{x\to +\infty} \frac{1}{|\calP(x)|} \, \sum_{\p \in \calP(x)} f(\vartheta_\p) = \int_{\ST_A} f\, d\mu,
\]
where $\calP(x)$ is the set of good prime ideals $\p\subseteq \OO_K$ of norm at most $x$ and $\mu$ is the Haar measure of $\ST_A$ normalized so that $\mu(\ST_A)=1$.
\end{conj}

\begin{remark}
\begin{romanenum}
\item
As noted earlier, it is conjectured that the conjugacy class of $\rho_{A,\ell}(\Frob_\p)$ in $G_A(\CC)$ does not depend on the choice of $\ell$ and $\iota$.  This conjecture would imply that the conjugacy class of $\vartheta_\p$ in $\ST_A$ does not depend on the choice of $\ell$ and $\iota$. 
\item
Fix a prime $\ell$ and an embedding $\iota\colon \QQ_\ell \hookrightarrow \CC$.  A common alternate definition of the Sato--Tate group is a maximal compact subgroup of $G_{A,\ell}(\CC)$, where we have used our embedding $\iota$.   The Mumford--Tate conjecture would imply that this definition agrees with ours (and is independent of the choice of $\ell$ and $\iota$).   Moreover, one can show that Conjecture~\ref{C:ST} as stated implies the Mumford--Tate conjecture for $A$.
\end{romanenum}
\end{remark}

For integrals of the form $\int_{ST_A^\circ} f \, d\mu$ with a central continuous $f\colon \ST_{A}^\circ\to \CC$, Weyl's integration formula reduces it to an integral over a maximal torus of $\ST_A^\circ$.  In \S\ref {SS:Hodge dim}, we gives some simple examples of such integrals that are computable when one knows the root datum of $\Hg_A$ and the weights of its action on $V_A$ (which is information of $\Hg_A$ that can be computed using the algorithm of \S\ref{SS:algorithms}).

\begin{remark}
There has been a lot of recent progress computing Sato--Tate groups $\ST_A$.
In \cite{MR2982436}, Fit\'{e}, Kedlaya, Rotger and Sutherland have classified all the possible Sato--Tate groups that arise for abelian surfaces.   Fit\'{e}, Kedlaya and Sutherland have also proved a classification for abelian threefolds, cf.~ \cite{FKS2019}.  In particular, they find that there are $433$ possible Sato--Tate groups of abelian threefolds up to conjugacy in $\operatorname{USp}(6)$; of these $14$ are connected.  
While our algorithm will only give the identity component of $\ST_A$, it has the benefit that it does not require a classification and hence can be used in higher dimensions where a classification is infeasible.

 There has also been much working on computing Sato--Tate groups of various Jacobians with complex multiplication, for example see \cite{MR3218802}, \cite{MR3502940}, \cite{MR3864839}, \cite{EmoryGoodson};  these are particularly interesting because the group $\ST_A/\ST_A^\circ$ can be relatively large.
\end{remark}

\subsection{Dimension of Hodge classes} \label{SS:Hodge dim}

For an integer $n\geq 0$, the $n$-th  cohomology group $H^n(A(\CC),\QQ)$ has a Hodge structure of weight $n$ which gives a decomposition 
\[
H^n(A(\CC), \CC) = \bigoplus_{p+q=n} H^{p,q}(A).
\]	
The cohomology class $\cl(Z) \in H^{2p}(A(\CC),\CC)$ of an algebraic subvariety $Z$ of $A_\CC$ of codimension $p$ lies in the $\QQ$-vector space
\[
\mathcal{H}^p(A):=H^{2p}(A(\CC),\QQ)  \cap H^{p,p}(A)
\]
of \defi{Hodge classes}.  
The renowned \defi{Hodge conjecture} for $A$ predicts that for each integer $0\leq p\leq g$, the cohomology classes $\cl(Z)$ of algebraic subvarieties of  $A_\CC$ with codimension $p$ span the $\QQ$-vector space $\mathcal{H}^p(A)$.

We are particularly interested in computing the dimension of the vector spaces $\mathcal{H}^p(A)$ since the Hodge conjecture will sometimes predict the existence of  exceptional algebraic cycles. The following lemma shows that $\dim_\QQ \mathcal{H}^p(A)$ is determined by the action of $\Hg_A$ on $V_A$.   By using the techniques of \S\ref{SS:algorithms}, we can then make predictions for $\dim_\QQ \mathcal{H}^p(A)$ by considering just a few Frobenius polynomials.
\\
 
For integers $n\geq 0$ and $0\leq k \leq 2g$, the representation $V_A$ of $\GL_{V_A}$ gives rise to another representation $(\bigwedge^k V_A)^{\otimes n}$. We define $M_{k,n}$ to be the dimension of the subspace of $(\bigwedge^k V_A)^{\otimes n}$ fixed by $\Hg_A$.

\begin{lemma} \label{L:basic Hodge dimension}
For each integer $0\leq p \leq g$, we have  $\dim_\QQ \mathcal{H}^p(A) = M_{2p,1}$. 
\end{lemma}
\begin{proof}
There is a natural isomorphism of Hodge structures $V_0:=H^{2p}(A(\CC),\QQ) = \Hom({\bigwedge}^{2p} V_A,\QQ )$.  We have $\mathcal{H}^p(A) = V_0^{\Hg_A}$ by Theorem~17.3.3 of \cite{MR2062673}, so the dimension of $\mathcal{H}^p(A)$ is equal to $\dim_\QQ V_0^{\Hg_A} =\dim_\QQ ({\bigwedge}^{2p} V_A)^{\Hg_A}=M_{2p,1}$.
\end{proof}

For $B\in \GL_{V_A}(\CC)$, denote its characteristic polynomial by $\sum_{k=0}^{2g} (-1)^k a_k(B)\, x^k$.   The trace of the action of $B$ on $(\bigwedge^k V_A)^{\otimes n}$ is $a_k(B)^n$.  Since $\ST_A^\circ$ is a maximal compact subgroup of $\Hg_A(\CC)$, $M_{k,n}$ is also the dimension of the subspace of $(\bigwedge^k V_A)^{\otimes n} \otimes_\QQ \CC$ fixed by $\ST_A^\circ$.  Therefore,
\[
M_{k,n} = \int_{\ST_A^\circ} a_k^n \, d\mu,
\]
where $\mu$ is the Haar measure on $\ST_A^\circ$ satisfying $\mu(\ST_A^\circ)=1$.
\\

We now explain how to compute the values $M_{k,n}$ which by Lemma~\ref {L:basic Hodge dimension} allows us to compute the dimension of Hodge classes. We suppose that the root datum $\Psi(\Hg_A)$ of $\Hg_A$ is known; it comes with a character group $X$, a finite set of roots $R\subseteq X$, and there is a Weyl group $W$ that acts faithfully on $X$.  Further suppose we also know the weights $\alpha_1,\cdots,\alpha_{2g} \in X$ of the representation $V_A$ of $\Hg_A$ stated with multiplicity.   

We can identify $X$ with $X(T)$ for a maximal torus $T$ of $\Hg_A$.  Let $T_0$ be the maximal compact subgroup of $T(\CC)$.    Weyl's integration formula implies that
 \[
M_{k,n} = \frac{1}{|W|} \int_{T_0} a_k(t)^n \, \prod_{\alpha\in R}(1-\alpha(t))\, dt,
 \]
 where $dt$ is the Haar measure on $T_0$ normalized to have volume $1$.  
Fix a basis $\beta_1,\ldots, \beta_m$ of the free abelian group $X$.  For $\alpha\in X$, we have $\alpha=\prod_{i=1}^m \beta_i^{e_i(\alpha)}$ for unique $e_i(\alpha) \in \ZZ$.  So there is an isomorphism of Lie groups $T_0 \to U(1)^m$, $t\mapsto (\beta_1(t),\ldots, \beta_m(t))$, where $U(1)$ is the unit circle in $\CC^\times$.  We have $a_k(t)=\sum_{J \subseteq \{1,\ldots, 2g\},\, |J|=k} \, \prod_{j\in J} \alpha_j(t)$ and hence
\begin{align*}
M_{k,n} =\frac{1}{|W|} \int_{U(1)^{m}} \bigg(\sum_{J \subseteq \{1,\ldots, 2g\},\, |J|=k} \, \prod_{j\in J} \prod_{i=1}^{m} t_i^{e_i(\alpha_j)}\bigg)^n \cdot \prod_{\alpha\in R} \Big(1-\prod_{i=1}^{m} t_i^{e_i(\alpha)}\bigg) \, dt_1 \, \cdots \, dt_{m},
\end{align*}
where $dt_i$ is the Haar measure on $U(1)$ with volume $1$.  Such integrals are readily computable by expanding out and using that for an integer $b\in \ZZ$, the integral $\int_{U(1)} t_i^b \, dt_i$ is $1$ if $b=0$ and is $0$ otherwise.

\begin{remark}
By Proposition~1 of \cite{MR4038255}, $M_{1,2}$ is the rank of the group $\End(A_{\Kbar})$.   By Proposition~2 of \cite{MR4038255}, $M_{2,1}$ is the rank of the N\'eron--Severi group $\operatorname{NS}(A_{\Kbar})$.
\end{remark}

\subsection{Endomorphism ring} \label{SS:endomorphism ring}

In this section, we describe how to compute some partial information for the ring $\End(A_{\Kbar})\otimes_\ZZ \QQ$.    We have an isomorphism
\[
\End(A_{\Kbar})\otimes_\ZZ \QQ \cong B_1\times \cdots \times B_s,
\]
where $B_i$ is a central simple algebra over a number field $L_i$.   For each $1\leq i \leq s$, let $m_i$ be the positive integer satisfying $\dim_{L_i} B_i = m_i^2$.

\begin{thm} \label{T:endomorphism determination}
Assume that Conjectures~\ref{C:MT} and \ref{C:Frob conj} hold for $A$.  Then for almost all prime ideals $\q$ and $\p$ of $\OO_K$ that split completely in $K_A^\conn$,   
 the polynomials $P_{A,\q}(x)$ and $P_{A,\p}(x)$ determine the pairs $(L_1,m_1),\ldots, (L_s,m_s)$ up to reordering and replacing the $L_i$ by isomorphic number fields.
\end{thm}

We now explain how the pairs $(L_i,m_i)$ can be obtained from the root datum $\Psi(G_A^\circ)$ with its Galois action and the weights.      We may assume that our root datum is $\Psi(G_A^\circ,T)$ for some maximal torus $T$ of $G_A^\circ$.  Let $\Omega \subseteq X(T)$ be the weights of the representation $V_A$.   Let $\Gamma:=\Gamma(G_A^\circ, T)$ be the subgroup of $\Aut(\Psi(G_A^\circ,T))\subseteq \Aut(X(T))$ containing the Weyl group $W:= W(G_A^\circ, T)$ and whose image in $\Out(\Psi(G_A^\circ)) =\Aut(\Psi(G_A^\circ,T))/W$ agrees with the image of $\mu_{G_A^\circ}$.   Let $k$ be the fixed field in $\Qbar$ of the kernel of $\mu_{G_A^\circ}$; we have a natural isomorphism $\Gal(k/\QQ) = \Gamma/W$.

Let $\Omega_1,\ldots, \Omega_{s'}$ be the $\Gamma$-orbits of $\Omega$.   For each $1\leq i \leq s'$, choose a $W$-orbit $\OO_i \subseteq \Omega_i$.  Let $H_i$ be the group of $\sigma\in \Gamma$ for which $\sigma(\OO_i)=\OO_i$.   Using the isomorphism $\Gamma/W=\Gal(k/\QQ)$, let $L_i'$ be the subfield of $k$ fixed by $H_i/W$.   Let $m_i'$ be the multiplicity of any weight in $\OO_i$; this is well-defined since the $W$-action preserves multiplicities.

Theorem~\ref{T:endomorphism determination} is an immediate consequence of Theorem~\ref{T:main} and the following proposition which we will prove in \S\ref{S:endomorphism proof}.

\begin{prop} \label{P:endomorphism}
We have $s'=s$ and after reordering the $\Gamma$-orbits of $\Omega$, we have $m_i'=m_i$ and $L_i'\cong L_i$ for all $1\leq i \leq s$.
\end{prop}

\begin{remark}
\begin{romanenum}
\item
When $A$ is isogenous over $\Kbar$ to a power of a simple abelian variety, Theorem~\ref{T:endomorphism determination} is also a consequence of a theorem of Costa, Lombardo and Voight \cite{CLV} which uses a weaker version of Theorem~\ref{T:Frobenius Galois groups new}.

\item
In general, one cannot recover the structure of $\End(A_{\Kbar})\otimes_\ZZ \QQ$ from the Frobenius polynomials of two ``random'' prime ideals that split completely in $K_A^\conn$.    Since this ring is isomorphic to the subring of $\End(V_A)$ that commutes with the $\MT_A$-action, this also means that we cannot always recover $\MT_A=G_A^\circ$, up to isomorphism, with the setup of Theorem~\ref{T:main}.

For example, consider an abelian variety $A/K$ satisfying $K_A^\conn=K$ that is either the square of a non-CM elliptic curve or is a QM surface (i.e., its endomorphism ring is an indefinite quaternion algebra).  For all primes ideals $\p\subseteq \OO_K$ away from a set of density $0$, one can show that $P_{A,\p}(x)= Q_\p(x)^2$ for an irreducible quadratic $Q_\p(x)\in \ZZ[x]$.   Moreover, the splitting fields of $Q_\q(x)$ and $Q_\p(x)$ will be linearly disjoint over $\QQ$ for almost all $\q$ and $\p$.  So from a few Frobenius polynomials, we are unable to determine the structure of $\End(A_{\Kbar})\otimes_\ZZ \QQ$, but we can say that it is a central simple algebra over $\QQ$ and $\dim_\QQ (\End(A_{\Kbar})\otimes_\ZZ \QQ) = 2^2$.
\end{romanenum}
\end{remark}

\subsection{Overview of the proof of Theorem~\ref{T:main}} \label{SS:overview}
We now give an overview of the constructions underlying Theorem~\ref{T:main}.  Throughout, we assume that Conjectures~\ref{C:MT} and \ref{C:Frob conj} for $A$ hold.   

  Let $G$ be a quasi-split inner form of $G_{A}^\circ=\MT_A$, cf.~\S\ref{SS:quasi-split}.  In particular, there is an isomorphism $f\colon G_{\Qbar}\xrightarrow{\sim} (G_A^\circ)_{\Qbar}$ such that $f^{-1} \circ \sigma(f)$ is an inner automorphism of $G_{\Qbar}$ for all $\sigma\in \Gal_{\QQ}$.  Let $r$ be the rank of the reductive group $G_A^\circ$ and hence also the rank of $G$.   Since we are assuming the Mumford--Tate conjecture for $A$, the rank of each $G_{A,\ell}^\circ$ is also $r$.

For each nonzero prime ideal $\p\subseteq \OO_K$ for which $A$ has good reduction, let $\Phi_{A,\p}$ be the subgroup of $\Qbar^\times$ generated by the set of roots $\calW_{A,\p} \subseteq \Qbar$ of $P_{A,\p}(x)$.   There is an action of $\Gal_\QQ$ on the set $\calW_{A,\p}$, and hence also on the group $\Phi_{A,\p}$, since $P_{A,\p}(x)$ has rational coefficients.  If $\Phi_{A,\p}$ is torsion-free, then it is a free abelian group of rank at most $r$ and $\p$ splits completely in $K_A^\conn$, cf.~Lemma~\ref{L:common rank}(\ref{L:common rank i}).   By Lemma~\ref{L:common rank}(\ref{L:common rank ii}), there is a set $S$ of prime ideals of density $0$ such that if $\p\notin S$ splits completely in $K_A^\conn$, then $A$ has good reduction at $\p$ and $\Phi_{A,\p}$ is a free abelian group of rank $r$.\\

Consider a nonzero prime ideal $\p \subseteq\OO_K$ for which $A$ has good reduction and $\Phi_{A,\p}$ is a free abelian group of rank $r$.  Associated to $\p$ is a \defi{Frobenius torus} $T_\p$ of $G$.  More precisely, Theorem~\ref{T:Frobenius tori} implies that there is a maximal torus $T_\p$ of $G$ and an element $t_\p \in T_\p(\QQ)$ such that $f(t_\p) \in \GL_{V_A}(\Qbar)$ has characteristic polynomial $P_{A,\p}(x)$ and such that the homomorphism $X(T_\p) \to \Qbar^\times$,  $\alpha\mapsto \alpha(t_\p)$ induces an isomorphism 
\begin{align} \label{E: XT intro}
X(T_\p) \xrightarrow{\sim} \Phi_{A,\p}.
\end{align}
The isomorphism (\ref{E: XT intro}) respects the $\Gal_\QQ$-actions.    

Let $L\subseteq \Qbar$ be any number field for which $G_L$ is split.    The group $\Gal(L(\calW_{A,\p})/L)$ acts faithfully on $\Phi_{A,\p}$.  The Weyl group $W(G,T_\p)$ of $G$ with respect to $T_\p$ also acts faithfully on $X(T_\p)$.  Via the isomorphism (\ref {E: XT intro}), we can identify $\Gal(L(\calW_{A,\p})/L)$ with a subgroup of $W(G,T_\p)$, cf.~\S\ref{S:Galois groups of Frob}.    In particular, we have $[L(\calW_{A,\p}):L]\leq |W(G)|=|W(G_A^\circ)|$.\\

Now fix a nonzero prime ideal $\q\subseteq \OO_K$ for which $A$ has good reduction and $\Phi_{A,\q}$ is a free abelian group of rank $r$ (it suffices to take $\q\notin S$ that splits completely in $K_A^\conn$).   Define $L:=\QQ(\calW_{A,\q})$.  The group $G_L$ is split since we have a maximal torus $T_\q$ of $G$ that splits over $L$.   We have $L\supseteq k$ where $k$ is the field from \S\ref{SS:intro Galois groups}.  By Theorem~\ref{T:Frobenius Galois groups new}, there is a set $S'\supseteq S$ of prime ideals of density $0$ such that if $\p\notin S'$ splits completely in $K_A^\conn$, then $A$ has good reduction at $\p$ and
the following conditions hold:
\begin{itemize}
\item
$\Phi_{A,\p}$ is a free abelian group of rank $r$, 
\item
$[L(\calW_{A,\p}): L] = |W(G_A^\circ)|$. 
\end{itemize}
Take any $\p$ satisfying the above conditions.  Using the isomorphism (\ref {E: XT intro}) as an identification, we have an inclusion $\Gal(L(\calW_{A,\p})/L) \subseteq W(G,T_\p)$.  Therefore, $\Gal(L(\calW_{A,\p})/L) = W(G,T_\p)$ since the two groups both have cardinality $|W(G_A^\circ)|$.\\

With our choice of  $\q$ and $\p$ as above, we now explain that Theorem~\ref{T:main} holds with these primes.  We shall use the isomorphism (\ref{E: XT intro}) as an identification $X(T_\p)=\Phi_{A,\p}$.   Under this identification, the set of weights $\Omega \subseteq X(T_\p)$ of the representation 
\[
\rho\colon G_{\Qbar} \xrightarrow{f} (G_A^\circ)_{\Qbar} \subseteq \GL_{V_A \otimes_\QQ \Qbar}
\] 
with respect to $T_\p$ agrees with the set of roots $\calW_{A,\p}$ of $P_{A,\p}(x)$.  Moreover, the multiplicity of a weight $\alpha\in \Omega$ agrees with the multiplicity of $\alpha$ viewed as a root of $P_{A,\p}(x)$.  The Galois group $\Gal(L(\calW_{A,\p})/L)$ acts faithfully on the set $\calW_{A,\p}$ and hence also on the group $\Phi_{A,\p}$.    With respect to our identification $X(T_\p)=\Phi_{A,\p}$, we have $\Gal(L(\calW_{A,\p})/L) = W(G,T_\p)$.

So from the polynomials $P_{A,\p}(x)$ and $P_{A,\q}(x)$, we have found a group of characters $X(T_\p)$ with $T_\p$ a maximal torus of $G$, the $\Gal_\QQ$-action on $X(T_\p)$, the set of weights $\Omega\subseteq X(T_\p)$ of $\rho$ with multiplicities, and the group $W(G,T_\p)$ acting faithfully on $X(T_\p)$.  From this information, we explain in \S\ref{S:roots} how to compute the set of roots $R(G,T_\p) \subseteq X(T_\p)$ of $G$ with respect to $T_\p$.  In this root computation, we make use of the fact that the each irreducible representation of $(G_A^\circ)_{\Qbar}$ in $V_A\otimes_\QQ \Qbar$ is \emph{minuscule}, cf.~Proposition~\ref {P:minuscule MT}.    Finally, we will observe that the root datum $\Psi(G,T_\p)$ is directly determined by $X(T_\p)$, $R(G,T_\p)$ and the action of $W(G,T_\p)$ on $X(T_\p)$.    Since $G$ is an inner form of $G_A^\circ$, this gives us the desired quantities of Theorem~\ref{T:main}.

For the Hodge group $\Hg_A$, the character group can be taken to be $\Phi_{A,\p}/\ang{N(\p)}$ and the roots and weights are the image of those of $G_A^\circ$, see \S\ref{SS:root data for Hg} for details.

\subsection{Algorithms} \label{SS:algorithms}

We take it as given that for a nonzero prime ideal $\p$ for which $A$ has good reduction, the polynomial $P_{A,\p}(x)$ can be computed; one can also ignore a finite number of troublesome primes.  This should certainly be possible if $A$ is given explicitly.   The (reverse) of the polynomial $P_{A,\p}(x)$ will occur as a factor in a zeta function and there are many techniques and strategies for computing zeta functions.  Given a polynomial $P_{A,\p}(x)$, the structure of the group $\Phi_{A,\p}$ is computable, cf.~\S\ref{SS:computing PhiAp}. \\
  
We now assume that Conjectures~\ref{C:MT} and \ref{C:Frob conj} for $A$ hold. Let $r$ be the rank of $G_A^\circ$.    As noted in the overview in \S\ref{SS:overview},  there are nonzero prime ideals $\q$ and $\p$ of $\OO_K$ such that $A$ has good reduction and the following conditions hold:
\begin{alphenum}
\item \label{I:q condition intro}
$\Phi_{A,\q}$ is a free abelian group of rank $r$,
\item \label{I:p condition intro}
$\Phi_{A,\p}$ is a free abelian group of rank $r$ and $[L(\calW_{A,\p}): L] = |W(G_A^\circ)|$, where $L:=\QQ(\calW_{A,\q})$.
\end{alphenum}
Moreover, the set of $\q$ satisfying (\ref{I:q condition intro}) has positive density.  For a fixed $\q$, the set of $\p$ satisfying (\ref{I:p condition intro}) also has positive density (in fact, these densities can be bounded below by a positive constant that depends only on the dimension of $A$).

Assume that we know the Frobenius polynomials of two prime ideals $\q$ and $\p$ as above.   Let $\pi_1,\ldots, \pi_n$ be the distinct elements of $\calW_{A,\p}$ (working in some splitting field of $P_{A,\p}(x)$).   Let $\varphi\colon \ZZ^n \to \Phi_{A,\p}$ be the surjective homomorphism satisfying $\varphi(e_i)=\pi_i$, where $e_1,\ldots, e_n$ is the standard basis of $\ZZ^n$.  In \S\ref{SS:computing PhiAp}, we explain how to compute to compute $\ker \varphi$ and hence $X:=\ZZ^n/\ker \varphi$ gives the structure of $\Phi_{A,\p}$.  The Galois action on the roots $\pi_1,\ldots, \pi_n$ of $P_{A,\p}(x)$ gives an injective homomorphism $\Gal(\QQ(\calW_{A,\p})/\QQ) \hookrightarrow S_n$.  This gives an action of $\Gal(\QQ(\calW_{A,\p})/\QQ)$ on $\ZZ^n/\ker \varphi$ that agrees with the action on $\Phi_{A,\p}$ via the isomorphism induced by $\varphi$.  Looking over the overview in \S\ref{SS:overview} and using our isomorphism $X\xrightarrow{\sim} \Phi_{A,\p}$, one can compute the set of weights $\Omega\subseteq X$ with multiplicities and the Weyl group $W:=\Gal(L(\calW_{A,\p})/L)$ with its action on $X$ (with $L=\QQ(\calW_{A,\q})$).   As noted in the overview, the set of roots in $X$ of $G_A^\circ$ can also be found using the algorithm from \S\ref{S:roots}.  We can thus find the root datum of $G_A^\circ$ and the action of $\Gal(\QQ(\calW_{A,\p})/\QQ)$ on $X$ gives the homomorphism $\mu_{G_A^\circ}$.    As an aside, we note that we do not need to know the field $K_A^\conn$ for these computations.

So everything in Theorem~\ref{T:main} is computable, under Conjectures~\ref{C:MT} and \ref{C:Frob conj}, assuming we can find prime ideals $\q$ and $\p$ satisfying (\ref{I:q condition intro}) and (\ref{I:p condition intro}), respectively.  Moreover, in such a situation one should be able to make the computational \emph{unconditional}.  Indeed, the existence of a prime $\q$ satisfying (\ref{I:q condition intro}) implies the Mumford--Tate conjecture for $A$, cf.~Theorem~4.3 of \cite{MR1339927} and Lemma~\ref{L:common rank}.   Conjecture~\ref{C:Frob conj} is not required if one can choose $\q$ and $\p$ so that $A$ has ordinary reduction at these primes, cf.~\cite{MR1370746}; this is doable in practice.

\emph{If} we knew the integers $r$ and $w:=|W(G_A^\circ)|$, then we could just check primes $\q$ and $\p$ until we found a pair that satisfies (\ref{I:q condition intro}) and (\ref{I:p condition intro}), and we would have a deterministic algorithm to compute the information stated in Theorem~\ref{T:main}.     \\

However, $r$ and $w$ may not be known a priori.  By looking at several good prime ideals $\p$, say with norm up to some bound, we can make predictions for the value $r$ and $w$.    Let $r'$ be maximal rank of a torsion-free group $\Phi_{A,\p}$ that we encounter.   We have $r'\leq r$, cf.~Lemma~\ref{L:common rank}(\ref{L:common rank i}).    Fix a prime ideal $\q$ for which $\Phi_{A,\q}$ is free abelian of rank $r'$.    Define $L:=\QQ(\calW_{A,\q})$ and let $w'$ be the maximal value of the degree $[L(\calW_{A,\p}): L]$ that is encountered amongst all good prime ideals $\p$ for which $\Phi_{A,\p}$ is free of rank $r'$.  If $r'=r$, then $w'\leq w$, cf.~\S\ref{S:Galois groups of Frob}.

Given $r'$ and $w'$, we can apply the above algorithm to obtain a prediction for the information of Theorem~\ref{T:main} (or we might have incompatible input for the root computations which, under the Mumford--Tate conjecture, says that $r'\neq r$ or $w'\neq w$).     Assuming Conjectures~\ref{C:MT} and \ref{C:Frob conj}, we will obtain correct results by considering sufficiently many primes.\\

When trying to implement the above algorithm, besides possible finding the Frobenius polynomial, the biggest computational bottleneck is computing the splitting fields of the Frobenius polynomials $P_{A,\p}(x)$ over $\QQ$.  In \S\ref{S:computational remarks}, we explain how without computing such splitting fields we can still compute the root datum $\Psi(G_A^\circ)$ up to isomorphism, 
 the \emph{image} of the homomorphism $\mu_{G_A^\circ} \colon \Gal_\QQ \to \Out(\Psi(G_A^\circ))$, and 
the set of weights of the representation $G_A^\circ \subseteq \GL_{V_A}$ and their multiplicities.  

An implementation of this algorithm in 
\texttt{Magma} \cite{MR1484478} can be found at:
\begin{center}
\url{https://github.com/davidzywina/monodromy}
\end{center}

\begin{remark}
The sets of density $0$ in our proof of Theorem~\ref{T:main} arise from several applications of the Chebotarev density theorem.    Making these sets computable is tricky because  Chebotarev's theorem is applied to Galois extensions arising  from the representations $\rho_{A,\ell}$ whose image is difficult to give explicitly (if we knew its image easily, we would easily obtain $G_{A,\ell}^\circ$).

At least under the Generalized Riemann Hypothesis (GRH) for number fields, it is possible to say something about the $\q$ with smallest norm satisfying (\ref{I:q condition intro}).   Let $r_0$ be the common rank of the reductive groups $G_{A,\ell}$, cf.~Lemma~\ref{L:common rank}(\ref{L:common rank iii}); it equals $r$ assuming the Mumford--Tate conjecture.  In \S7 of \cite{ZywinaEffectiveOpen}, it is show assuming GRH that there is a prime ideal $\q$ of $\OO_K$ satisfying 
\[
N(\q) \leq C ( \max\{[K:\QQ], h(A), \log D \})^f
\]
for which $A$ has good reduction and $\Phi_{A,\q}$ is a free abelian group of rank $r_0$, where $C$ and $f$ are positive constants that depend only on $g$, $h(A)$ is the  (logarithmic absolute and semistable) Faltings height of $A$, and $D$ is the product of rational primes that ramify in $K$ or are divisible by a prime ideal of $\OO_K$ for which $A$ has bad reduction.  The constants $C$ and $f$ are computable but not worked out in \cite{ZywinaEffectiveOpen}; they are likely too large to be of practical use.
\end{remark}

\subsection{A CM example}

Let $A$ be the Jacobian of the smooth projective curve $C$ over $\QQ$ defined by the equation $y^2=x^9-1$; it is an abelian variety over $\QQ$ of dimension $4$. \\

Computing $\Phi_{A,p}$ for a few small primes, say $5\leq p\leq 100$, we observe that $\Phi_{A,p}$ has torsion when $p\not\equiv 1 \pmod{9}$ and is free of rank $4$ when $p\equiv 1 \pmod{9}$.  For those primes satisfying $p\equiv 1 \pmod{9}$, we also observe that $\QQ(\calW_{A,p})=\QQ(\zeta_9)$, where $\zeta_9$ is a primitive $9$-th root of unity in $\Qbar$. 

Assuming these properties for the groups $\Phi_{A,p}$ are typical, we now make a prediction for the group $G_A^\circ$ (while implicitly assuming that Conjectures~\ref{C:MT} and \ref{C:Frob conj} hold for $A$); we will prove these predictions afterwards.    Using that conditions (\ref{I:q condition intro}) and (\ref{I:p condition intro}) should hold for a positive density set of primes, we expect that the rank of $G_A^\circ$ is $4$ and $W(G_A^\circ)=1$.   In particular, we expect $G_A^\circ$ to be a torus of rank $4$.

Consider any good prime $p$ for which $\Phi_{A,p}$ is free of rank $4$.   We then obtain a Frobenius maximal torus $T_p$ of $G_A^\circ$ which equals $G_A^\circ$ (note that $G_A^\circ$ is quasi-split since it is a torus).   
There is an element $t_p\in T_p(\QQ)=G_A^\circ(\QQ)$ such that we have an isomorphism
\begin{align}\label{E:XGA intro}
X(G_A^\circ)=X(T_p) \xrightarrow{\sim} \Phi_{A,p},\quad \alpha\mapsto \alpha(t_p)
\end{align}
for which the weights $\Omega \subseteq X(G_A^\circ)$ of the representation $V_A$ of $G_A^\circ$ correspond with the roots $\calW_{A,\p}\subseteq \Phi_{A,p}$ of $P_{A,p}(x)$.   We expect there to be $a,b,c \in \Qbar$ such that the roots of $P_{A,p}(x)$ are:
\begin{align} \label{E:roots CM example}
a,\quad b, \quad c, \quad p/a, \quad p/b, \quad p/c,  \quad abc/p, \quad p^2/(abc)
\end{align}
and they are all distinct; this can be computed for any suitable prime (say $p=19$) and then will hold in general by our isomorphism (\ref{E:XGA intro}). We can also describe the action of $\Gamma:=\Gal(\QQ(\zeta_9)/\QQ)$ on the roots (\ref{E:roots CM example}).    We can choose $a$, $b$ and $c$ so that $\Gal(\QQ(\zeta_9)/\QQ(\sqrt{-3}))$ permutes these three roots transitively.    This then describes the action of $\Gamma$ on these roots since complex conjugation takes a root $\alpha$ to $p/\alpha$.   

We now have our description of the group $G_A^\circ$ and its representation up to isomorphism.   Since $G_A^\circ$ is a torus, it is determined up to isomorphism, by the group $X(G_A^\circ)$ and the $\Gal_\QQ$-action.    So we can take  $X:=X(G_A^\circ)$ to be the free (multiplicative) abelian group on symbols $a$, $b$, $c$ and $p$.  We let $\Gamma:=\Gal(\QQ(\zeta_9)/\QQ)$ act on the group $X$ by letting it fix $p$, letting $\Gal(\QQ(\zeta_9)/\QQ(\sqrt{-3}))$ act transitively on $\{a,b,c\}$, and letting complex conjugation map any $\alpha\in \{a,b,c\}$ to $p/\alpha$.   The representation $V_A$ of $G_A^\circ$ is then determined, up to isomorphism, by letting (\ref{E:roots CM example}) be the weights in $X$ (each with multiplicity $1$).

Following \S\ref{SS:Hodge dim}, one can then show that  $\dim_\QQ \mathcal{H}^2(A)=M_{4,1}=8$.   Indeed, Shioda has proved that $\dim_\QQ \mathcal{H}^2(A)=8$, cf.~Example 6.1 of \cite{MR641669} (moreover, $\mathcal{H}^2(A)$ is not generated by the intersection of pairs of divisors of $A$).\\

We now prove the above predictions.    First observe that the abelian variety $A$ has complex multiplication.  Moreover, we have an isomorphism of rings $\End(A_{\Qbar}) = \ZZ[\zeta_9]$, where $\zeta_9$ corresponds with the automorphism of $A_{\Qbar}$ arising from the automorphism $(x,y)\mapsto (\zeta_9x,y)$ of $C_{\Qbar}$.  The field of definition of the endomorphisms of $A_{\Qbar}$ is $\QQ(\zeta_9)$, so $\QQ(\zeta_9) \subseteq K_A^\conn$.  This explains the condition $p\equiv 1\pmod{9}$ above; we need this to hold for $p$ to split completely in $K_A^\conn$.  In fact, one can show that $K_A^\conn=\QQ(\zeta_9)$.    

The group $G_A^\circ$ is a torus; we can identify it with a subgroup of the torus $\Res_{L/\QQ}(\GG_m)$, where $L:=\End(A_{\Qbar})\otimes_\ZZ \QQ$, since the actions of $G_A^\circ$ and $L$ commute and $V_A$ is a $1$-dimension $L$-vector space.  In particular, $W(G_A^\circ)=1$.  The Mumford--Tate conjecture holds for all CM abelian varieties and in particular for $A$,  cf.~\cites{MR228500, MR3384679}.     Conjecture~\ref{C:Frob conj} also holds for all CM abelian varieties and in particular $A$; this follows for example from the idelic description of the Galois representations as given in \S7 of \cite{MR236190}.

Since $W(G_A^\circ)=1$ and the Conjectures~\ref{C:MT} and \ref{C:Frob conj} hold for $A$, to verify our predictions for $G_A^\circ$ it suffices to prove that $r=4$.  Suppose on the contrary that $r\neq 4$.  We have $r\geq 5$ since we have observed that there are primes $p$ for which $\Phi_{A,p}$ has rank $4$.  Since the the Mumford--Tate conjecture for $A$ holds, there is a good prime $p$ such that $\Phi_{A,p}$ is a free abelian group of rank $r\geq 5$.    There are $a,b,c,d\in \Qbar$ such the roots of $P_{A,p}(x)$ are:
\begin{align} \label{E:roots CM example 2}
a,\quad b, \quad c, \quad d, \quad p/a,\quad p/b, \quad p/c,\quad p/d.
\end{align}
We deduce that $r=5$ and that $\Phi_{A,p}$ has basis $\{a,b,c,d,p\}$.  Arguing as above, though now unconditionally,  we have an isomorphism (\ref{E:XGA intro}) and the weights are given by (\ref{E:roots CM example 2}).   Following \S\ref{SS:Hodge dim}, one can compute that  $\dim_\QQ \mathcal{H}^2(A)=M_{4,1}=6$.  However, this contradicts the result of Shioda mentioned above.  Therefore, $r=4$ as claimed.

\begin{remark}
As noted by Shioda in Example 6.1 of \cite{MR641669}, $A$ is isogenous to $B\times E$, where $B$ is a simple abelian variety of dimension $3$ and $E$ is an elliptic curve. 
So  $P_{A,p}(x)=P_{B,p}(x) \cdot P_{E,p}(x)$
for all good primes $p$.   When $\Phi_{A,p}$ is free of rank $4$, we can choose the roots (\ref{E:roots CM example}) of $P_{A,p}(x)$ so that $abc/p$ and $p^2/(abc)$ are the distinct roots of $P_{E,p}(x)$.
\end{remark}

\subsection{Another example}
Let $C$ be the smooth projective curve over $\QQ$ with affine model 
\begin{align} \label{E:affine model example 2}
y^2=x(x^{20}+7x^{18}-7x^2-1).   
\end{align}
Let $A$ be the Jacobian of $C$; it is an abelian variety over $\QQ$ of dimension $10$.   

We have computed the group $\Phi_{A,p}$ for all odd primes $p\leq 200$ for which the right hand side of (\ref{E:affine model example 2}) is separable modulo $p$.  For such primes $p$, the group $\Phi_{A,p}$ is torsion-free if and only if $p$ is an element of $\calP:=\{17,41,73,89,97,113,137,193\}$.   Moreover, $\Phi_{A,\p}$ has rank $6$ for all $p\in \calP$. 

A computation shows that the Galois group $\Gamma:=\Gal(\QQ(\calW_{A,p})/\QQ)$ has order $2^5\cdot 5!$ for all $p\in \calP$.   For any prime $p \in \calP$, we have $P_{A,p}(x)=Q_p(x)^2$ for a polynomial $Q_p(x)\in \QQ[x]$ whose discriminant is in $-2\cdot (\QQ^\times)^2$.  In particular, $\QQ(\sqrt{-2})\subseteq \QQ(\calW_{A,p})$.   Moreover, for any distinct $p,q\in \calP$, one can verify that $\QQ(\calW_{A,q}) \cap \QQ(\calW_{A,p}) = \QQ(\sqrt{-2})$.   Thus for any distinct $p,q \in \calP$, we have $[L(\calW_{A,p}):L] = 2^4\cdot 5!$ where $L=\QQ(\calW_{A,q})$.\\

By considering only the primes $p\leq 200$ and the conditions (\ref{I:q condition intro}) and (\ref{I:p condition intro}) of \S\ref{SS:algorithms}, we predict that the rank of $G_A^\circ$ is $6$ and $|W(G_A^\circ)|=2^4\cdot 5!$.    We now assume that $\rank G_A^\circ \leq 6$ and $|W(G_A^\circ)|\leq 2^4\cdot 5!$.    Under these assumptions, we shall compute the root datum of $G_A^\circ$.    Fix distinct primes $q$ and $p$ in $\calP$.  

Using Lemma~\ref{L:common rank}(\ref{L:common rank i}) and Proposition~\ref{P:MT inclusion}, we have inequalities: $6=\rank \Phi_{A,p} \leq \rank G_{A,\ell}^\circ \leq \rank G_A^\circ$. By our assumed bound on the rank of $G_A^\circ$, the groups $G_A^\circ$ and $G_{A,\ell}^\circ$ both have rank $6$.   Theorem~4.3 of \cite{MR1339927} implies that the Mumford--Tate conjecture for $A$ holds.  Conjecture~\ref{C:Frob conj} will holds for the primes $p$ and $q$ by Proposition~\ref{P:Noot}(\ref{P:Noot ii}) (by considering the middle coefficient of the Frobenius polynomials, we find that $A$ has ordinary reduction at $p$ and $q$).   So excluding our initial assumptions, the following computations will be unconditional.

Since $P_{A,p}(x)=Q_p(x)^2$ with $Q_p(x)\in \QQ[x]$ separable of degree $10$, there are $\alpha_1,\ldots,\alpha_5\in \Qbar$ such that the distinct roots of $P_{A,p}(x)$ are $\alpha_1,\ldots,\alpha_5,p/\alpha_1,\ldots,p/\alpha_5$.  Since $\Phi_{A,p}$ has rank $6$, it has basis $\alpha_1,\ldots,\alpha_5,p$.  Let 
\[
\varphi\colon \Phi_{A,p}\hookrightarrow \QQ^6
\]
 be the injective homomorphism satisfying  $\varphi(\alpha_i)=e_i+1/2\cdot e_6$ for $1\leq i \leq 5$ and $\varphi(p)=e_6$, where $e_1,\ldots, e_6$ is the standard basis of $\QQ^6$.  Define the set
\[
\Omega:= \varphi(\calW_{A,p}) = \big\{ \pm e_i + \tfrac{1}{2}\cdot e_6 : 1\leq i \leq 5 \big\}.
\]
Let $X\subseteq \QQ^6$ be the image of $\varphi$; equivalently, the subgroup of $\QQ^6$ generated by $\Omega$.  The map $\varphi$ defines an isomorphism $\Phi_{A,p} \xrightarrow{\sim} X$.

The Galois group $\Gamma:=\Gal(\QQ(\calW_{A,p})/\QQ)$ has order $2^5\cdot 5!$.  The group $\Gamma$ is as large as possible when we take into account that it induces an action on the set of pairs $\{\alpha_1,p/\alpha_1\},\ldots, \{\alpha_5,p/\alpha_5\}$.  Via the isomorphism $\varphi\colon \Phi_{A,p}\xrightarrow{\sim} X$ we obtain a faithful action of $\Gamma$ on $X$ and hence also on $X\otimes_\ZZ \QQ =\QQ^6$.  In terms of the action on $\QQ^6$, we can identify $\Gamma$ with the group of signed permutation matrices in $GL_6(\QQ)$ that fix $e_6$.   Define the field $L:=\QQ(\calW_{A,q})$ and the group $W:=\Gal(L(\calW_{A,p})/L)$.   The group acts $W$ acts faithfully on $\Phi_{A,p}$.    Since $\QQ(\calW_{A,q}) \cap \QQ(\calW_{A,p})=\QQ(\sqrt{-2})$, restriction to $\QQ(\calW_{A,p})$ allows to identify $W$ with the subgroup $\Gal(\QQ(\calW_{A,p})/\QQ(\sqrt{-2})) \subseteq \Gamma$.   Via its action on $\QQ^6$, we can also view $W$ as the subgroup of $\Gamma$ that acts on $\pm e_1,\ldots, \pm e_5$ via signed permutations that only change an even number of signs.

Let $G$ be the quasi-split inner form of $G_A^\circ$.    Using our assumptions and the overview of \S\ref{SS:overview}, we deduce that there is an maximal torus $T_\p$ of $G$ for which we have an isomorphism
\begin{align}\label{E:second ex}
X(T_p) = \Phi_{A,p} \xrightarrow{\stackrel{\varphi}{\sim}} X
\end{align}
of groups with compatible actions of $\Gamma=\Gal(\QQ(\calW_{A,p})/\QQ)$.  Moreover, the groups $W(G,T_p)$ that acts faithfully on $X(T_p)$ can be identified with $W$ via the isomorphism (\ref{E:second ex}).  The set $\Omega\subseteq X$ corresponds to the usual set of weights in $X(T_p)$; each weight has multiplicity $2$ since $Q_p(x)$ is separable.  Let $R$ be the subset of $X$ corresponding to the set $R(G,T_p)$ of roots under the isomorphism (\ref{E:second ex}). 

We now compute $R$.  The group $W$ acts transitively on $\Omega$.  Define the set
\[
C_\Omega:=\{\alpha-\beta: \alpha,\beta \in \Omega,\, \alpha\neq \beta\};
\]
it is the union of the two $W$-orbits $\{\pm e_i \pm e_j : 1 \leq i < j \leq 5\}$ and $\{\pm 2e_i: 1\leq i \leq 5\}$.  Applying Algorithm~\ref{algorithm 1} and Proposition~\ref{P:Lie 1}, we find that the root system $R$ is irreducible and is a subset of $C_\Omega$.  By Propositions~\ref{P:Lie 2} and \ref{P:Lie 3}, we find that $R$ is of Lie type $D_5$ and that 
\[
R=\{\pm e_i \pm e_j : 1 \leq i < j \leq 5\}.
\]
To finish computing the quantities of Theorem~\ref{T:main}, we observe that the set of roots $R\subseteq X$ and the action of the Weyl group $W$ on $X$ determines the desired root datum, cf.~Lemma~\ref{L:root datum from W}.  The homomorphism $\mu_{G_A^\circ}$ arises from the natural isomorphism $\Gal(\QQ(\sqrt{-2})/\QQ)=\Gamma/W$.\\

Finally, Lemma~\ref{L:example 2 unconditional} makes the above computations unconditional.   For the proof, we need to know the structure of the ring $\End(A_{\Qbar})\otimes_\ZZ \QQ$; this is not something that can be determined by considering only Frobenius polynomials (though following \S\ref{SS:endomorphism ring}, we can predict that it is a semisimple central algebra over $\QQ$ of degree $2$).

\begin{lemma} \label{L:example 2 unconditional}
The group $ G_A^\circ$ has rank $6$ and $W(G_A^\circ)$ has cardinality $2^4\cdot 5!$.
\end{lemma}
\begin{proof}
Let $\alpha$ and $\beta$ be the automorphisms of $C_{\Qbar}$ given by $(x,y)\mapsto (-1/x, y/x^{11})$ and $(x,y)\mapsto (-x,i\cdot y)$, respectively.   Let $\bbar{\alpha}$ and $\bbar{\beta}$ be the automorphisms of the Jacobian $A_{\Qbar}$ induced by $\alpha$ and $\beta$, respectively.   We have $\alpha^2=\iota$, $\beta^2=\iota$ and $\alpha\circ \beta = \iota \circ \beta\circ \alpha$, where $\iota$ is the hyperelliptic involution.     Since $\iota$ induces multiplication by $-1$ on  $A_{\Qbar}$, we find that the $\QQ$-subalgebra $D$ of $\End(A_{\Qbar})\otimes_\ZZ \QQ$ generated by $\bbar\alpha$ and $\bbar\beta$ is a quaternion algebra over $\QQ$ for which $D\otimes_\QQ \RR$ is isomorphic to the real quaternions.

We claim that $D=\End(A_{\Qbar})\otimes_\ZZ \QQ$.   Assuming the claim, we now prove the lemma.  Let $\Hg_A$ be the Hodge group of $A$, cf.~\S\ref {SS:MT and Hodge groups}.  By Corollary~5.26 of \cite{MR2663452}, we have $(Hg_A)_\CC \cong \SO(10)$; moreover, the group $\Hg_A$ is as large as possible when you take into account a polarization and the endomorphisms of $A$.  Therefore, $\Hg_A$ has rank $5$ and its Weyl group is isomorphic to the group $W(D_5)$ which has order $2^4 5!$.   The lemma is now immediate since $G_A^\circ = \GG_m \cdot \Hg_A \supsetneq \Hg_A$.\\

Since $D\subseteq \End(A_{\Qbar})\otimes_\ZZ \QQ$, to prove the above claim it suffices to show that $\End(A_{\Qbar})\otimes_\ZZ \QQ$ has dimension $4$ over $\QQ$.     Fix a number field $K\subseteq \Qbar$ so that $A_K$ is isogenous to $\prod_{i=1}^s A_i$, where the $A_i/K$ are powers of simple abelian varieties that are pairwise nonisogenous.  By increasing $K$, we may further assume that $\End(A_K)=\End(A_{\Qbar})$, that $K_A^\conn=K$, and that $K_{A_i}^\conn=K$ for all $1\leq i \leq s$.   We have 
\[
\End(A_{\Qbar})\otimes_\ZZ \QQ \cong B_1 \times \cdots \times B_s, 
\]
where $B_i:=\End(A_i)\otimes_\ZZ \QQ$.  Each $B_i$ is a central simple algebra over its center $L_i$.   Let $m_i$ be the degree of $B_i$, i.e., the positive integer for which $m_i^2=\dim_{L_i} B_i$.  The ring $\End(A_{\Qbar})\otimes_\ZZ \QQ$ is noncommutative since it contains $D$.  So we may assume that $m_1\geq 2$ after possibly renumbering the $A_i$.

Take any prime $p\in \calP$.  Let $\p \subseteq \OO_K$ be a prime ideal dividing $p$ and let $d$ be the degree of $\FF_\p$ over $\FF_p$.    We have a surjective map 
\[
\calW_{A,p} \to \calW_{A_K,\p},\quad \pi \mapsto \pi^d;
\]
it is a bijection since $\Phi_{A,p}$ is torsion-free.   This bijection respects the $\Gal_\QQ$-actions and $P_{A,p}(x)$ is the square of an irreducible polynomial, so $P_{A_K,\p}(x)=Q(x)^2$ for an irreducible $Q(x)\in \QQ[x]$.  Since $P_{A_K,\p}(x)=\prod_{i=1}^s P_{A_i,\p}(x)$, the polynomial $P_{A_1,\p}(x)$ is either $Q(x)$ or $Q(x)^2$.    However, Lemma~6.1 of \cite{MR3264675} implies that $P_{A_i,\p}(x)$ is an $m_i$-th power (moreover, it shows that all the weights of $G_{A_i}^\circ$ acting on $V_{A_i}$ have multiplicity $m_i$).   Since $m_1\geq 2$, we deduce that $P_{A_i,\p}(x)=Q(x)^2=P_{A_K,\p}(x)$.  Therefore, $s=1$ and $m_1=2$.   So $\End(A_{\Qbar}) \otimes_\ZZ \QQ\cong B_1$ and $\dim_{L_1} B_1=2^2$.   To complete the lemma, it thus suffices to prove that $L_1=\QQ$.

Take any prime $q\in \calP-\{p\}$.   By Corollary~7.4.4 of \cite{MR3904148},  $L_1$ is isomorphic to a subfield of $\QQ(\pi)$ for any $\pi\in \calW_{A,p}$.     Similarly, $L_1$ is isomorphic to a subfield of $\QQ(\pi)$ for any $\pi\in \calW_{A,q}$.   Since $\QQ(\calW_{A,p})\cap \QQ(\calW_{A,q})=\QQ(\sqrt{-2})$, the field $L_1$ is isomorphic to $\QQ$ or $\QQ(\sqrt{-2})$.  From our earlier description of the Galois groups $W=\Gal(\QQ(\calW_{A,p})/\QQ(\sqrt{-2}))$ and $\Gamma= \Gal(\QQ(\calW_{A,p})/\QQ)$, we find that $\sqrt{-2} \notin \QQ(\pi)$ for all $\pi \in \calW_{A,p}$.  Therefore, $L_1=\QQ$.
\end{proof}

\begin{remark}
Our set $\calP$ is precisely the set of primes $p\leq 200$ that are congruent to $1$ modulo $8$.  This suggests that $K_A^\conn$ is $\QQ(\zeta_8)$.
\end{remark}

\subsection{Structure of paper}

In \S\ref {S:reductive groups}, we recall what we need concerning reductive groups and root data.   In \S\ref{S:l-adic monodromy groups}, we review $\ell$-adic monodromy groups.   In \S\ref{SS:the group Phi}, we describe some properties of the groups $\Phi_{A,\p}$ and in \S\ref{SS:computing PhiAp} we explain how they can be computed.  In \S\ref{S:MT groups}, we recall the Mumford--Tate and Hodge groups of $A$ and state the conjectures our theorems are conditional on.    

In \S\ref{S:Frobenius tori}, we explain how to certain prime ideals $\p$, we can construct a maximal torus $T_\p$ of a quasi-split inner form of $G_A^\circ$.    Theorem~\ref{T:Frobenius Galois groups new} is proved in \S\ref{S:Galois groups of Frob}; much of the proof is dedicated to reducing to a setting considered by the author in an earlier paper.   In \S\ref{S:roots}, we explain how to compute roots of certain reductive groups over a field of characteristic $0$ given a maximal torus, weights of a faithful representation, and the Weyl group.

In \S\ref{S:root datum of MT}, we prove Theorems~\ref{T:main} and \ref{T:ST}.  We make some additional computational remarks in  \S\ref{S:computational remarks}.  Finally, Proposition~\ref{P:endomorphism} is proved in \S\ref{S:endomorphism proof}.

\subsection{Notation}   \label{SS:notation}

For a number field $K$, we denote by $\OO_K$ the ring of integers of $K$.   For a non-zero prime ideal $\p$ of $\OO_K$, we define its residue field $\FF_\p:=\OO_K/\p$.   Throughout, $\ell$ will always denote a rational prime.  When talking about prime ideals of a number field $K$, density will always refer to \emph{natural density}.  

For a scheme $X$ over a commutative ring $R$ and a commutative $R$-algebra $S$, we denote by $X_S$ the base extension of $X$ by $\Spec S$.  

  Let $V$ be a free module of finite rank over a field $F$.  Denote by $\GL_V$ the $F$-scheme such that $\GL_V(R)= \Aut_R(V\otimes_F R)$ for any commutative $F$-algebra $R$ with the obvious functoriality.   For an algebraic group $G$ over a field $F$, we denote by $G^\circ$ the neutral component of $G$, i.e., the connected component of the identity of $G$.  Note that $G^\circ$ is an algebraic subgroup of $G$.


\section{Reductive groups}  \label{S:reductive groups}

In this section, we recall what we need about reductive groups. For background on reductive groups see \cite{MR546587}.  Let $k$ be a field of characteristic $0$ and define $\Gal_k:=\Gal(\kbar/k)$, where $\kbar$ is a fixed algebraic closure of $k$.

Let $G$ be a connected and reductive group over $k$.  Fix a maximal torus $T$ of $G$.  

\subsection{Character groups}

   We define $X(T)$ to be the group of homomorphisms $T_{\kbar} \to (\GG_{m})_{\kbar}$; it is a free abelian group whose rank is the dimension of $T$.   There is a natural action of $\Gal_k$ on $X(T)$ since $T$ is defined over $k$. The group $X(T)$ with its $\Gal_k$-action determine the torus $T$ up to isomorphism.

\subsection{Weyl group} \label{SS:Weyl definition}

Let $N_G(T)$ be the normalizer of $T$ in $G$.   The torus $T$ is its own centralizer in $G$.  The (absolute) \defi{Weyl group} of $G$ with respect to $T$ is the group
 \[
 W(G,T) := N_{G}(T)(\kbar)/T(\kbar);
 \]
 it is finite.  Since $G$ and $T$ are defined over $k$, there is a natural action of $\Gal_k$ on $W(G,T)$.  

The Weyl group $W(G,T)$ acts on the abelian group $X(T)$; for $w\in W(G,T)$ and $\alpha\in X(T)$, we have
\[
(w\cdot \alpha)(t)= \alpha(n^{-1} t n),
\]
where $n\in N_G(T)(\kbar)$ is a representative of $w$.  The action of $W(G,T)$ on $X(T)$ is faithful and hence we can identify if with a subgroup of $\Aut(X(T))$.

Take any other maximal torus $T'$ of $G$.  There is an element $g\in G(\kbar)$ satisfying $g T_{\kbar} g^{-1} = T'_{\kbar}$.   The homomorphism $N_G(T)_{\kbar} \to N_G(T')_{\kbar}$, $n\mapsto gng^{-1}$ is an isomorphism which induces an isomorphism $W(G,T) \xrightarrow{\sim} W(G,T')$.    A different choice of $g$ will alter this isomorphism by composing with an inner automorphism.   So by using these isomorphisms, we can define a group $W(G)$ for which there is an isomorphism to each $W(G,T)$ that is distinguished up to composition with an inner automorphism.

\subsection{Weights}

Let $\rho\colon G\to \GL_V$ be a representation, where $V$ is a finite dimensional vector space over $k$.   For each character $\alpha\in X(T)$, let $V_\alpha$ be the set of $v\in V\otimes_k \kbar$ satisfying $\rho(t) v =\alpha(t) v$ for all $t\in T(\kbar)$.  Each $V_\alpha$ is a $\kbar$-vector space and we have a direct sum 
\[
V\otimes_k \kbar = \bigoplus_{\alpha \in X(T)} V_\alpha.
\]    
A \defi{weight} of $\rho$ relative to $T$ is an $\alpha\in X(T)$ for which $V_\alpha\neq 0$.   We will denote the set of weights of $\rho$  by $\Omega_\rho$ or $\Omega_V$.  Since $G$, $T$ and $\rho$ are defined over $k$, we find that $\Omega_V$ is stable under the $\Gal_k$-action on $X(T)$.   The set of weights $\Omega_V$ is also stable under the action of the Weyl group $W(G,T)$.  The \defi{multiplicity} of a weight $\alpha$ is the dimension of $V_\alpha$ over $\kbar$.

A \defi{root} of $G$ relative to $T$ is a nontrivial weight of the adjoint representation of $G$.   Denote the set of such roots by $R(G,T)$; it is stable under the actions of $\Gal_k$ and $W(G,T)$.

\subsection{Root datum}

\begin{defn}
A \defi{root datum} is a $4$-tuple $(X,R,X^\vee,R^\vee)$ where 
\begin{itemize}
\item
$X$ and $X^\vee$ are free abelian groups of finite rank with a perfect pairing $\ang{\,,\,}\colon X \times X^\vee \to \ZZ$,
\item
$R$ and $R^\vee$ are finite subsets of $X$ and $X^\vee$, respectively, with a bijection $R\to R^\vee$, $\alpha\mapsto \alpha^\vee$,
\end{itemize}
such that the following conditions hold for all $\alpha\in R$:
\begin{alphenum}
\item
$\ang{\alpha,\alpha^\vee} = 2$,
\item
$s_\alpha(R)= R$, where $s_\alpha$ is the automorphism of $X$ defined by $s_\alpha(x)= x - \ang{x,\alpha^\vee}\, \alpha$,
\item
$s_\alpha^\vee(R^\vee)= R^\vee$, where $s_\alpha^\vee$ is the automorphism of $X^\vee$ defined by $s_\alpha^\vee(y)= y - \ang{\alpha,y}\, \alpha^\vee$.
\end{alphenum}
\end{defn}
We say that a root datum $(X,R,X^\vee,R^\vee)$ is \defi{reduced} if for each $\alpha\in R$, the only elements of $R$ that are linearly dependent with $\alpha$ in $X\otimes_\ZZ \QQ$ are $\pm \alpha$.

An \defi{isomorphism} between root data $(X,R,X^\vee,R^\vee)$ and $(X',R',X'^\vee,R'^\vee)$ is an isomorphism $\varphi\colon X\xrightarrow{\sim} X'$ of groups that gives a bijection from $R$ to $R'$ and whose dual $X'^\vee\to X^\vee$, with respect to the perfect pairings, gives a bijection from $R'^\vee$ to $R^\vee$. 

Fix a root datum $\Psi=(X,R,X^\vee,R^\vee)$.  Let $\Aut(\Psi)$ be the group of automorphisms of $\Psi$; it is a subgroup of $\Aut(X)$.   The \defi{Weyl group} of $\Psi$ is the subgroup $W(\Psi)$ of $\Aut(X)$ generated by the set of $s_\alpha$ with $\alpha\in R$.   The Weyl group $W(\Psi)$ is a finite and normal subgroup of $\Aut(\Psi)$.  The \defi{outer automorphism group} of $\Psi$ is $\Out(\Psi):=\Aut(\Psi)/W(\Psi)$.  

\begin{lemma} \label{L:root datum from W}
The root datum $\Psi=(X,R,X^\vee,R^\vee)$ is determined, up to isomorphism, by the following:
\begin{alphenum}
\item  \label{L:root datum from W a}
the group $X$ and the set of roots $R$,
\item \label{L:root datum from W b}
the group $W(\Psi)$ and its action on $X$.
\end{alphenum}
\end{lemma}
\begin{proof}
We may assume that $X$, $R$, and $W(\Psi)$ are known, where $W(\Psi)$ is given by its faithful action on $X$.    We may further assume that $X^\vee=\Hom_\ZZ(X,\ZZ)$ with its perfect pairing $X\times X^\vee \to \ZZ$, $(\alpha,\beta)\mapsto \beta(\alpha)$.  So to prove the lemma, we need only explain how to find the coroot $\alpha^\vee \in X^\vee$ for each $\alpha\in R$.

Take any $\alpha \in R$.  Note that $s_\alpha(\alpha)=-\alpha$ and that $\{x\in X : s_\alpha(x)=x\}$ is a free $\ZZ$-module of rank $r-1$, where $r$ is the rank of $X$.  Now take any $s\in W(\Psi)$ for which $s(\alpha)=-\alpha$ and for which $\{x\in X : s(x)=x\}$ is a free $\ZZ$-module of rank $r-1$.  

We claim that $s=s_\alpha$.   The group $W(\Psi)$ acts on the $\RR$-vector space $V:=X\otimes_\ZZ \RR$.  Since $W(\Psi)$ is finite, there is an inner product $(\,,\,)$ on $V$ satisfying $(w\cdot v_1,w\cdot v_2)=(v_1,v_2)$ for all $v_1,v_2\in V$ and $w\in W(\Psi)$.  By the conditions on $s$, we have $V=V_{-1} \oplus V_1$ and $V_{-1}=\RR\alpha$, where $V_\lambda$ is the eigenspace of $s$ with eigenvalue $\lambda$.    Since the inner product is invariant under $s$, we find that $V_{-1}$ and $V_1$ are orthogonal.   Since $V_{-1}=\RR\alpha$ and $V_1$ is the orthogonal complement of $V_{-1}$ in $V$, we deduce that the action of $s\in W(\Psi)$ on $V$ depends only on $\alpha$; this proves the claim.

From the claim, we can thus determine $s_\alpha \in W(\Psi)$ for any $\alpha\in R$.   Since $\ang{x,\alpha^\vee} \alpha = x-s_\alpha(x)$ for all $x\in X$ and the pairing is perfect, we can then find $\alpha^\vee$.
\end{proof}

\subsection{Root datum of \texorpdfstring{$G$}{G}} \label{SS:root datum of G}

To our reductive group $G$ and maximal torus $T$, we have a reduced root datum
\[
\Psi(G,T):= \big(X(T), R(G,T), X^\vee(T), R^\vee(G,T) \big).
\]

We have already defined $X(T)$ and $R(G,T)$.  We define $X^\vee(T)$ to be the group of cocharacters $(\GG_{m})_{\kbar}\to T_{\kbar}$.   Composition defines a perfect pairing $X(T) \times X^\vee(T) \to \End((\GG_{m})_{\kbar}) =\ZZ$,  where the last isomorphism uses that endomorphisms of $(\GG_{m})_{\kbar}$ are all obtained by  raising to an integer power.  The Weyl group $W(\Psi(G,T))$ acts faithfully on $X(T)$ and we can identify it with $W(G,T)$.  We will not define the set of coroots $R^\vee(G,T)$ since it is not needed in our application (we will construct our root datum in terms of Lemma~\ref{L:root datum from W}). 

\subsection{Outer automorphism groups} \label{SS:outer automorphisms}

The root datum $\Psi(G,T)$ depends only on $G$ and $T$ base changed to $\kbar$.
To ease notation, we will assume that $k$ is algebraically closed throughout \S\ref{SS:outer automorphisms}.   

Let $G'$ be another connected reductive group defined over $k$ with a fixed maximal torus $T'$.   If there is an isomorphism $f\colon G \to G'$ satisfying $f(T)=T'$,   then $X(T)\to X(T')$, $\alpha\mapsto \alpha\circ f^{-1}|_{T'}$ is an isomorphism of groups that gives an isomorphism 
\[
f_* \colon \Psi(G,T) \xrightarrow{\sim} \Psi(G',T')
\] 
of root data.  The following fundamental result says that every isomorphism between $\Psi(G,T)$ and $\Psi(G',T')$ arises from such an $f$.  For each $g\in G(k)$, let $\inn(g)$ be the automorphism of $G$ obtained by conjugation by $g$.

\begin{thm} \label{T:main isomorphism}
If $\varphi \colon \Psi(G,T) \xrightarrow{\sim} \Psi(G',T')$ is an isomorphism of root data, then there is an isomorphism $f\colon G \xrightarrow{\sim} G'$ satisfying $f(T)=T'$ and $\varphi=f_*$.   The isomorphism $f$ is unique up to composing with $\inn(t)$ for some $t\in T(k)$.
\end{thm}

\subsubsection{Same group and torus}  \label{SSS:same group and torus}

Consider the special case where $G=G'$ and $T=T'$.   Let $\Aut(G, T)$ be the subgroup of $\Aut(G)$ consisting of all automorphisms $f$ that satisfy $f(T)=T$.   The map 
\begin{align} \label{E:Aut GT}
\Aut(G,T)\to \Aut(\Psi(G,T)), \quad f\mapsto f_*
\end{align} 
is a group homomorphism.   By  Theorem~\ref{T:main isomorphism}, the homomorphism (\ref{E:Aut GT}) is surjective with kernel $\inn(T(k))$.     Let $\Inn(G,T)$ be the group of inner automorphisms $f$ of $G$ that satisfy $f(T)=T$.  We have $\Inn(G,T)=\inn(N_G(T)(k))$, so the image of $\Inn(G,T)$ under (\ref{E:Aut GT}) is $W(G,T)$.   In particular, (\ref{E:Aut GT}) induces an isomorphism
\[
\Aut(G,T)/\Inn(G,T) \xrightarrow{\sim}  \Out(\Psi(G,T)).\\
\]

\begin{remark}
The natural homomorphism $\Aut(G,T)/\Inn(G,T)\to \Out(G):=\Aut(G)/\Inn(G)$ is an isomorphism since all maximal tori of $G$ are conjugate, where $\Inn(G)$ is the group of inner automorphism of $G$.
\end{remark}

\subsubsection{Same group} \label{SSS:same group}  
Suppose that $G=G'$ but the tori $T$ and $T'$ need not agree.   Since $k$ is algebraically closed, there is an inner automorphism $f$ of $G$ satisfying $f(T)=T'$ which induces an isomorphism $f_*\colon \Psi(G,T) \xrightarrow{\sim} \Psi(G,T')$.     Our choice of $f$ is not unique and can vary by composing with some $f' \in \Inn(G,T)$.    We have $(f\circ f')_* = f_* \circ f'_*$ and from \S\ref{SSS:same group and torus} we know that $f'_*$ lies in $W(G,T)$.    So our isomorphism $\Psi(G,T) \xrightarrow{\sim} \Psi(G,T')$ is unique up to composition with an element of the Weyl group.   Using these isomorphisms, we will often suppress the torus and denote the root datum by $\Psi(G)$ (note that it is only uniquely determined up to an automorphism in $W(G,T)$).

With our fixed $f$ as above, we have an isomorphism 
\begin{align} \label{E:AutPsi}
\Aut(\Psi(G,T))\to \Aut(\Psi(G,T')), \quad\varphi\mapsto f_*\circ \varphi \circ f_*^{-1}.
\end{align}  
A different choice of $f$, would alter this isomorphism by conjugation by $W(G,T')$.    We thus have a canonical isomorphism $\Out(\Psi(G,T))\xrightarrow{\sim} \Out(\Psi(G,T'))$.   Using these isomorphisms, we can suppress the torus and denote the group simply by $\Out(\Psi(G))$.

\subsubsection{General case} \label{SSS:general case}   Consider again the general case where $G$ and $G'$ need not be equal.  For an isomorphism $g\colon \Psi(G,T)\xrightarrow{\sim} \Psi(G',T')$, the group isomorphism $\Aut(\Psi(G,T))\to \Aut(\Psi(G',T'))$, $\varphi\mapsto g\circ \varphi \circ g^{-1}$ induces an isomorphism
\begin{align} \label{E:Out isomorphism}
[g]\colon \Out(\Psi(G))\xrightarrow{\sim} \Out(\Psi(G')).
\end{align}
A different choice of $g$ will alter (\ref{E:Out isomorphism}) 
 by composition with an inner automorphism of the group $\Out(\Psi(G'))$.

\subsection{Galois action} \label{SS:Galois action}
Since $T$ is defined over $k$, there is a natural action of $\Gal_k$ on $X(T)$ and $X^\vee(T)$ that respects the perfect pairing $X(T)\times X^\vee(T) \to \ZZ$.  Since $G$ and $T$ are defined over $k$, we find that the sets $R(G,T)$ and $R^\vee(G,T)$ are stable under these  Galois actions.  The Galois action on $X(T)$ thus gives rise to a homomorphism
\[
\varphi_{G,T} \colon \Gal_k \to \Aut(\Psi(G,T)).
\]
By composing $\varphi_{G,T}$ with the quotient map $\Aut(\Psi(G,T))\to\Out(\Psi(G))$, we obtain a homomorphism 
\[
\mu_G \colon \Gal_k \to \Out(\Psi(G))
\] 
that does not depend on the initial choice of $T$.\\

Let $\Gamma(G,T)$ be the subgroup of $\Aut(\Psi(G,T))$ that is the inverse image of $\mu_G(\Gal_k)$ under the homomorphism $\Aut(\Psi(G,T))\to\Out(\Psi(G))$.    Note that 
\[
\varphi_{G,T}(\Gal_k) \subseteq \Gamma(G,T).
\]     
From the isomorphism (\ref{E:AutPsi}), we can suppress the torus and denote the common group by $\Gamma(G)$; note that it is only uniquely determined up to conjugation by $W(G)$.

\subsection{Inner forms}
Let $G'$ be a connected reductive group defined over $k$ with a fixed maximal torus $T'$.   We say that $G'$ is a \defi{form} of $G$ if they become isomorphic algebraic groups when base extended to $\kbar$.    Note that $G'$ is a form of $G$ if and only if the root data $\Psi(G',T')$ and $\Psi(G,T)$ are isomorphic.    We say that $G'$ is an \defi{inner form} of $G$
there is an isomorphism $f\colon G_{\kbar}\to G'_{\kbar}$ such that $f^{-1}\circ \sigma(f)$ is an inner automorphism of $G_{\kbar}$ for all $\sigma\in \Gal_k$.  

We shall say that the homomorphisms $\mu_G$ and $\mu_{G'}$ from \S\ref{SS:Galois action} 	\defi{agree} if there is an isomorphism $g\colon \Psi(G,T)\xrightarrow{\sim}\Psi(G',T')$ of root data such that $[g]\circ \mu_G = \mu_{G'}$ with $[g]$ as defined in \S\ref{SSS:general case}.

\begin{prop} \label{P:characterize inner forms}
Let $G'$ be a connected and reductive group defined over $k$ that is a form of $G$.   Then $G'$ is an inner form of $G$ if and only if $\mu_G$ and $\mu_{G'}$ agree.  In particular, the class of inner forms of $G$ is determined by the abstract root datum $\Psi(G)$ and the homomorphism $\mu_G\colon \Gal_k \to \Out(\Psi(G))$.
\end{prop}
\begin{proof}
First take any isomorphism $f\colon G_{\kbar}\to G'_{\kbar}$ satisfying $f(T_{\kbar})=T'_{\kbar}$.  Take any $\sigma\in \Gal_k$.  Since our reductive groups and tori are all defined over $k$, we have an isomorphism $\sigma(f)\colon G_{\kbar} \xrightarrow{\sim} G'_{\kbar}$ that satisfies $\sigma(f)(T_{\kbar})=T'_{\kbar}$.   For any character $\alpha\in X(T)$, we have
\[
(\sigma(\alpha \circ f^{-1}))\circ f \, |_{T_{\kbar}}= \sigma(\alpha) \circ (\sigma(f)^{-1}\circ f) \, |_{T_{\kbar}}= \sigma(\alpha) \circ (f^{-1} \circ \sigma(f))^{-1} \, |_{T_{\kbar}}.
\]
Therefore, 
\begin{align} \label{E:is it inner}
f_*^{-1} \circ \varphi_{G',T'}(\sigma) \circ f_* =  \xi_\sigma \circ  \varphi_{G,T}(\sigma)
\end{align}
where $\xi_\sigma:=(f^{-1}\circ \sigma(f) )_* \in \Aut(\Psi(G,T))$.  

First suppose that $G'$ is an inner form of $G$ and hence $f$ may be chosen so that $f^{-1} \circ \sigma(f)$ is an inner automorphism of $G_{\kbar}$ for all $\sigma\in \Gal_k$.   We have $f^{-1} \circ \sigma(f) \in \Inn(G_{\kbar},T_{\kbar})$ and hence $\xi_\sigma \in W(G,T)$, cf.~\S\ref{SSS:same group and torus}.    From (\ref{E:is it inner}) and $\xi_\sigma \in W(G,T)$, we deduce that $\mu_{G'}(\sigma) = [f_*] \circ  \mu_{G}(\sigma)$.  Therefore, $\mu_G$ and $\mu_{G'}$ agree.  

Now suppose that $\mu_G$ and $\mu_{G'}$ agree.  There is an isomorphism $g\colon \Psi(G,T)\xrightarrow{\sim}\Psi(G',T')$ of root data such that $[g]\circ \mu_G = \mu_{G'}$.   By Theorem~\ref{T:main isomorphism}, there is an isomorphism $f\colon G_{\kbar} \xrightarrow{\sim} G'_{\kbar}$ satisfying $f(T_{\kbar})=T'_{\kbar}$ and $g=f_*$.  Take any $\sigma\in \Gal_k$.   From $[f_*]\circ \mu_G(\sigma) = \mu_{G'}(\sigma)$, we deduce that $f_*\circ \varphi_{G,T}(\sigma) \circ f_*^{-1} \circ \varphi_{G',T'}(\sigma)^{-1}$ lies in $W(G',T')$ and hence 
\[
f_*^{-1}\circ (f_*\circ \varphi_{G,T}(\sigma) \circ f_*^{-1} \circ \varphi_{G',T'}(\sigma)^{-1})\circ f_* = \varphi_{G,T}(\sigma) \circ (f_*^{-1} \circ \varphi_{G',T'}(\sigma) \circ f_*)^{-1} 
\]
lies in $W(G,T)$.  By (\ref{E:is it inner}), we find that $\xi_\sigma := (f^{-1}\circ \sigma(f))_*$ is an element of $W(G,T)$.  From \S\ref{SSS:same group and torus},  we deduce that  $f^{-1}\circ \sigma(f)\in \Inn(G_{\kbar},T_{\kbar})$ for all $\sigma\in \Gal_k$.   Therefore, $G'$ is an inner form of $G$.

Finally, the last statement of the propostion is clear since the root datum of $G$  determines the forms of $G$.
\end{proof}

\subsection{Quasi-split groups} \label{SS:quasi-split}

Recall that a connected reductive group $G$ defined over $k$ is \defi{quasi-split} if it has a Borel subgroup defined over $k$, i.e., there is an algebraic subgroup $B$ of  $G$ for which $B_{\kbar}$ is a Borel subgroup of $G_{\kbar}$.  

A connected reductive group $G$ over $k$ has an inner form $G_0$ that is quasi-split which is  unique up to isomorphism, cf.~\cite{MR3729270}*{Corollary~23.53}.   We will refer to the group $G_0$ as \emph{the} quasi-split inner form of $G$.   Proposition~\ref{P:characterize inner forms} implies that $G_0$, up to isomorphism, is determined by the abstract root datum $\Psi(G)$ and the homomorphism $\mu_G\colon \Gal_k \to \Out(\Psi(G))$.  

The quotient homomorphism $N_G(T)(\kbar)\to W(G,T)$ and the inclusion $N_G(T)(\kbar)\hookrightarrow G(\kbar)$ induces maps $H^1(k,N_G(T))\to H^1(k,W(G,T))$ and $H^1(k,N_G(T))\to H^1(k,G)$, respectively, of Galois cohomology sets.

\begin{prop} \label{P:Ragh}
Suppose that $G$ is quasi-split and that  the maximal torus $T$ contains a maximal split torus of $G$. Then for any $\xi\in H^1(k,W(G,T))$, there is a lift $\bbar\xi\in H^1(k,N_G(T))$ of $\xi$ which maps to the trivial class in $H^1(k,G)$.
\end{prop}
\begin{proof}
Let $G'$ be the derived subgroup of $G$ and define $T':=G'\cap T$.  The group $G'$ is semisimple and $T'$ is a maximal torus of $G'$ containing a maximal split torus of $G'$.   The inclusion $N_{G'}(T') \subseteq N_G(T)$ induces an isomorphism $W(G',T')=W(G,T)$ that respects the $\Gal_k$-actions.   It thus suffices to prove the proposition with $(G,T)$ replaced by $(G',T')$.  The semisimple case of the proposition is Theorem~1.1 of \cite{MR2125504}.  
\end{proof}

\subsection{Semisimple conjugacy classes} \label{SS:ss conjugacy classes}

Let $R$ be the affine coordinate ring of the connected reductive group $G$ over $k$.  The group $G$ acts on $R$ by composition with inner automorphisms.  We define $R^{G}$ to be the $k$-subalgebra of $R$ consisting of those elements fixed by this $G$-action, i.e., the algebra of \defi{central functions} on $G$.   Define the $k$-variety
\[
G^\sharp:=\Spec(R^{G})
\]
and denote by $\cl_{G} \colon G \to G^\sharp$ the morphism induced by the inclusion $R^{G} \hookrightarrow R$ of $k$-algebras.  The morphism $\cl_G$ is surjective.

Take any algebraically closed field $L/k$.   For any $g\in G(L)$, we have $g=g_s g_u$ for unique commuting $g_s,g_u\in G(L)$ with $g_s$ semisimple and $g_u$ unipotent in $G$.     For $g,h\in G(L)$, one can show that $g_s$ and $h_s$ are conjugate in $G(L)$ if and only if $\cl_{G}(g)=\cl_{G}(h)$.   In particular, $G^\sharp(L)$ can be identified with the set of semisimple elements of $G(L)$ up to conjugacy.

\subsection{Minuscule representations} \label{SS:minuscule}

Let $\rho\colon G\to \GL_U$ be an irreducible representation, where $U$ is a finite dimensional $k$-vector space.     The Weyl group $W(G,T)$ acts on the set of weights $\Omega_U \subseteq X(T)$ of the representation $U$ relative to $T$.

We say that the representation $\rho$ is \defi{minuscule} if $W(G,T)$ acts transitively on $\Omega_U$.      The property of $\rho$ being minuscule does not depend on the choice of $T$.	(Note that some authors exclude the trivial irreducible representation in the definition of minuscule.)

\section{\texorpdfstring{$\ell$}{l}-adic monodromy groups} \label{S:l-adic monodromy groups}

Fix an abelian variety $A$ of dimension $g\geq 1$ defined over a number field $K$.   

\subsection{\texorpdfstring{$\ell$}{l}-adic monodromy groups}
\label{SS:ell-adic monodromy group definition}
Take any rational prime $\ell$.  For an integer $e\geq 1$, let $A[\ell^e]$ be the $\ell^e$-torsion subgroup of $A(\Kbar)$; it is a free $\ZZ/\ell^e\ZZ$-module of rank $2g$.    The \defi{$\ell$-adic Tate module} is $T_\ell(A):= \varprojlim_e A[\ell^e]$, where the inverse limit is with respect to multiplication by $\ell$ maps $A[\ell^{e+1}]\to A[\ell^e]$; it is a free $\ZZ_\ell$-module of rank $2g$.    There is a natural action of $\Gal_K$ on the groups $A[\ell^e]$ and hence on the $\ZZ_\ell$-module $T_\ell(A)$.   Define the $\QQ_\ell$-vector space $V_\ell(A):=T_\ell(A) \otimes_{\ZZ_\ell}\QQ_\ell$; it has dimension $2g$.  We have a Galois action on $V_\ell(A)$ that we can express in terms of a continuous representation
\[
\rho_{A,\ell} \colon \Gal_K \to \Aut_{\QQ_\ell}(V_\ell(A))=\GL_{V_\ell(A)}(\QQ_\ell).
\]

\begin{defn}
The \defi{$\ell$-adic monodromy group} of $A$ is the algebraic subgroup $G_{A,\ell}$ of $\GL_{V_\ell(A)}$ obtained by taking the Zariski closure of $\rho_{A,\ell}(\Gal_K)$.  
\end{defn}

From the work of Faltings, cf.~\cite{MR861971}, we know that the group $G_{A,\ell}^\circ$ is reductive.  Define $K_A^\conn$ to be the subfield of $\Kbar$ fixed by the kernel of the homomorphism 
\begin{align*}
\Gal_K \xrightarrow{\rho_{A,\ell}} G_{A,\ell}(\QQ_\ell)\to G_{A,\ell}(\QQ_\ell)/G_{A,\ell}^\circ(\QQ_\ell).  
\end{align*}
Equivalently, $K_{A}^\conn$ is the smallest extension of $K$ in $\Kbar$ that satisfies $\rho_{A,\ell}(\Gal_{K_A^\conn}) \subseteq G_{A,\ell}^\circ(\QQ_\ell)$.  

\begin{prop} \label{P:connected}
\begin{romanenum}
\item \label{P:connected i}
The field $K_A^\conn$ depends only on $A$.  In particular, it is independent of $\ell$.   
\item \label{P:connected ii}
The degree $[K_A^\conn:K]$ can be bounded in terms of $g$.
\end{romanenum}
\end{prop}
\begin{proof}
Part (\ref{P:connected i}) was proved by Serre \cite{MR1730973}*{133}; see also \cite{MR1441234}.   From (\ref{P:connected i}), we find that $K_A^\conn$ is a subfield of $K(A[\ell^\infty])$ and hence $[K_A^\conn:K]$ divides $[K(A[\ell]):K]\ell^{e_\ell}$ for some integer $e_\ell$.  Since $[K(A[\ell]):K]$ divides $|\GL_{2g}(\FF_\ell)|$, we deduce that $[K_A^\conn :K]$ divides $|\GL_{2g}(\FF_\ell)| \ell^{e_\ell}$.   Therefore, $[K_A^\conn:K]$ must divide $|\GL_{2g}(\FF_2)|\cdot |\GL_{2g}(\FF_3)|$ which completes the proof of (\ref{P:connected ii}).
\end{proof}

\subsection{Compatibility}

Take any nonzero prime ideal $\p$ of $\OO_K$ for which $A$ has good reduction.  Denote by $A_\p$ the abelian variety over $\FF_\p$ obtained by reducing $A$ modulo $\p$.      There is a unique polynomial $P_{A,\p}(x) \in \ZZ[x]$ such that $P_{A,\p}(n)$ is the degree of the isogeny $n-\pi$ for each integer $n$, where $\pi$ is the Frobenius endomorphism of $A_\p/\FF_\p$.   The polynomial $P_{A,\p}(x)$ is monic of degree $2g$.  

For each rational prime $\ell$ for which $\p\nmid \ell$, the representation $\rho_{A,\ell}$ is unramified at $\p$ and satisfies 
\[
\det(xI - \rho_{A,\ell}(\Frob_\p)) = P_{A,\p}(x).
\]
Note that $\rho_{A,\ell}(\Frob_\p)$ is semisimple in $\GL_{V_\ell(A)}$ since	 $\pi$ acts semisimply on the $\ell$-adic Tate module of $A_\p$.    From Weil, we know that all of the roots of $P_{A,\p}(x)$ in $\CC$ have absolute value $N(\p)^{1/2}$. 

\subsection{The group \texorpdfstring{$\Phi_{A,\p}$}{Phi{A,p}}} 
\label{SS:the group Phi}

Take any nonzero prime ideal $\p\subseteq \OO_K$ for which $A$ has good reduction.   Let $\calW_{A,\p}$ be the set of roots of $P_{A,\p}(x)$ in $\Qbar$.   We define $\Phi_{A,\p}$ to be the subgroup of $\Qbar^\times$ generated by $\calW_{A,\p}$.    The abelian group $\Phi_{A,\p}$ is finitely generated and has a natural $\Gal_\QQ$-action since the coefficients of $P_{A,\p}(x)$ are rational.   

Let $r$ be the rank of the reductive group $G_{A,\ell}^\circ$ for a fixed prime $\ell$.  By the following lemma, the integer $r$ does not depend on the choice of $\ell$.

\begin{lemma} \label{L:common rank}
\begin{romanenum}
\item   \label{L:common rank i}
If $\Phi_{A,\p}$ is torsion-free, then $\p$ splits completely in $K_A^\conn$ and $\Phi_{A,\p}$ has rank at most $r$.
\item   \label{L:common rank ii}
There is a density $0$ set $S$ of prime ideals of $\OO_K$ such that if $\p\notin S$ splits completely in $K_A^\conn$, then $\Phi_{A,\p}$ is a free abelian group of rank $r$. 
\item \label{L:common rank iii}
The integer $r$ does not depend on the choice of $\ell$.  
\end{romanenum}
\end{lemma}
\begin{proof}
Choose an embedding $\Qbar\hookrightarrow \Qbar_\ell$.  Take any nonzero prime ideal $\p\nmid \ell$ of $\OO_K$ for which $A$ has good reduction.    

Let $T_{\p,\ell}$ be the Zariski closure in $G_{A,\ell}$ of the subgroup generated by the semisimple element $t_{\p,\ell}:=\rho_{A,\ell}(\Frob_\p)$.   Let $X(T_{\p,\ell})$ be the group of characters $(T_{\p,\ell})_{\Qbar_\ell}\to(\GG_m)_{\Qbar_\ell}$.  Since $T_{\p,\ell}$ is generated by $t_{\p,\ell}$, the homomorphism $f\colon X(T_{\p,\ell}) \to \Qbar_\ell^\times$, $\alpha\mapsto \alpha(t_{\p,\ell})$ is injective.  Let $\Omega_{\p,\ell} \subseteq X(T_{\p,\ell})$ be the set of weights of $T_{\p,\ell} \subseteq \GL_{V_\ell(A)}$ acting on $V_\ell(A)$; the set $\Omega_{\p,\ell}$ generates $X(T_{\p,\ell})$ since this action is faithful.    The elements $\{\alpha(t_{\p,\ell}):\alpha\in \Omega_{\p,\ell}\}$ are the roots of $P_{A,\p}(x)$ in $\Qbar_\ell$ and generate the image of $f$.   Therefore, we have an isomorphism $f\colon X(T_{\p,\ell}) \to \Phi_{A,\p}$, where we are using our fixed embedding $\Qbar\hookrightarrow \Qbar_\ell$.   In particular, the groups $X(T_{\p,\ell})$ and $\Phi_{A,\p}$ are isomorphic.   

Suppose that $\Phi_{A,\p} \cong X(T_{\p,\ell})$ is torsion-free and hence $T_{\p,\ell}$ is a torus.   We thus have $T_{\p,\ell}\subseteq G_{A,\ell}^\circ$ and hence $\rho_{A,\ell}(\Frob_\p) \in G_{A,\ell}^\circ(\QQ_\ell)$.  Therefore, $\p$ splits completely in $K_A^\conn$; note that $\rho_{A,\ell}(\Frob_\p) \in G_{A,\ell}^\circ(\QQ_\ell)$ if and only if $\p$ splits completely in $K_A^\conn$.   Since $T_{\p,\ell}\subseteq G_{A,\ell}^\circ$ is a torus and $G_{A,\ell}^\circ$ has rank $r$, we deduce that $T_{\p,\ell}$ has dimension at most $r$ and hence $\Phi_{A,\p} \cong X(T_{\p,\ell})$ has rank at most $r$.  This proves (\ref{L:common rank i}) for all good primes $\p\nmid \ell$.   Part (\ref{L:common rank i}) for any excluded good primes $\p | \ell$ will follow once we prove  (\ref{L:common rank iii}) since we may choose a different initial prime $\ell$.

By Theorem~1.2 of \cite{MR1441234}, there is a closed proper subvariety $Y\subseteq G_{A,\ell}$, that is stable under conjugation by $G_{A,\ell}$, such that if $\p\nmid \ell$ is a prime ideal of $\OO_K$ for which $A$ has good reduction and $\rho_{A,\ell}(\Frob_\p) \in G_{A,\ell}^\circ(\QQ_\ell) - Y(\QQ_\ell)$, then $T_{\p,\ell}$ is a maximal torus of $G_{A,\ell}^\circ$.    The Chebotarev density theorem then implies that there is a set $S$ of prime ideals of $\OO_K$ with density 0 such that if $\p\notin S$ splits completely in $K_A^\conn$, then $T_{\p,\ell}$ is a maximal torus of $G_{A,\ell}^\circ$.   Since $X(T_{\p,\ell}) \cong \Phi_{A,\p}$, there is a set $S$ of prime ideals of $\OO_K$ with density 0 such that if $\p\notin S$ splits completely in $K_A^\conn$, then $\Phi_{A,\p}$  is a free abelian group of rank $r$.   This proves (\ref{L:common rank ii}).    This also gives a characterization of the integer $r$ that does not depend on $\ell$ and hence proves (\ref{L:common rank iii}). 
\end{proof}

\begin{defn} \label{D:SA}
We define $\calS_A$ to be the set of nonzero prime ideals $\p\subseteq \OO_K$ for which $A$ has good reduction and for which $\Phi_{A,\p}$ is a free abelian group of rank $r$, where $r$ is the common rank of the reductive groups $G_{A,\ell}^\circ$. 
\end{defn}

\begin{lemma} \label{L:density of SA}
\begin{romanenum}
\item \label{L:density of SA i}
The set $\calS_A$ has density $1/[K_A^\conn:K]$.  
\item \label{L:density of SA ii}
Take any prime ideal $\p\in \calS_A$ and let $\P$ be a prime ideal of $\OO_{K_A^\conn}$ that divides $\p$.   Then $\P \in \calS_{A'}$ and $P_{A',\P}(x)=P_{A,\p}(x)$ where $A'$ is the base change of $A$ to $K_A^\conn$.
\end{romanenum}
\end{lemma}
\begin{proof}
The set of prime ideals of $\OO_K$ that split completely in $K_A^\conn$ has density $1/[K_A^\conn:K]$ by the Chebotarev density theorem.  Part (\ref{L:density of SA i}) thus follows  from Lemma~\ref{L:common rank}(\ref{L:common rank i}) and (\ref{L:common rank ii}).

We now prove (\ref{L:density of SA ii}).  The prime ideal $\p \in \calS_A$ splits completely in $K_A^\conn$ by Lemma~\ref{L:common rank}(\ref{L:common rank i}).   We thus have $\FF_\P=\FF_\p$ and hence the reduction of $A$ and $A'$ at the primes $\p$ and $\P$, respectively, agree.   Therefore, $P_{A,\p}(x)=P_{A',\P}(x)$ and hence $\Phi_{A',\P}= \Phi_{A,\p}\cong \ZZ^r$.   The reductive group $G_{A',\ell}=G_{A,\ell}^\circ$ has rank $r$ and hence $\P\in \calS_{A'}$.
\end{proof}

\subsection{Computing \texorpdfstring{$\Phi_{A,\p}$}{Phi{A,p}}} \label{SS:computing PhiAp}
In this section, we describe how to compute the structure of the finitely generated abelian group $\Phi_{A,\p}$ from $P_{A,\p}(x)$; this is important to ensure that our results are algorithmic (also see Remark~\ref{R:revisit later}).

Let $\pi_1,\ldots, \pi_n$ be the distinct elements of $\calW_{A,\p}$.    We have a surjective homomorphism 
\[
\varphi\colon \ZZ^n \to \Phi_{A,\p},\quad e\mapsto \prod_{i=1}^n \pi_i^{e_i}.
\]

We first describe the $e \in \ZZ^n$ for which $\varphi(e)$ is a root of unity.  Define the number field $L:=\QQ(\pi_1,\ldots, \pi_n)$.  For each nonzero prime $\lambda$ of $\OO_L$, let $v_\lambda\colon L^\times\twoheadrightarrow\ZZ$ be the $\lambda$-adic valuation.

\begin{lemma} \label{L:root of unity linear}
Take any $e\in \ZZ^n$.  The following are equivalent:
\begin{alphenum}
\item \label{L:root of unity linear a}
$\varphi(e)$ is a root of unity,
\item \label{L:root of unity linear b}
$\sum_{i=1}^n e_i v_\lambda(\pi_i) = 0$ 
holds for all prime ideals $\lambda$ of $\OO_L$ dividing $N(\p)$,
\item \label{L:root of unity linear c}
$\sum_{i=1}^n e_i v_\lambda(\sigma^{-1}(\pi_i)) = 0$ 
holds for a fixed prime ideal $\lambda$ of $\OO_L$ dividing $N(\p)$ and all $\sigma\in \Gal(L/\QQ)$.
\end{alphenum}
\end{lemma}
\begin{proof}
For any prime ideal $\lambda|N(\p)$ of $\OO_L$ and $\sigma\in \Gal(L/\QQ)$, one can check that $v_{\sigma(\lambda)}(\alpha) = v_\lambda(\sigma^{-1}(\alpha))$ holds for all $\alpha\in L$.   The group $\Gal(L/\QQ)$ acts transitively on the prime ideals of $\OO_L$ that divide $N(\p)$ since $L/\QQ$ is Galois, and hence (\ref{L:root of unity linear b}) and (\ref{L:root of unity linear c}) are equivalent.

For a fixed $e\in \ZZ^n$, define $\alpha:=\varphi(e)=\prod_{i=1}^n \pi_i^{e_i} \in L^\times$.    Observe that for a nonzero prime $\lambda$ of $\OO_L$, we have $v_\lambda(\alpha)=\sum_{i=1}^n e_i v_\lambda(\pi_i)$.  If $\alpha$ is a root of unity, then we have $v_\lambda(\alpha)=0$ for all $\lambda$.   Therefore, (\ref{L:root of unity linear a}) implies (\ref{L:root of unity linear b}).  We now assume that (\ref{L:root of unity linear b}) holds, i.e., $v_\lambda(\alpha)=0$ for all prime ideals $\lambda|N(\p)$ of $\OO_L$.   It suffices to prove that $\varphi(e)$ is a root of unity,

Take any nonzero prime ideal $\lambda \nmid N(\p)$ of $\OO_L$.  For each $\pi_i$, we have $\pi_i \bbar{\pi}_i = N(\p)$, where $\bbar{\pi}_i$ is the complex conjugate of $\pi_i$ under any complex embedding.   So $v_\lambda(\pi_i)+v_\lambda(\bbar\pi_i)=0$.   Since $\pi_i$ and $\bbar\pi_i$ are algebraic integers, we have $v_\lambda(\pi_i)\geq 0$ and $v_\lambda(\bbar\pi_i)\geq 0$, and hence $v_\lambda(\pi_i)=0$.   Therefore, $v_\lambda(\alpha)=0$.   Combining this with our assumption that (\ref{L:root of unity linear b}) holds, we deduce that $v_\lambda(\alpha)=0$ for all nonzero prime ideals $\lambda$ of $\OO_L$.   This implies that $\alpha\in \OO_L^\times$.

Take any embedding $\iota\colon L\hookrightarrow \CC$.  From Weil, we know that each $\iota(\pi_i)$ has absolute value $N(\p)^{1/2}$.  Therefore, $|\iota(\alpha)|=   N(\p)^{(e_1+\cdots +e_n)/2}$ for any $\iota$ and hence $|N_{L/\QQ}(\alpha)| = N(\p)^{[L:\QQ]\,  (e_1+\cdots +e_n)/2}$.   We have $N_{L/\QQ}(\alpha)=\pm 1$ since $\alpha\in \OO_L^\times$, so $e_1+\cdots +e_n=0$.    Therefore, $\alpha$ has absolute value $1$ under any embedding into $\CC$.   Since $\alpha$ is a unit in $\OO_L$ with absolute value $1$ under any embedding into $\CC$, we deduce that $\alpha$ is a root of unity.
\end{proof}

Let $M_0\subseteq \ZZ^n$ be the subgroup consisting of all $e\in \ZZ^n$ satisfying $\sum_{i=1}^n e_i v_\lambda(\pi_i) = 0$ for all prime ideals $\lambda$ of $\OO_L$ dividing $N(\p)$.   By Lemma~\ref{L:root of unity linear}, we have $\varphi^{-1}(\mu_L)=M_0$, where $\mu_L$ is the (finite) subgroup of roots of unity in $L^\times$.    Define the homomorphism
\[
\psi:=\varphi|_{M_0} \colon M_0 \to \mu_L,\quad e\mapsto \prod_{i=1}^n \pi^{e_i}.
\]
Computing $\psi$ on a basis of $M_0$, one can then explicitly compute $M := \ker(\psi) \subseteq \ZZ^n$.    Observe that $M$ is the kernel of $\varphi$.   Therefore, $\varphi$ induces an isomorphism
\[
\bbar{\varphi}\colon \ZZ^n/M \xrightarrow{\sim} \Phi_{A,\p}.
\]
Using Smith normal form, it is straightforward to compute the structure of $\ZZ^n/M$, and hence of $\Phi_{A,\p}$ as a finitely generated abelian group; in particular, one can compute its rank and whether it has nontrivial torsion.   The image of the standard basis elements in $\ZZ^n$ corresponds with the roots $\pi_1,\ldots, \pi_n$.

The action of $\Gal_\QQ$ on $\Phi_{A,\p}$ factors through $\Gal(L/\QQ)$; it is described by its action on the $\pi_i$.  Using $\bbar\varphi$, this gives an action of $\Gal(L/\QQ)$ on $\ZZ^n/M$.

\begin{remark} \label{R:revisit later}
We will revisit the computation of the groups $\Phi_{A,\p}$ in \S\ref{SS:computing Phi again}, where for simplicity we restrict to primes $\p$ for which $A$ has good and ordinary reduction.  We will be able to determine the structure of $\Phi_{A,\p}$ as an abelian group without having to explicitly compute a splitting field $L/\QQ$ of $P_{A,\p}(x)$.
\end{remark}

\section{The Mumford--Tate group}  \label{S:MT groups}

Fix a nonzero abelian variety $A$ defined  over a number field $K$.  Throughout we fix an embedding $\Kbar \hookrightarrow \CC$.   In particular, using this embedding we can view $K$ as a subfield of $\CC$ and $A(\CC)$ as a complex manifold.

\subsection{Mumford--Tate and Hodge groups} \label{SS:MT and Hodge groups}
The homology group 
\[
V_A:=H_1(A(\CC),\QQ),
\] 
is a vector space of dimension $2 \dim A$ over $\QQ$.   It is endowed with a $\QQ$-Hodge structure of type $\{(-1,0),(0,-1)\}$ and hence we have a decomposition 
\[
V_A\otimes_\QQ\CC = H_1(A(\CC),\CC)=V_A^{-1,0} \oplus V_A^{0,-1}
\]
satisfying $V_A^{0,-1}=\bbar{V_A^{-1,0}}$.  Let
\[
\mu\colon \GG_{m,\CC} \to \GL_{V_A\otimes_\QQ\CC}  
\]
be the cocharacter for which $\mu(z)$ is the automorphism of $V_A\otimes_\QQ\CC$ that is multiplication by $z$ on $V_A^{-1,0}$ and the identity on $V_A^{0,-1}$ for each $z\in\CC^\times=\GG_{m}(\CC)$.

\begin{defn}
The \defi{Mumford--Tate group} of $A$ is the smallest algebraic subgroup $\MT_A$ of $\GL_{V_A}$, defined over $\QQ$, which contains $\mu(\CC^\times)$.     The \defi{Hodge group} of $A$ is the smallest algebraic subgroup $\Hg_A$ of $\GL_{V_A}$, defined over $\QQ$, which contains $\mu(U(1))$, where $U(1):=\{z\in \CC^\times : |z|=1\}$.
\end{defn}

We have $\MT_A = \GG_m\cdot \Hg_A$, where $\GG_m$ is the group of homotheties in $\GL_{V_A}$, cf.~\cite{MR2062673}*{Remark~17.3.1}.  The groups $\MT_A$ and $\Hg_A$ are connected and reductive, cf.~Proposition~17.3.6 and Remark~17.3.1 of \cite{MR2062673}.     

The endomorphism ring $\End(A_{\Kbar})=\End(A_\CC)$ acts on $V_A$ and preserves the Hodge decomposition.  Therefore, $\End(A_{\Kbar})$ commutes with $\mu$ and hence also $\MT_A$.  Moreover, the ring $\End(A_{\Kbar})\otimes_\ZZ \QQ$ is naturally isomorphic to the commutant of $\MT_A$ in $\End_\QQ(V_A)$ (or equivalently, the commutant of $\Hg_A$ in $\End_\QQ(V_A)$), cf.~\cite{MR2062673}*{Proposition~17.3.4}. 

Choosing a polarization of $A$, we obtain a nondegenerate alternating pairing $E\colon V_A\times V_A\to \QQ$. We have $\Hg_A \subseteq \Sp_{V_A,E}$, where $\Sp_{V_A,E}\subseteq \GL_{V_A}$ is the symplectic group with respect to the pairing $E$, cf.~\cite{MR2062673}*{Proposition~17.3.2}.  Moreover, we have $\MT_A \subseteq \GSp_{V_A,E}$ and $\MT_A \cap \Sp_{V_A,E} = \Hg_A$.

\begin{remark}
As noted in \S\ref{S:introduction}, $\MT_A$ agrees with the identity component $G_A^\circ$, where $G_A\subseteq \GL_{V_A}$ is the motivic Galois groups of $A$ with respect to the category of motives in the sense of Andr\'e \cite{MR1423019}.
\end{remark}

\subsection{The Mumford--Tate conjecture}  \label{SS:MT conj}
Take any prime $\ell$.  The \defi{comparison isomorphism} $V_\ell(A)= V_A \otimes_\QQ \QQ_\ell$ induces an isomorphism $\GL_{V_\ell(A)} = \GL_{V_A,\,\QQ_\ell}$ of algebraic groups over $\QQ_\ell$.   Note that these isomorphisms depends on our fixed embedding $\Kbar\hookrightarrow \CC$.

 The following conjecture says that the connected $\QQ_\ell$-algebraic groups $G_{A,\ell}^\circ$ and $\MT_{A,\QQ_\ell}$ are the same when we use the comparison isomorphism as an identification, cf.~\S3 of \cite{MR0476753}.  

\begin{conj}[Mumford--Tate conjecture]   \label{C:MT}
For each prime $\ell$, we have $G_{A,\ell}^\circ= \MT_{A,\QQ_\ell}$.
\end{conj}
 
One inclusion of the Mumford--Tate conjecture is known unconditionally, see Deligne's proof in \cite[I, Prop.~6.2]{MR654325}.

\begin{prop}\label{P:MT inclusion}
For each prime $\ell$, we have $G_{A,\ell}^\circ \subseteq  \MT_{A,\QQ_\ell}$.
\end{prop}

The groups $G_{A,\ell}^\circ$ are reductive and have common rank, cf.~Lemma~\ref{L:common rank}(\ref{L:common rank iii}).    By Theorem~4.3 of \cite{MR1339927}, we have $G_{A,\ell}^\circ = \MT_{A,\QQ_\ell}$ if and only if the ranks of $G_{A,\ell}^\circ$ and $\MT_{A}$ are equal.  In particular, the Mumford--Tate conjecture for $A$ holds for one prime $\ell$ if and only if it holds for all $\ell$. 

\subsection{Frobenius compatibility} \label{SS:Frobenius conjugacy classes}

Define the field $L:=K_A^\conn$.    For each prime $\ell$, we have $\rho_{A,\ell}(\Gal_L)\subseteq G_{A,\ell}^\circ(\QQ_\ell)$.  By Proposition~\ref{P:MT inclusion}, we have a well-defined Galois representation 
\[
  \Gal_{L} \xrightarrow{\rho_{A,\ell}}  G_{A,\ell}^\circ(\QQ_\ell) \subseteq \MT_{A}(\QQ_\ell).
\] 
For a nonzero prime ideal $\p\nmid\ell$ of $\OO_L$ for which $A$ has good reduction, we obtain a conjugacy class $\rho_{A,\ell}(\Frob_\p)$ in $\MT_A(\QQ_\ell)$ and hence a well-defined element $F_{\p,\ell}:=\cl_{\MT_A}(\rho_{A,\ell}(\Frob_\p))$ in $\MT_A^\sharp(\QQ_\ell)$ with notation as in \S\ref{SS:ss conjugacy classes}.  Conjecturally, $F_{\p,\ell}$ is an element of $\MT_A^\sharp(\QQ)$ that is independent of $\ell$.   

\begin{conj}   \label{C:Frob conj}   
Let $L:=K_A^\conn$ and let $\p$ be any nonzero prime ideal of $\OO_L$ for which $A$ has good reduction. Then there exists an element $F_\p \in \MT_A^\sharp(\QQ)$ such that $F_{\p,\ell}=F_\p$ for all rational primes $\ell$ satisfying $\p\nmid \ell$.   

Equivalently, for all primes $\ell$ satisfying $\p\nmid \ell$ and all embeddings $\iota\colon \QQ_\ell \to \CC$, the conjugacy class of $\rho_{A,\ell}(\Frob_\p)$ in $\MT_A(\CC)$, using the embedding $\iota$, does not dependent on the choice of $\ell$ or $\iota$.
\end{conj}

\begin{remark}
\begin{romanenum}
\item
Observe that the image of $F_{\p,\ell}$ under the natural map $\MT_A^\sharp(\QQ_\ell) \to \GL_{V_A}^\sharp(\QQ_\ell)$ lies in $\GL_{V_A}^\sharp(\QQ)$ and is independent of $\ell$; this is equivalent to the characteristic polynomial $P_{A,\p}(x)$ of $\rho_{A,\ell}(\Frob_\p)$ having rational coefficients that do not depend on $\ell$.  So the above conjecture can be viewed as a strengthening of our earlier notion of compatibility. 

\item  
A more general version of Conjecture~\ref{C:Frob conj} was alluded to in \S\ref{S:introduction} that would work for any good prime ideal $\p$ of $\OO_K$ where the motivic Galois group of $A$ takes the role of $\MT_A$, cf.~Conjectures 12.6 and 3.4 in \cite{MR1265537}.
\end{romanenum} 
\end{remark}

As noted in Remark~\ref{R:main remarks}, Kisin and Zhou have announced a proof of Conjecture~\ref {C:Frob conj} when $\p\nmid 2$.  
In the meantime, the following partial results of Noot are useful.

\begin{prop}[Noot]
Let $\p$ be a nonzero prime ideal of the ring of integers of $K_A^\conn$ for which $A$ has good reduction.
\label{P:Noot}
\begin{romanenum}
\item \label{P:Noot i}
Suppose that the derived subgroup of $(\MT_A)_{\Qbar}$ has no direct factor isomorphic to $\SO_{2n}$ with $n\geq 4$, and that the quotient of two distinct roots of $P_{A,\p}(x)$ is never a root of unity.  Then $F_{\p,\ell}$ lies in $\MT^\sharp(\QQ)$ and is independent of $\ell$.
\item \label{P:Noot ii}
If $A$ has ordinary reduction at $\p$, then $F_{p,\ell}$ is in $\MT_A^\sharp(\QQ)$ and is independent of $\ell$.  Moreover, there is a semisimple $t_\p \in \MT_A(\QQ)$ that is conjugate to $\rho_{A,\ell}(\Frob_\p)$ in $\MT_A(\QQ_\ell)$ for all $\ell$ satisfying $\p\nmid \ell$.
\end{romanenum}
\end{prop}
\begin{proof}
We first prove (\ref{P:Noot i}).  In the notation of \cite{MR2472133}, we have $\operatorname{Conj}(\MT_A)=\MT_A^\sharp$ and $\operatorname{Conj}'(\MT_A)$ equals $\operatorname{Conj}(\MT_A)$ by our assumption on the derived subgroup of $(\MT_A)_{\Qbar}$.  Part (\ref{P:Noot i}) follows from Theorem~1.8 of \cite{MR2472133}.    Part (\ref{P:Noot ii}) follows from Theorem~2.2 of \cite{MR1370746}.
\end{proof}

\subsection{Minuscule representations}

The following proposition gives some constraints on the representation $\MT_A \hookrightarrow \GL_{V_A}$.   For the definition of minuscule, see \S\ref{SS:minuscule}.

\begin{prop}  \label{P:minuscule MT}
\begin{romanenum}
\item \label{P:minuscule MT i}
Every irreducible representation $U\subseteq V_A\otimes_\QQ \Qbar$ of $(\MT_A)_{\Qbar}$ is minuscule.
\item \label{P:minuscule MT ii}
Each irreducible component of the root system associated to $\MT_A$ is of classical type $A_n$, $B_n$, $C_n$ or $D_n$.
\end{romanenum}
\end{prop}
\begin{proof}   
Fix a maximal torus $T$ of $\MT_A$ and set  $W:=W(\MT_A,T)$.   Note that the algebraic group $\MT_A$ and its faithful representation $V_A$ satisfy the assumptions of \S3.2 of \cite{MR563476}. 
 
   Let $\Omega_{V_A} \subseteq X(T)$ be the set of weights of the representation $V_A$ of $\MT_A$.    Let $\Omega_{V_A}^+\subseteq \Omega_{V_A}$ be the set of highest weights of the irreducible representations $U\subseteq V_A\otimes_\QQ \Qbar$ of $(\MT_A)_\Qbar$  with respect to a fixed base of the root system $R \subseteq X(T)\otimes_\ZZ \QQ$ of $\MT_A$.  

Take any irreducible representation $U\subseteq V_A\otimes_\QQ \Qbar$ of $(\MT_A)_{\Qbar}$.  Let $\chi\in \Omega_{V_A}^+$ be the highest weight of $U$; it is the unique element in $\Omega_{V_A}^+ \cap \Omega_U$.    We have $W \cdot \chi \subseteq \Omega_U$ since $\Omega_U$ is $W$-stable.  Part (\ref{P:minuscule MT i}) of the proposition is thus equivalent to showing that $\Omega_{V_A} = W\cdot \Omega_{V_A}^+$.  This follows from Proposition~4 of \cite{MR563476}.

Part (\ref{P:minuscule MT ii}) is a consequence of the corollary to Proposition~7 in \cite{MR563476}.
\end{proof}

\section{Frobenius tori} \label{S:Frobenius tori}

Fix a nonzero abelian variety $A$ defined over a number field $K$.  Let $G$ be the quasi-split inner form of the Mumford--Tate group $\MT_A$.  Consider a representation
\[
\rho\colon G_{\Qbar}\xrightarrow{\sim} (\MT_A)_{\Qbar} \subseteq \GL_{V_A,\Qbar},
\] 
where the isomorphism is one arising from $G$ being an inner twist of $\MT_A$.   In particular, the representation $\rho$ of $G_{\Qbar}$ is well-defined up to isomorphism.

Let $\calS_A$ be the set of prime ideals of $\OO_K$ from Definition~\ref{D:SA}.  For each prime ideal $\p\in \calS_A$, we now show that $G$ has a maximal torus $T_\p$ satisfying some interesting properties.

\begin{thm} \label{T:Frobenius tori}
Assume that Conjectures~\ref{C:MT} and \ref{C:Frob conj} hold for $A$.   Take any $\p\in \calS_A$ and let $\P$ be a prime ideal of $\OO_{K_A^\conn}$ that divides $\p$.  Then there is a semisimple $t_\p \in G(\QQ)$ satisfying the following properties:
\begin{alphenum}
\item \label{T:Frobenius tori a}
$\cl_{\MT_A}(\rho(t_\p))=F_\P$, with $F_\P \in \MT_A^\sharp(\QQ)$ as in Conjecture~\ref{C:Frob conj},
\item \label{T:Frobenius tori b}
$t_\p$ lies in a unique maximal torus $T_\p$ of $G$, 
\item \label{T:Frobenius tori c}
the map
\begin{align}\label{E:Tp isomorphism}
X(T_\p) \to \Phi_{A,\p},\quad  \alpha \mapsto \alpha(t_\p)
\end{align}
is a well-defined isomorphism of abelian groups that respects the $\Gal_\QQ$-actions,
\item \label{T:Frobenius tori d}
under the isomorphism (\ref{E:Tp isomorphism}), the set $\Omega_\p\subseteq X(T_\p)$ of weights of $\rho$ with respect to $T_\p$ corresponds with the set $\calW_{A,\p}\subseteq \Phi_{A,\p}$ of roots of $P_{A,\p}(x)$.
\end{alphenum}
\end{thm}  

\subsection{Proof of Theorem~\ref{T:Frobenius tori}}

Let $A'$ be the base change of $A$ to $K_A^\conn$.  By Lemma~\ref{L:density of SA}(\ref{L:density of SA ii}), we have $\P\in \calS_{A'}$ and $P_{A,\p}(x)=P_{A',\P}(x)$.    In particular, $\calW_{A,\p}=\calW_{A',\P}$ and $\Phi_{A,\p}=\Phi_{A',\P}$. 
So to prove the theorem, there is no harm in replacing the pair  $(A,\p)$ by $(A',\P)$.  So without loss of generality, we may assume that $K_A^\conn=K$ and hence $\P=\p$.\\
 
Fix a maximal torus $T$ of $G$ that contains a maximal split torus of $G$.  The torus $T$ has rank $r$, where $r$ is the common rank of the groups $G_{A,\ell}$ since we are assuming the Mumford--Tate conjecture.    Let $\Omega\subseteq X(T)$ be the set of weights of $\rho$ with respect to $T$. 

Define the  set
\[
B_\p := \{ t\in T(\Qbar): \cl_{\MT_A}(\rho(t))=F_\p \}. 
\]
Conjugation induces an action of the Weyl group $W(G,T)$ on $B_\p$.

\begin{lemma} \label{L:simply transitive}
\begin{romanenum}
\item  \label{L:simply transitive i}
The action of $W(G,T)$ on $B_\p$ is simply transitive.
\item  \label{L:simply transitive ii}
Every $t_0\in B_\p$ generates $T_{\Qbar}$ as an algebraic group.
\item  \label{L:simply transitive iii}
For each $t_0\in B_\p$, the map 
\[
f\colon X(T)\to \Phi_{A,\p},\quad \alpha\mapsto \alpha(t_0)
\] 
is a well-defined group isomorphism that maps $\Omega$ to $\calW_{A,\p}$.
\end{romanenum}
\end{lemma}
\begin{proof}
Take any $t_0\in B_\p$.  We have $\cl_{\MT_A}(\rho(t_0))=F_\p$, so the characteristic polynomial of $\rho(t_0)$ is $P_{A,\p}(x)$.  Therefore, $\calW_{A,\p}$ is equal to the set of $\alpha(t_0)$ with $\alpha\in \Omega$.  The group $X(T)$ is generated by $\Omega$ since $\rho|_T\colon T \hookrightarrow \GL_{V_A}$ is an embedding.   Therefore, the image of $X(T)\to \Qbar^\times$, $\alpha\mapsto \alpha(t_0)$ is generated by $\calW_{A,\p}$.  We thus have a well-defined and surjective homomorphism 
\[
f\colon X(T) \to \Phi_{A,\p},\quad \alpha\mapsto \alpha(t_0)
\]
that sends $\Omega$ to $\calW_{A,\p}$.

The group $X(T)$ is free abelian of rank $r$. The group $\Phi_{A,\p}$ is also a free abelian of rank $r$ since $\p$ is in $\calS_A$.   Since $f$ is a surjective homomorphism between two free abelian groups of rank $r$, it must be an isomorphism.  This completes the proof of part (\ref{L:simply transitive iii}).

If $t_0$ did not generate the torus $T_{\Qbar}$, then there would be a nontrivial character $\alpha\in X(T)$ for which $\alpha(t_0)=1$.   Part (\ref{L:simply transitive ii}) is thus an immediate consequence of (\ref{L:simply transitive iii}).

It remains to prove (\ref{L:simply transitive i}).   Take any two $t_0, t_1 \in B_\p$.   We have $\rho(t_1)=h \rho(t_0) h^{-1}$ for some $h\in \MT_A(\Qbar)$ since $t_0$ and $t_1$ are semisimple and $\cl_{\MT_A}(\rho(t_0))=F_\p=\cl_{\MT_A}(\rho(t_1))$.  We thus have $t_1 = g t_0 g^{-1}$ for the unique $g\in G(\Qbar)$ that satisfies $\rho(g)=h$.   Since $t_1$ and $t_0$ both generate $T_{\Qbar}$ by (\ref{L:simply transitive ii}), we have $T_{\Qbar}=g\cdot T_{\Qbar} \cdot g^{-1}$.   So $g$ normalizes $T_{\Qbar}$ and gives rise to an element of $W(G,T)$ that sends $t_0$ to $t_1$.   The group $W(G,T)$ thus acts transitively on $B_\p$ since $t_0$ and  $t_1$ were arbitrary elements of $B_\p$.

Finally, we show that $W(G,T)$ acts faithfully on $B_\p$.    Suppose that $w\cdot t_0 = t_0$ for some $w\in W(G,T)$ and $t_0\in B_\p$.  By part (\ref{L:simply transitive ii}) this implies that $w$ acts trivially on the torus $T_{\Qbar}$.  We thus have $w=1$ since the group $W(G,T)$ acts faithfully on $T_{\Qbar}$.
\end{proof}

\begin{lemma} \label{L:Frobenius tori done}
If there is an element $t_\p \in G(\QQ)$ that is conjugate in $G(\Qbar)$ to some $t_0\in B_\p$, then $t_\p$ satisfies all the properties of Theorem~\ref{T:Frobenius tori}.
\end{lemma}
\begin{proof}
Assume that there is a $t_\p \in G(\QQ)$ which satisfies  $t_\p = g t_0 g^{-1}$ for some $g\in G(\Qbar)$ and $t_0\in B_\p$.

Let $T_\p$ be the algebraic subgroup of $G$ generated by $t_\p$.   Using Lemma~\ref{L:simply transitive}(\ref{L:simply transitive ii}), we find that $(T_\p)_{\Qbar}$ is a maximal torus of $G_{\Qbar}$; moreover, we have $(T_\p)_{\Qbar} =g\cdot T_{\Qbar} \cdot g^{-1}$.   Therefore, $T_\p$ is a maximal torus of $G$.  In particular, $t_\p$ is semisimple and $T_\p$ is the unique maximal torus of $G$ containing $t_\p$;  this proves that (\ref{T:Frobenius tori b}) in Theorem~\ref{T:Frobenius tori} holds for our $t_\p$.

Using $(T_\p)_{\Qbar} =g\cdot T_{\Qbar} \cdot g^{-1}$, we have a group isomorphism $\varphi\colon X(T_\p)\xrightarrow{\sim} X(T)$ satisfying $\varphi(\alpha)(t)=\alpha(g t g^{-1})$ for each $\alpha\in X(T_\p)$.   Define $f_\p \colon X(T_\p) \xrightarrow{\sim} \Phi_{A,\p}$ to be the isomorphism $f_\p = f \circ \varphi$ where $f$ is the isomorphism associated to $t_0$ from Lemma~\ref{L:simply transitive}(\ref{L:simply transitive iii}).   For each $\alpha\in X(T_\p)$, we have 
\[
f_\p(\alpha)=f(\varphi(\alpha))=(\varphi(\alpha))(t_0)= \alpha(g t_0 g^{-1}) = \alpha(t_\p).
\]
For each $\sigma\in \Gal_\QQ$, we have $\sigma(f_\p(\alpha))=\sigma(\alpha(t_\p))=\sigma(\alpha)(t_\p)=f_\p(\sigma(\alpha))$, where we have used that $\sigma(t_\p)=t_\p$ since $t_\p\in G(\QQ)$.  This proves that (\ref{T:Frobenius tori c}) in Theorem~\ref{T:Frobenius tori} holds for our $t_\p$.

Let $\Omega_\p\subseteq X(T_\p)$ be the set of weights of $\rho$ with respect to $T_\p$.  By Lemma~\ref{L:simply transitive}(\ref{L:simply transitive iii}), we have $f(\Omega)=\calW_{A,\p}$.  We have $\varphi(\Omega_\p)=\Omega$ and thus $f_\p(\Omega_\p)=f(\varphi(\Omega_\p))=f(\Omega)=\calW_{A,\p}$.  Therefore, part (\ref{T:Frobenius tori d}) of Theorem~\ref{T:Frobenius tori} holds for our $t_\p$.

We have $\rho(t_\p)=\rho(g) \rho(t_0) \rho(g)^{-1}$ and hence $\cl_{\MT_A}(\rho(t_\p))=\cl_{\MT_A}(\rho(t_0))$.  Since $t_0\in B_\p$, we have $\cl_{\MT_A}(\rho(t_\p))=F_\p$.   Therefore, part (\ref{T:Frobenius tori a}) of Theorem~\ref{T:Frobenius tori} holds for our $t_\p$.
\end{proof}

Fix an element $t_0\in B_\p$.    There is a natural action of $\Gal_\QQ$ on $B_\p$ since $T$, $\rho$ and $\cl_{\MT_A}$ are defined over $\QQ$ and $F_\p \in \MT_A^\sharp(\QQ)$.    For any $\sigma\in \Gal_\QQ$, Lemma~\ref{L:simply transitive}(\ref{L:simply transitive i}) implies that there is a unique $\xi_\sigma\in W(G,T)$ that satisfies $\sigma(t_0)=\xi_\sigma^{-1}\cdot t_0$.   For $\sigma,\tau\in \Gal_\QQ$, we have 
\[
(\sigma\tau)(t_0)=\sigma(\tau(t_0))=\sigma(\xi_\tau^{-1}\cdot t_0)=\sigma(\xi_\tau)^{-1} \cdot \sigma(t_0) =\sigma(\xi_\tau)^{-1} \cdot (\xi_\sigma^{-1}\cdot t_0) = (\xi_\sigma \sigma(\xi_\tau))^{-1}\cdot t_0
\]
ane hence $\xi_{\sigma\tau}=\xi_\sigma \sigma(\xi_\tau)$.  Therefore, the map 
\[
\xi\colon \Gal_\QQ \to W(G,T),\quad \sigma\mapsto \xi_\sigma
\]
is a $1$-cocycle.

\begin{lemma} \label{L:quasisplit cocycle}
There is an element $g\in G(\Qbar)$ such that   $g^{-1}\sigma(g)$ normalizes $T_\Qbar$ and is a representative of $\xi_\sigma$ in $W(G,T)$ for all $\sigma\in \Gal_\QQ$.
\end{lemma}
\begin{proof}
Let $p_1\colon H^1(\QQ,N_G(T)) \to H^1(\QQ, W(G,T))$ and $p_2\colon H^1(\QQ,N_G(T)) \to H^1(\QQ, G)$ be the maps of Galois cohomology sets induced by the natural maps $N_G(T)(\Qbar)\twoheadrightarrow W(G,T)$ and $N_G(T)\hookrightarrow G$.  By Proposition~\ref{P:Ragh}, there is an $x\in H^1(\QQ,N_G(T))$  such that $p_1(x)=[\xi]$ and $p_2(x)$ is the trivial class.  Here we have used that $G$ is quasi-split and that $T$ contains a maximally split torus in $G$.

Since $p_2(x)$ is the trivial class, there is an element $g\in G(\Qbar)$ such that $x=[\zeta]$, where $\zeta\colon \Gal_\QQ \to N_G(T)(\Qbar)$ is a $1$-cocycle satisfying $\zeta_\sigma=g^{-1}\sigma(g)$ for all $\sigma\in \Gal_\QQ$.  In particular,  $g^{-1}\sigma(g)= \zeta_\sigma$ is in $N_G(T)(\Qbar)$ for $\sigma\in \Gal_\QQ$.   Let $[g^{-1} \sigma(g)]$ be the image of $g^{-1}\sigma(g)$ in $W(G,T)$.  Since $[\xi]=p_1(x)=p_1([\zeta])$, there is an element $w\in W(G,T)$ such that 
\[
\xi_\sigma = w^{-1}\cdot  [g^{-1} \sigma(g)] \cdot\sigma(w)\] 
for all $\sigma\in \Gal_\QQ$.  Let $n\in N_G(T)(\Qbar)$ be a representative of $w$.     After replacing $g$ by $gn^{-1}$, we may further assume that $w=1$; the lemma is now immediate.
\end{proof}

Take $g\in G(\Qbar)$ as in Lemma~\ref{L:quasisplit cocycle} and define $t_\p:=g t_0 g^{-1} \in G(\Qbar)$.  For any $\sigma\in \Gal_\QQ$, we have
\[
\sigma(t_0)=\xi_\sigma^{-1} \cdot t_0 = \sigma(g)^{-1} g t_0 g^{-1} \sigma(g)
\]
and hence $\sigma(t_\p)=\sigma(g t_0 g^{-1}) =\sigma(g) \sigma(t_0) \sigma(g)^{-1} = g t_0 g^{-1}=t_\p$.   Since $\sigma(t_\p)=t_\p$ holds for all $\sigma\in \Gal_\QQ$, we have $t_\p\in G(\QQ)$.  Theorem~\ref{T:Frobenius tori} is now an immediate consequence of Lemma~\ref{L:Frobenius tori done}.

\section{Galois groups of Frobenius polynomials} \label{S:Galois groups of Frob}

Fix a nonzero abelian variety $A$ over a number field $K$ for which $K_A^\conn=K$. Throughout \S\ref{S:Galois groups of Frob}, we shall assume that Conjectures~\ref{C:MT} and \ref{C:Frob conj} hold for the abelian variety $A$.    \\

Let $G$ be the quasi-split inner form of the Mumford--Tate group $\MT_A$. Let $\calS_A$ be the set of prime ideals of $\OO_K$ from Definition~\ref{D:SA}; it has  density $1$ by Lemma~\ref{L:density of SA}(\ref{L:density of SA i}).  For each prime ideal $\p\in \calS_A$, we have a maximal torus $T_\p$ of $G$ as in Theorem~\ref{T:Frobenius tori} generated by a semisimple $t_\p\in G(\QQ)$.

Fix a prime ideal $\p\in \calS_A$.  With notation as in \S\ref{SS:Galois action}, the natural action of $\Gal_\QQ$ on the group $X(T_\p)$ can be expressed in terms of a representation
\[
\varphi_\p:=\varphi_{G,T_\p}\colon \Gal_\QQ \to \Aut(\Psi(G,T_\p)) \subseteq \Aut(X(T_\p))
\]
and we have an inclusion $\varphi_\p(\Gal_\QQ)\subseteq \Gamma(G,T_\p)$.    

Let  $k_G$ be the fixed field in $\Qbar$ of the $\ker \mu_G$, where $\mu_G$ is defined in \S\ref{SS:Galois action}.  The field $k_G$ can also be described as the minimal extension of $\QQ$ in $\Qbar$ for which $\varphi_\p(\Gal(\Qbar/k_G)) \subseteq W(G,T_\p)$ holds for a fixed $\p\in \calS_A$; note that it does not depend on the choice of prime $\p$.   The splitting field of $T_\p$ is $\QQ(\calW_{A,\p})$ and hence $k_G \subseteq \QQ(\calW_{A,\p})$ for all $\p\in \calS_A$.  We can similarly define a field $k_{\MT_A}$ associated to the Mumford--Tate group of $A$ and it equals $k_G$ by Proposition~\ref{P:characterize inner forms}.

The following is the main result of this section; it describes the image of $\varphi_\p$ for all $\p\in \calS_A$ away from a set of density $0$.

\begin{thm} \label{T:Frobenius Galois groups}
Assume that Conjectures~\ref{C:MT} and \ref{C:Frob conj} hold for $A$.
For a fixed number field $k_G \subseteq L \subseteq \Qbar$, we have
\[
\varphi_\p(\Gal_L) = W(G,T_\p) \quad \text{ and } \quad \varphi_\p(\Gal_\QQ)=\Gamma(G,T_\p)
\]
for all $\p\in \calS_A$ away from a set of density $0$.
\end{thm}

The following corollary describes the Galois group of $P_{A,\p}(x)$ for almost all prime ideals $\p$ of $\OO_K$.   

\begin{cor} \label{C:Frobenius Galois groups}
Assume that Conjectures~\ref{C:MT} and \ref{C:Frob conj} hold for $A$.   For a fixed number field $k_{\MT_A} \subseteq L \subseteq \Qbar$, we have
\[
\Gal(L(\calW_{A,\p})/L) \cong W(\MT_A) \quad \text{ and } \quad \Gal(\QQ(\calW_{A,\p})/\QQ) \cong \Gamma(\MT_A)
\]
for all nonzero prime ideals $\p$ of $\OO_K$ away from a set of density $0$.
\end{cor}
\begin{proof}
As already noted, we have $k_G=k_{\MT_A}$.   
For $\p\in S_A$, we have $\varphi_\p(\Gal_L) \cong \Gal(L(\calW_{A,\p})/L)$ and $\varphi_\p(\Gal_\QQ) \cong \Gal(\QQ(\calW_{A,\p})/\QQ)$ since $X(T_\p)$ and $\Phi_{A,\p}$ are isomorphic $\Gal_\QQ$-modules by Theorem~\ref{T:Frobenius tori} and the group $\Phi_{A,\p}$ is generated by $\calW_{A,\p}$.   Since $G$ and $\MT_A$ are inner forms of each other, we have $W(G,T_\p)\cong W(\MT_A)$ and  $\Gamma(G,T_\p)\cong \Gamma(\MT_A)$.  The corollary is now a direct consequence of Theorem~\ref{T:Frobenius Galois groups} and $\calS_A$ having density $1$.
\end{proof}

\subsection{Proof of Theorem~\ref{T:Frobenius Galois groups new}}
Recall that $G_A^\circ=\MT_A$.  In the special case where $K_A^\conn=K$, the theorem is equivalent to Corollary~\ref{C:Frobenius Galois groups}.  We now consider the general case.  

Let $A'$ be the base change of $A$ to $K_A^\conn$; it has the same Mumford--Tate group.  Take any $\p\in \calS_A$ and any prime ideal $\P$ of $\OO_{K_A^\conn}$ dividing $\p$.  By Lemma~\ref{L:density of SA}(\ref{L:density of SA ii}), we have $\P \in \calS_{A'}$ and $P_{A',\P}(x)=P_{A,\p}(x)$.  In particular, $L(\calW_{A,\p})=L(\calW_{A',\P})$.  Therefore, 
\begin{align*}
& |\{ \p \in \calS_A :  N(\p)\leq x,\, \Gal(L(\calW_{A,\p})/L) \not\cong W(\MT_A)\}| \\\leq & |\{ \p \in \calS_{A'} :  N(\P)\leq x,\, \Gal(L(\calW_{A',\P})/L) \not\cong W(\MT_A)\}|\\
 =& o(x/\log x)
\end{align*}
as $x\to \infty$, where we have used that $A'$ satisfies the special case of the theorem already proved.  Therefore for all $\p\in \calS_A$ away from a set of density $0$, we have $\Gal(L(\calW_{A,\p})/L) \cong W(\MT_A)$.  The first part of theorem is immediate since $\calS_A$ consists of all prime ideals of $\OO_K$, away from a set of density $0$, that split completely in $K_A^\conn$.  The proof of the second part of the theorem is identical.

\subsection{Proof of Theorem~\ref{T:Frobenius Galois groups}}
Since $L\supseteq k_{G}$, we have $\varphi_\p(\Gal_L) \subseteq W(G,T_\p)$.  From the definition of $\Gamma(G,T_\p)$, the homomorphism $\Gal_\QQ \xrightarrow{\varphi_\p} \Gamma(G,T_\p)\to \Gamma(G,T_\p)/W(G,T_\p)$ is surjective and hence to prove the theorem it suffices to show that $\varphi_\p(\Gal_L) = W(G,T_\p)$ for all $\p\in \calS_A$ away from a set of density $0$

After possibly replacing $L$ by a finite extension, we may assume that the group $G_L$ is split.  Let $\Lambda$ be the set of primes $\ell$ that split completely in $L$.  Note that $G_{\QQ_\ell}$ is split for all $\ell\in \Lambda$.\\

Fix a prime $\ell \in \Lambda$.          By the Mumford--Tate conjecture, we have $G_{A,\ell}=(\MT_A)_{\QQ_\ell}$.              Since $G$ is the quasi-split inner form of $\MT_A$, the group $G_{\QQ_\ell}$ is the quasi-split inner form of $(\MT_A)_{\QQ_\ell}$.    Since $G_{\QQ_\ell}$ is split by our choice of $\ell$, we deduce that $(\MT_A)_{\QQ_\ell}$ is split.  By the Mumford--Tate conjecture, $G_{A,\ell}$ is also split.  Fix a split maximal torus $T$ of $G_{A,\ell}$.   \\

Take any prime ideal $\p\in \calS_A'$ satisfying $\p\nmid \ell$, where $\calS_A'$ is the set of $\p\in \calS_A$ for which $N(\p)$ is prime. 

Choose an element $t_{\p,\ell}\in T(\Qbar_\ell)$ that is conjugate to $\rho_{A,\ell}(\Frob_\p)$ in $G_{A,\ell}(\Qbar_\ell)$.  As in \S5 of \cite{MR3264675}, we obtain a homomorphism
\[
\psi_{\p,\ell}\colon \Gal_{\QQ_\ell}\to W(G_{A,\ell},T)
\]
that is characterized by the property that $\sigma(t_{\p,\ell})= \psi_{\p,\ell}(\sigma)^{-1}(t_{\p,\ell})$ for all $\sigma\in \Gal_{\QQ_\ell}$.  Note that a different choice of $t_{\p,\ell} \in T(\Qbar_\ell)$ would alter $\psi_{\p,\ell}$ by an inner automorphism.  A different choice $T'$ of split maximal torus of $G_{A,\ell}$, would give rise to the same homomorphism after first composing with an appropriate isomorphism $W(G_{A,\ell},T) \xrightarrow{\sim} W(G_{A,\ell},T')$.\\

We now relate $\psi_{\p,\ell}$ with our homomorphism $\varphi_\p$.  
Choose an embedding $\Qbar \hookrightarrow \Qbar_\ell$; it gives rise to an injective homomorphism $\Gal_{\QQ_\ell}=\Gal(\Qbar_\ell/\QQ_\ell) \hookrightarrow \Gal(\Qbar/\QQ)=\Gal_\QQ$.   We have $\varphi_\p(\Gal_{\QQ_\ell}) \subseteq W(G,T_\p)$ since $\varphi_\p(\Gal_L)\subseteq W(G,T_\p)$ and $\ell$ splits completely in $L$.  Let
\[
\varphi_{\p,\ell}\colon \Gal_{\QQ_\ell} \to W(G) 
\]
be the homomorphism obtained by composing $\Gal_{\QQ_\ell} \hookrightarrow \Gal_\QQ$ with $\varphi_\p$ and an isomorphism $W(G,T_\p)\xrightarrow{\sim} W(G)$ as described in \S\ref{SS:Weyl definition}.  

\begin{lemma} \label{L:beta connecting}
For each $\p  \in \calS_A'$ satisfying $\p\nmid \ell$, there is an isomorphism $\beta\colon W(G_{A,\ell},T) \xrightarrow{\sim} W(G)$ satisfying $\varphi_{\p,\ell}=\beta\circ \psi_{\p,\ell}$.  Moreover, one can take the isomorphism $\beta$ so that, up to composition with an inner automorphism of $W(G)$, it does not depend on $\p$.
\end{lemma}
\begin{proof}
Let $\rho\colon G_{\Qbar_\ell} \to (\MT_A)_{\Qbar_\ell}$ be an isomorphism arising from $G$ being an inner twist of $\MT_A$.   With our fixed $t_\p\in T_\p(\QQ)$ as in Theorem~\ref{T:Frobenius tori}, we have $\cl_{\MT_A}(\rho(t_\p))=F_\p$.  By Conjecture~\ref{C:Frob conj},  $\rho(t_\p)$ is conjugate to $t_{\p,\ell}$ in $G_{A,\ell}(\Qbar_\ell)$.  So by composing $\rho$ with an inner automorphism, we may assume that $\rho(t_\p)=t_{\p,\ell}$.  Since the torus $T_\p$ is generated by $t_\p$, we find that $\rho$ restricts to an isomorphism $(T_\p)_{\Qbar_\ell} \to T_{\Qbar_\ell}$.  We thus have a group isomorphism
\[
\rho^* \colon X(T) \xrightarrow{\sim} X(T_\p),\quad \alpha\mapsto \alpha\circ \rho|_{(T_\p)_{\Qbar_\ell}}.
\]

Take any $\sigma\in \Gal_{\QQ_\ell}$ and $\alpha\in X(T_\p)$.  We have $(\varphi_\p(\sigma) \alpha)(t_\p) = \sigma(\alpha)(t_\p) =\sigma(\alpha(t_\p))$ since $t_\p \in T_\p(\QQ)$.    We have $\alpha(t_\p)= \alpha(\rho^{-1}(t_{\p,\ell})) = ((\rho^*)^{-1}(\alpha))(t_{\p,\ell})$.   The character $(\rho^*)^{-1}(\alpha)$ in $X(T)$ is fixed by $\sigma$ since $T$ is a split torus defined over $\QQ_\ell$.   Therefore,
\begin{align*}
(\varphi_\p(\sigma) \alpha)(t_\p)  &=  ((\rho^*)^{-1}(\alpha))(\sigma(t_{\p,\ell}))\\
&= ((\rho^*)^{-1}(\alpha))(\psi_{\p,\ell}(\sigma)^{-1}t_{\p,\ell}))\\
& = ((\rho^*)^{-1}(\alpha))(\psi_{\p,\ell}(\sigma)^{-1} \rho(t_{\p})) \\
&=  ((\rho^* \circ \psi_{\p,\ell}(\sigma) \circ (\rho^*)^{-1})(\alpha))(t_\p).
\end{align*}
Since $t_\p$ generates $T_\p$ and $\alpha\in X(T_\p)$ was arbitrary, we deduce that $\varphi_\p(\sigma)= \rho^* \circ \psi_{\p,\ell}(\sigma) \circ (\rho^*)^{-1}$ holds for all $\sigma\in \Gal_{\QQ_\ell}$.       The isomorphism $N_{G_{A,\ell}}(T)(\Qbar_\ell) \to N_G(T_\p)(\Qbar_\ell)$ given by $n\mapsto \rho^{-1}(n)$ induces an isomorphism $\beta\colon W(G_{A,\ell},T)\to W(G,T_\p)$.    

Take any $n \in N_{G_{A,\ell}}(T)(\Qbar_\ell)$ and let $w:=[n]$ be its image in $W(G_{A,\ell},T)$.    For any $\alpha \in X(T_\p)$ and $t\in T_\p(\Qbar_\ell)$, we have 
\begin{align*}
((\rho^* \circ w \circ (\rho^*)^{-1})(\alpha))(t) &= \alpha( \rho^{-1}(n^{-1} \rho(t) n)) 
= \alpha( (\rho^{-1}(n))^{-1} \, t \, \rho^{-1}(n) ) = (\beta(w)(\alpha))(t).
\end{align*}
Since $\alpha$ and $t$ were arbitrary, we deduce that $\rho^* \circ w \circ (\rho^*)^{-1} = \beta(w)$ for all $w\in W(G_{A,\ell},T)$.  In particular, we have $\varphi_\p(\sigma) = \beta(\psi_{\p,\ell}(\sigma))$ for all $\sigma\in \Gal_{\QQ_\ell}$.   The last statement in the lemma follows by noting that in the proof, we only changed $\rho$ by composition with an inner automorphism.
\end{proof}
   
Let $\bbar\rho_{A,\ell}\colon \Gal_K\to \Aut_{\FF_\ell}(A[\ell])$ be the representation describing the Galois action on the $\ell$-torsion points of $A$.   

\begin{lemma}  \label{L:sieving input}
Take any subset $C$ of $W(G)$ that is stable under conjugation.   Take any prime $\ell \in \Lambda$.  There is a subset $U_\ell$ of $\bbar\rho_{A,\ell}(\Gal_K)$ that is stable under conjugation and satisfies the following properties:
\begin{itemize}
\item 
If $\p\in \calS_A'$ satisfies $\p\nmid \ell$ and $\bbar\rho_{A,\ell}(\Frob_\p) \subseteq U_\ell$, then $\varphi_{\p,\ell}$ is unramified and $\varphi_{\p,\ell}(\Frob_\ell)\subseteq C$.
\item Let $K'$ be a finite extension of $K$ and let $\kappa$ be a subset of $\Gal_K$ that consists of a union of cosets of $\Gal_{K'}$.  Then we have
\[
\frac{|\bbar{\rho}_{A,\ell}(\kappa) \cap U_\ell|}{|\bbar{\rho}_{A,\ell}(\kappa)|} = \frac{|C|}{|W(G)|} + O(1/\ell)
\]
where the implicit constant depends only on $A$ and $K'$.
\end{itemize}
\end{lemma}
\begin{proof}
By excluding a finite number of $\ell$, one can assume that the group scheme $\calG_{A,\ell}$ of \S5 of \cite{MR3264675} is a split reductive group scheme (we can take $U_\ell=\emptyset$ and adjust the implicit constant appropriately to deal with these excluded primes).   We may assume that our split torus $T$ of $G_{A,\ell}$ is chosen to be the generic fiber of a split maximal torus of $\calG_{A,\ell}$.  

Take any subset $C'$ of $W(G_{A,\ell},T)$ that is stable under conjugation.   Lemma 5.1 of \cite{MR3264675} implies that 
Then there is a subset $U_\ell$ of $\bbar\rho_{A,\ell}(\Gal_K)$ that is stable under conjugation and satisfies the following:
\begin{itemize}
\item
If $\p \in \calS_{A}'$ satisfies $\p\nmid \ell$ and $\bbar\rho_{A,\ell}(\Frob_\p) \subseteq U_\ell$, then $\psi_{\p,\ell}$ is unramified and $\psi_{\p,\ell}(\Frob_\ell)\subseteq C'$.
\item
Let $K'$ be a finite extension of $K$ and let $\kappa$ be a subset of $\Gal_K$ that consists of a union of cosets of $\Gal_{K'}$.  Then we have
\[
\frac{|\bbar{\rho}_{A,\ell}(\kappa) \cap U_\ell|}{|\bbar{\rho}_{A,\ell}(\kappa)|} = \frac{|C'|}{|W(G_{A,\ell}, T_0)|} + O(1/\ell),\\
\]
where the implicit constant depends only on $A$ and $K'$.
\end{itemize}
The lemma is now an immediate consequence of Lemma~\ref{L:beta connecting} by taking $C':=\beta^{-1}(C)$.
\end{proof}

Serre proved that there is a finite Galois extension $K'/K$ such that $(\prod_\ell \rho_{A,\ell})(\Gal_{K'})$ equals $\prod_\ell \rho_{A,\ell}(\Gal_{K'})$ where the products are over all primes $\ell$, cf.~\cite{MR1730973}*{\#138}.  The following is Proposition 2.12 of \cite{MR3264675}.

\begin{lemma}  \label{L:independence density}
Let $\Lambda_0$ be a finite set of rational primes.   For each prime $\ell \in \Lambda_0$, fix a subset $U_\ell$ of the group $\bbar\rho_{A,\ell}(\Gal_K)$ that is stable under conjugation.  Then the set of all nonzero primes ideals $\p \subseteq \OO_K$ for which $\bbar\rho_{A,\ell}(\Frob_\p)\subseteq U_\ell$ for all $\ell\in \Lambda_0$ has density
\[
\sum_{C} \frac{|C|}{|\Gal(K'/K)|}\cdot \prod_{\ell\in \Lambda_0} \frac{|\bbar\rho_{A,\ell}(\Gamma_C) \cap U_\ell |}{|\bbar\rho_{A,\ell}(\Gamma_C)|},
\]
where $C$ varies over the conjugacy classes of $\Gal(K'/K)$ and $\Gamma_C$ is the set of $\sigma\in \Gal_K$ for which $\sigma|_{K'}\in C$. 
\end{lemma}
   
For $\p \in \calS_A$, we know that $\varphi_\p(\Gal_L)\subseteq W(G,T_\p)$.  Let $\varphi_\p'\colon \Gal_L \to W(G)$ be the homomorphism obtained by composing $\varphi_\p$ with an isomorphism $W(G,T_\p)\cong W(G)$ as in \S\ref{SS:Weyl definition}.    It thus remains to show that $\varphi_\p'(\Gal_L)= W(G)$ for all $\p\in \calS_A$ away from a set of density $0$.   

Take any conjugacy class $C$ of $W(G)$.   Define the set 
\[
\calB_C=\{\p\in \calS_A' : \varphi_\p'(\Gal_L)\cap C = \emptyset\}.
\]   
Fix an $x\geq 2$.  Let $\Lambda(x)$ be the set of $\ell\in \Lambda$ satisfying $\ell\leq x$.  For each $\ell\in \Lambda$, let $U_\ell$ be the subset of $\bbar\rho_{A,\ell}(\Gal_K)$ as in Lemma~\ref{L:sieving input}.

 Take any prime $\p\in \calB_C$  not dividing one of the finite number of primes in $\Lambda(x)$.   Suppose that $\bbar\rho_{A,\ell}(\Frob_\p) \subseteq U_\ell$ for some $\ell\in \Lambda(x)$.
Then $\varphi_{\p,\ell}$ is unramified at $\ell$ and satisfies $\varphi_{\p,\ell}(\Frob_\ell) \subseteq C$ by our choice of $U_\ell$.   Since $\ell$ splits completely in $L$, we deduce that $\varphi_\p'(\Gal_L)$ contains an element in the conjugacy class $C$.  However, this is contradicts that $\p \in \calB_C$.     Therefore, $\bbar\rho_{A,\ell}(\Frob_\p) \not\subseteq U_\ell$ for all $\ell\in \Lambda(x)$ and all $\p \in \calB_C$ that do not divide any of the finitely many primes in $\Lambda(x)$.

Using Lemma~\ref{L:independence density},  we deduce that the set $\calB_C$ is contained in a set of primes that has density 
\begin{align*}
\delta(x):=&\sum_{\kappa} \frac{|\kappa|}{|\Gal(K'/K)|}\cdot \prod_{\ell\in \Lambda(x)} \frac{|\bbar\rho_{A,\ell}(\Gamma_\kappa) \cap (\bbar\rho_{A,\ell}(\Gal_K)-U_\ell) |}{|\bbar\rho_{A,\ell}(\Gamma_\kappa)|}\\
=& \sum_{\kappa} \frac{|\kappa|}{|\Gal(K'/K)|}\cdot \prod_{\ell\in \Lambda(x)} \bigg( 1- \frac{|\bbar\rho_{A,\ell}(\Gamma_\kappa)\cap U_\ell) |}{|\bbar\rho_{A,\ell}(\Gamma_\kappa)|}\bigg),
\end{align*}
where $\kappa$ varies over the conjugacy classes of $\Gal(K'/K)$ and $\Gamma_\kappa$ is the set of $\sigma\in \Gal_K$ for which $\sigma|_{K'}\in \kappa$. 

We may assume that $|C|/|W(G)|<1$ since  otherwise $W(G)=1$ and the theorem is immediate.  Choose $\varepsilon>0$ such that $1-|C|/|W(G)|+\varepsilon<1$.    Then by using the cardinality estimates from Lemma~\ref{L:sieving input}, we have
\[
\delta(x) =O\Big( (1-|C|/|W(G)|+\varepsilon)^{|\Lambda(x)|}   \Big),
\]
where the implicit constant depends on $A$, $K'$, $L$ and $\varepsilon$.    Since $0<1-|C|/|W(G)|+\varepsilon <1$, we have $\delta(x)\to 0$ as $x\to +\infty$.   Therefore, the set $\calB_C$ has density $0$.

Take any prime ideal $\p \in \calS_A'-(\cup_{C} \calB_C)$, where the union is over the set of conjugacy classes $C$ of $W(G)$.   By our choice of $\p$, $\varphi_\p'(\Gal_L)$ is a subgroup of $W(G)$ that satisfies $\varphi_\p'(\Gal_L)\cap C \neq \emptyset$ for all conjugacy classes $C$ of $W(G)$.  By Jordan's lemma \cite{MR1997347}*{Theorem~4'}, we deduce that $\varphi_\p'(\Gal_L)=W(G)$.   It thus suffices to show that $\calS_A'-(\cup_{C} \calB_C)$ has density $1$.   This is clear since the sets $\calB_C$ have density $0$, the set $\calS_A$ has density $1$, and the set of prime ideals $\p \subseteq \OO_K$ with $N(\p)$ a prime has density $1$.

\section{Roots systems and minuscule representations} \label{S:roots}

Fix a field $k$ of characteristic $0$ and an algebraic closure $\kbar$.  Let $G$ be a connected  reductive group defined over $k$.   Let 
\[
\rho\colon G\to \GL_V
\]
be a representation whose kernel is a subgroup of the center of $G$, where $V$ is a finite dimensional $k$-vector space.  Fix a maximal torus $T$ of $G$.   For every representation $U\subseteq V \otimes_k \kbar$ of $G_{\kbar}$, denote by $\Omega_U \subseteq X(T)$ the set of weights of $T_{\kbar}$ acting on $U$.  

Assume further that every irreducible representation $U \subseteq V\otimes_k \kbar$ of $G_{\kbar}$ is \defi{minuscule}, i.e., the Weyl group $W(G,T)$ acts transitively on $\Omega_U$. We will later apply the results of this section to the representation from Proposition~\ref{P:minuscule MT}.

Let $R(G,T) \subseteq X(T)$ be the set of roots of $G$ relative to $T$.      The goal of this section is describe how one can compute $R(G,T)$ directly from the action of $W(G,T)$ on $X(T)$ and the set $\Omega_V$.\\

For each $W(G,T)$-orbit $\Omega\subseteq \Omega_V$, define the subset
\[
C_\Omega:= \{ \alpha\beta^{-1}: \alpha,\beta \in \Omega, \alpha\neq \beta\}.
\]
of $X(T)$.  The set $C_\Omega$ is stable under the action of $W(G,T)$ on $X(T)$.     The following algorithm will produce a sequence of nonempty finite subsets $\calS_1,\ldots, \calS_s$ of $X(T)$; it will terminate since the finite sets $\mathcal{U}_i$ that arise have strictly decreasing cardinality.  

\begin{algorithm} \label{algorithm 1}
\hfill
\begin{enumerate}
\item
Set $i=1$ and define $\mathcal{U}_1:=\bigcup_\Omega C_\Omega$,  where $\Omega$ runs over the $W(G,T)$-orbits in $\Omega_V$.   If $\mathcal{U}_1=\emptyset$, then we take $s=0$ and terminate the algorithm.
\item  \label{I:recurse}
Choose a $W(G,T)$-orbit $\OO$ of $\mathcal{U}_i$ that has minimal cardinality.    Choose a $W(G,T)$-orbit $\Omega\subseteq \Omega_V$ for which $\OO \subseteq C_{\Omega}$.  
\item
Define $\calS_i$ to be the set of elements of $C_{\Omega}$ that lie in the span of $\OO$ in $X(T)\otimes_\ZZ\QQ$.
\item
Let $\mathcal{U}_{i+1}$ be the set of elements of $\mathcal{U}_i$ that are not in the span of $\calS_1\cup \cdots \cup \calS_i$ in $X(T)\otimes_\ZZ \QQ$.
\item
If $\mathcal{U}_{i+1}\neq \emptyset$, then increase $i$ by $1$ and go back to step (\ref{I:recurse}).
\end{enumerate}
\end{algorithm}

The following shows that the sets $\calS_i$ are related to the irreducible components of $R(G,T)$.

\begin{prop} \label{P:Lie 1}
For each $1\leq i \leq s$, there is a unique irreducible component $R_i$ of the root system $R(G,T)\subseteq X(T)\otimes_{\ZZ} \RR$ with $R_i\subseteq \calS_i$.  Moreover, 
\[
R(G,T)= \bigcup_{1\leq i \leq s} R_i
\]
is a disjoint union that gives the decomposition of $R(G,T)$ into irreducible root systems.
\end{prop}

We now explain how to compute the Lie type of the irreducible root systems $R_i$.   Take any $1\leq i \leq s$.    Let $W_i$ be the quotient of $W(G,T)$ that acts faithfully on $\calS_i$.   Let $r_i$ be the dimension of the span of $\calS_i$ in $X(T)\otimes_\ZZ \QQ$.

\begin{prop}  \label{P:Lie 2}
Fix an $1\leq i \leq s$.   Set $r:=r_i$ and $W:=W_i$.
\begin{romanenum}
\item
If $r\geq 1$, then $R_i$ has type $A_r$ if and only if $|W|=(r+1)!$.
\item
If $r\geq 3$, then $R_i$ has type $B_r$ if and only if $|W|=2^r r!$ and $\calS_i$ consists of at least three $W$-orbits.
\item
If $r\geq 2$, then $R_i$ has type $C_r$ if and only if $|W|=2^r r!$ and $\calS_i$ consists of two $W$-orbits.
\item
If $r\geq 4$, then $R_i$ has type $D_r$ if and only if $|W|=2^{r-1} r!$.
\item
If $r=6$, then $R_i$ has type $E_6$ if and only if $|W|=51840$.
\item
If $r=7$, then $R_i$ has type $E_7$ if and only if $|W|=2903040$.
\item
The root system $R_i$ is not of any of the following types: $G_2$, $F_4$, $E_8$.
\end{romanenum}
\end{prop}

\begin{remark}
Note that Proposition~\ref{P:Lie 2} determines the Lie type of each $R_i$.  The conditions on $r$ are imposed to avoid the ambiguity of the exceptional isomorphisms $A_1=B_1=C_1$, $B_2=C_2$ and $A_3=D_3$.
\end{remark}

The following proposition shows how one can determine $R_i$ as a subset of $\calS_i$ by considering the action of $W(G,T)$ on $\calS_i$  (with the type of $R_i$ being computable by Proposition~\ref{P:Lie 2}).

\begin{prop}  \label{P:Lie 3}
Fix an $1\leq i \leq s$.   Set $r:=r_i$, $W:=W_i$ and $\calS:=\calS_i$. 
\begin{romanenum}
\item
If $r\geq 1$ and $R_i$ is of type $A_r$, then $R_i$ is the unique $W$-orbit of $\calS$ of cardinality $r(r+1)$.
\item
If $r\geq 3$ and $R_i$ is of type $B_r$, then $R_i$ is the union of the unique $W$-orbits of $\calS$ of cardinality  $2r$ and $2r(r-1)$.

\item
If $r\geq 2$ and $R_i$ is of type $C_r$, then $R_i=\calS$.

\item
If $r\geq 4$ and $R_i$ is of type $D_r$, then $R_i$ is the unique $W$-orbit of $\calS$ with cardinality $2r(r-1)$.

\item
If $R_i$ is of type $E_6$, then $R_i$ is the unique $W$-orbit of $\calS$ with cardinality $72$.

\item
If $R_i$ is of type $E_7$, then $R_i$ is the unique $W$-orbit of $\calS$ with cardinality $126$.

\end{romanenum}
\end{prop}

Using Propositions~\ref{P:Lie 1}, \ref{P:Lie 2} and \ref{P:Lie 3}, we find that the set of roots $R(G,T)$ can be computed from the set of weights $\Omega_V\subseteq X(T)$ and the action of $W(G,T)$ on $X(T)$.

\subsection{Proof of Propositions~\ref{P:Lie 1}, \ref{P:Lie 2} and \ref{P:Lie 3}}
Without loss of generality, we may assume that $k$ is algebraically closed.  In this proof, we will view $X(T)$ and all the other character groups that arise as additive groups.   For example,  $C_\Omega$ is now the set of $\alpha-\beta$ with distinct $\alpha,\beta \in \Omega$.  Additive notation will make it notational easier to work with the $\RR$-vector space $X(T)_\RR:=X(T)\otimes_\ZZ \RR$ and the $\QQ$-vector space $X(T)_\QQ:=X(T)\otimes_\ZZ \QQ$.  

In \S\ref{SS:special case}, we will prove the three propositions under  additional assumptions.   The general case will be considered in \S\ref{SS:general case} where we will reduce to the special case considered in  \S\ref{SS:special case}.

\subsubsection{Special case} \label{SS:special case}
Throughout \S\ref{SS:special case}, we further suppose that $G$ is semisimple and almost simple, and that the representation $\rho$ is irreducible.   By \emph{almost simple}, we mean that the quotient of $G$ by its (finite) center is a simple algebraic group.  

To ease notation, set $W:=W(G,T)$ and $X:=X(T)$.   Let $r$ be the rank of $G$.   Since $G$ is almost simple, the root system $R:=R(G,T)$ in $X_\RR:=X(T)\otimes_\ZZ \RR$ is irreducible.

\begin{lemma}  \label{L:W irreducible}
The representation $X_\RR$ of the group $W$ is irreducible.   The subspace of $X_\RR$ fixed by $W$ is $0$.
\end{lemma}
\begin{proof}
Take any subspace $\calV\subseteq X_\RR$ stable under the action of $W$.   Let $(\,,\,)$ be an inner product on $X_\RR$ that is invariant under the action of $W$.  Let $\calV^\perp$ be the orthogonal complement of $\calV$ with respect to this inner product.   For any root $\alpha \in R$, the reflection $s_\alpha$ of $X_\RR$ defined by $s_\alpha(v)=v- (v,\alpha)/(\alpha,\alpha)\cdot \alpha$ lies in $W$.  

Take any root $\alpha\in R$ not in $\calV^\perp$; we have $(v,\alpha)\neq 0$ for some $v\in \calV$.  Therefore, $\alpha = (\alpha,\alpha)/(v,\alpha)\cdot (s_\alpha(v)-v)$ is an element of $\calV$, where we have used that $\calV$ is stable under the action of $W$.   

Therefore, every root $\alpha\in R$ either lies in $\calV$ or $\calV^\perp$.   That $R$ is an irreducible root system in $X_\RR$ implies that $\calV=0$ or $\calV=X_\RR$.   This implies that $X_\RR$ is an irreducible representation of $W$.

Finally, suppose $X_\RR^W \neq 0$.  Since $X_\RR$ is an irreducible representation of $W$, we deduce that $W$ acts trivially on $X_\RR$ which contradicts that $W$ acts faithfully on $X$.   Therefore, $X_\RR^W=0$.
\end{proof}

\begin{lemma}
Algorithm~\ref{algorithm 1} terminates with $s=1$.  We have 
\[
\calS_1 = C_{\Omega_V} = \{ \alpha-\beta: \alpha,\beta \in \Omega_V, \alpha\neq \beta\}
\]
and $\calS_1$ spans $X_\RR$.
\end{lemma}
\begin{proof}
Since $V$ is irreducible, it must be a minuscule representation of $G$ by our assumption on $\rho$.  Therefore, $W$ acts transitively on $\Omega_V$.  So in Algorithm~\ref{algorithm 1}, we have $\mathcal{U}_1=C_{\Omega_V}$.

 We have $|\Omega_V|>1$ since $\rho$ is irreducible and it image is a nontrivial semisimple group (by assumption the kernel of $\rho$ is central).  Since $|\Omega_V|>1$, we find that $\mathcal{U}_1$ is nonempty.   Choose any $W$-orbit $\OO$ of $\mathcal{U}_1$ that has minimal cardinality.    We have $\OO \subseteq C_{\Omega_V}$ and $\Omega_V$ is a transitive $W$-set.
 
 Let $\calV$ be the span of $\OO$ in $X(T)\otimes_\ZZ \QQ$.  The set $\OO$, and hence also the vector space $\calV$, is stable under the action of $W$.  Since $\OO\neq \{0\}$, we must have $\calV=X(T)\otimes_\ZZ \QQ$ by Lemma~\ref{L:W irreducible}.
So the set $\calS_1$ of elements in $C_{\Omega_V}$ that lie in $\calV$ is simply $C_{\Omega_V}$.  Since $\OO \subseteq C_{\Omega_V}$, we find that $\calS_1$ spans $\calV\otimes_\QQ \RR = X_\RR$.

Let $\mathcal{U}_{2}$ be the set of elements of $\mathcal{U}_1$ that are not in the span of $\calS_1$ in $\calV$; it is empty.  The lemma follows by applying Algorithm~\ref{algorithm 1} with the above computations.
\end{proof}

Since $\calS_1$ spans $X_\RR$, the quotient $W_1$ of $W$ that acts faithfully on $\calS_1$ is simply $W$.    The dimension $r_1$ of the span of $\calS_1$ in $X_\RR$ is  equal to $r$.   Therefore, $W_1=W$ and $r_1=r$.  We set $\calS:=\calS_1$.\\

Since $R\subseteq X_\RR$ is an irreducible root system, it must be isomorphic to precisely one of the following: 
\begin{align} \label{E:Lie types}
\text{$A_r$, \quad$B_r$ ($r\geq 3$), \quad$C_r$ ($r\geq 2$), \quad$D_r$ ($r\geq 4$),\quad $E_6$, \quad $E_7$, \quad $E_8$,\quad $F_4$,\quad $G_2$;}
\end{align}
note that the constraints on $r$ are added to avoid exceptional isomorphisms.  See \cite{MR0240238} for a description of these irreducible root systems.

Once a choice of positive roots $R^+$ of $R\subseteq X_\RR$ has been made, the irreducible representation $\rho$ of $G$ will have a highest weight $\varpi \in \Omega_V \subseteq X_\RR$.   The actual choice of positive roots does not matter for us since we are interested in the set of weights
\[
\Omega_V = W\cdot \varpi;
\]
recall $W$ acts transitively on $\Omega_V$ since the representation $V$ is irreducible and hence minuscule.

We can now assume that the root system $R$ is one of (\ref{E:Lie types}) with explicit definitions as in \cite{MR0240238}*{Planche I--VI, p.~250--266}.    Bourbaki chooses a set of positive roots $R^+$ and \cite{MR0453824}*{VIII, \S7, no.~3} then gives all highest weights that could arise from a nontrivial irreducible minuscule representation.   This data has been collected in Table~\ref{table:minuscule} below.

For each irreducible root system, we give the highest weights of irreducible minuscule representations with notation as in \cite{MR0240238}*{Planche I--VI}; for the classical types, we will later express these weights in terms of roots.  Note that the root systems $G_2$, $F_4$ and $E_8$ do not have minuscule weights.  In Table~\ref{table:minuscule}, we also given the dimension of the corresponding representation and the order of the Weyl group.  

\begin{table}[ht]
\begin{tabular}{|c|c|c|c|c|c|c|c|}\hline \label{T:table 1} root system & $A_r$ & $B_r$ & $C_r$ & $D_r$ & $D_r$ & $E_6$ & $E_7$ \\\hline 
highest weight & $\varpi_s$ ($1\leq s\leq r$) & $\varpi_r$ & $\varpi_1$ & $\varpi_1$ & $\varpi_{r-1},\varpi_r$ & $\varpi_1$, $\varpi_6$ & $\varpi_7$\\
dimension & $\binom{r+1}{s}$ & $2^r$ & $2r$ & $2r$ & $2^{r-1}$ & $27$ & $56$\\
$|W|$ & $(r+1)!$ & $2^r r!$ & $2^r r!$ & $2^{r-1} r!$ & $2^{r-1} r!$ & $51840$ & $2903040$
 \\
\hline 
 \end{tabular}
\caption{Irreducible root systems and minuscule weights.}
 \label{table:minuscule}
\end{table}
The proof of the propositions, in the special setting of \S\ref{SS:special case}, can now be reduced to verifying the cases from Table~\ref{table:minuscule}.   For each of the root systems $R$ and highest weight $\varpi$ occurring in Table~\ref{table:minuscule}, we compute
\[
\Omega:=W\cdot \varpi,
\]
where $W$ is the Weyl group of $R$.  Now define
\[
\calS:=\{\alpha-\beta : \alpha,\beta \in \Omega,\, \alpha\neq \beta\}.
\]
Case by case, we will check below that $R$ can be determined from the action of $W$ on $\calS$ as stated in Proposition~\ref{P:Lie 3}.  Also we will verify that $\calS$ has two $W$-orbits when $R=C_r$ ($r\geq 2$) and at least three $W$-orbits when $R=B_r$ ($r\geq 3$).    This will complete the proof of Proposition~\ref{P:Lie 3} in the special case of \S\ref{SS:special case}.    Since $s=1$ and $R\subseteq \calS$, this will also complete the proof of Proposition~\ref{P:Lie 1} in the setting of \S\ref{SS:special case}. 

Finally, Proposition~\ref{P:Lie 2} in the setting of \S\ref{SS:special case} is now immediate; note that $r$ as given and $|W|$ will distinguish the cases except for $B_r$ and $C_r$.\\

\noindent $\bullet$
\textbf{Case 1: $R=A_r$ ($r\geq 1$).}

Let $\calV$ be the subspace of $\RR^{r+1}$ consisting of $x\in \RR^{r+1}$ satisfying $\sum_{i=1}^{r+1} x_i = 0$.   Let $e_1,\ldots, e_{r+1}$ be the standard basis of $\RR^{r+1}$.    Then 
\[
R=\{ e_i-e_j : i\neq j\}
\]
is a root system of type $A_r$ in $\calV$.  The Weyl group $W$ acts on $e_1,\ldots, e_{r+1}$ by all permutations.   

We have $\varpi_s=(e_1+\cdots + e_s) - \tfrac{s}{r+1}(e_1+\cdots+e_{r+1})$ for an integer $1\leq s\leq r$.    Therefore, 
\[
\Omega :=W\varpi_s=\Big\{ {\sum}_{i\in I} e_i - \frac{s}{r+1}(e_1+\cdots+e_{r+1}) : I\subseteq \{1,\ldots,r+1\},\, |I|=s\Big\}.
\] 
One can verify that the set $\calS$ consists of 
\begin{align} \label{E:weight Ar}
\alpha = {\sum}_{i\in I} e_i - {\sum}_{j\in J} e_j
\end{align}
with $I$ and $J$ disjoint subsets of $\{1,\ldots, r+1\}$ of equal cardinality $m$, where $1\leq m \leq r+1-s$.  The $W$-orbit of $\alpha \in \calS$ given by (\ref{E:weight Ar}) depends only on $m$ and has cardinality
\[
\binom{r+1}{m} \binom{r+1-m}{m} = \frac{(r+1)!}{(r+1-m)!m!}\frac{(r+1-m)!}{(r+1-2m)! m!} = \binom{r+1}{2m} \binom{2m}{m}.
\]
It is a straightforward exercise to verify that $\binom{r+1}{2m} \binom{2m}{m} \neq \binom{r+1}{2} \binom{2}{1}=r(r+1)$ for all $1< m \leq r$; in fact, we have $\binom{r+1}{2m} \binom{2m}{m} > \binom{r+1}{2} \binom{2}{1}$ whenever $r\geq 6$.   Therefore, $R$ is the unique $W$-orbit in $\calS$ of cardinality $r(r+1)$.\\

\noindent $\bullet$
\textbf{Case 2: $R=B_r$ ($r\geq 3$).}    

Let $e_1,\ldots, e_{r}$ be the standard basis of $\RR^{r}$.    Then 
\[
R=\{\pm e_i\} \cup\{\pm e_i \pm e_i : i <j \}
\]
is a root system of type $B_r$ in $\RR^r$.  The Weyl group $W$ acts on $\pm e_1,\ldots, \pm e_r$ via signed permutations.   Note that $R$ is the union of two $W$-orbits with cardinality $2r$ and $2r(r-1)$.

We have $\varpi_r= \frac{1}{2}(e_1+\cdots +e_r)$ and hence
\[
\Omega:=W\varpi_r=\big\{\tfrac{1}{2}(\pm  e_1 \pm \cdots \pm e_r)\big\}.
\]
The set $\calS$ consists of elements of the form $\alpha:=\sum_{i \in I} \varepsilon_i e_i$, where $I$ is a nonempty subset of $\{1,\ldots, r\}$ and $\varepsilon_i\in \{\pm 1\}$.   Note that the $W$-orbit $W\alpha$ depends only on $m:=|I|\geq 1$ and we have
\[
|W\alpha| =  \binom{r}{m} \cdot 2^m.
\]
Using $1\leq m \leq r$, it is a straightforward exercise to show that $m=1$ and $m=2$ give the unique $W$-orbits in $\calS$ of cardinality $2r$ and $2r(r-1)$, respectively. Therefore, $R$ is the union of the $W$-orbits with $m=1$ and $m=2$.  Finally, observe that $\calS$ consists of precisely $r$ orbits of $W$; they are indexed by $1\leq m \leq r$.\\

\noindent $\bullet$
\textbf{Case 3: $R=C_r$ ($r\geq 2$). }   

Let $e_1,\ldots, e_{r}$ be the standard basis of $\RR^{r}$.    Then 
\[
R=\{ \pm e_i \pm e_j  :  i<j  \} \cup \{ \pm 2 e_i \}
\]
is a root system of type $C_r$ in $\RR^r$.   The Weyl group $W$ acts on $\pm e_1,\ldots, \pm e_r$ via signed permutations.   Note that $R$ is the union of two $W$-orbits with cardinality $2r$ and $2r(r-1)$.

We have $\varpi_1= e_1$ and hence
\[
\Omega:=W\cdot \varpi_1= \{\pm e_1,\ldots, \pm e_r\}.
\]
In this case, one can check directly that $\calS=\{\alpha-\beta: \alpha,\beta\in \Omega,\, \alpha\neq \beta\}$ equals $R$.  Observe that $\calS=R$ consists of two $W$-orbits.
\\

\noindent $\bullet$
\textbf{Case 4: $R=D_r$ ($r\geq 4$) and highest weight $\varpi_1$. }

Let $e_1,\ldots, e_{r}$ be the standard basis of $\RR^{r}$.    Then 
\[
R= \{\pm e_i \pm e_j: i <j \}
\]
is a root system of type $D_r$ in $\RR^r$.   The Weyl group $W$ acts on $\pm e_1,\ldots, \pm e_r$ via signed permutations that only change an even number of signs. 

We are considering the minuscule weight $\varpi_1$ in this case.  We have $\varpi_1=e_1$ and hence
\[
\Omega:=W\cdot \varpi_1= \{\pm e_1,\ldots, \pm e_r\}.
\]
The set $\calS$ is the union of two $W$-orbits $\{ \pm e_i \pm e_j  :  i<j  \}$ and  $\{ \pm 2 e_i \}$ which have cardinality $2r(r-1)$ and $2r$, respectively.   Therefore, $R$ is the unique $W$-orbit in $\calS$ of cardinality $2r(r-1)$.
\\

\noindent $\bullet$
\textbf{Case 5: $R=D_r$ $(r\geq 4)$ and highest weight $\varpi_{r-1}$ or $\varpi_r$.}

Let $e_1,\ldots, e_{r}$ be the standard basis of $\RR^{r}$.    Then 
\[
R= \{\pm e_i \pm e_j : i <j \}
\]
is a root system of type $D_r$ in $\RR^r$.   The Weyl group $W$ acts on $\pm e_1,\ldots, \pm e_r$ via signed permutations that only change an even number of signs. 

We are considering a minuscule weight $\varpi\in\{\varpi_{r-1},\varpi_r\}$ in this case.  We have 
\[
\varpi=\frac{1}{2}(e_1+\cdots+ e_{r-1}+\varepsilon e_r),
\]
where $\varepsilon=-1$ if $\varpi=\varpi_{r-1}$ and $\varepsilon=1$ if $\varpi=\varpi_r$.  Therefore, 
\[
\Omega :=W\varpi=\big\{(\varepsilon_1  e_1 + \cdots +\varepsilon_r e_r)/2: \varepsilon_i\in \{\pm 1\},\, {\prod}_i \varepsilon_i=\varepsilon \big\}.
\] 
One can check that $\calS$ consists of elements of the form 
\[
\alpha:=\sum_{i \in I} \varepsilon_i e_i
\] 
with $\varepsilon_i\in \{\pm 1\}$, where $I$ is a nonempty subset of $\{1,\ldots, r\}$ of \emph{even} cardinality, and $\prod_{i=1}^r \varepsilon_i = \varepsilon$ if $I=\{1,\ldots, r\}$.

 Note that the $W$-orbit $W\alpha$ depends only on the even integer $m:=|I|\geq 1$.  We have
\[
|W\alpha| =  \binom{r}{m} \, 2^{m}
\]
if $m<r$ and $|W\alpha| =  2^{r-1}$ if $m=r$.  
Using that $1\leq m \leq r$ is even, it is a straightforward exercise to show there is a unique $W$-orbit in $\calS$ of cardinality $\binom{r}{2} \, 2^{2} = 2r(r-1)$.  
This orbit is  $\{\pm e_i \pm e_i : i <j \}=R$.    
\\

\noindent $\bullet$
\textbf{Case 6: $R=E_6$ or $R=E_7$.}   

This case is a direct and not particularly interesting computation.  We have also used \texttt{Magma} \cite{MR1484478}, which contains the information on the irreducible root systems from Bourbaki \cite{MR0240238}, to verify this case.  

The following \texttt{Magma} code checks the $E_7$ case; it shows that $\calS$ is the union of $W$-orbits of cardinality $56$, $126$ and $756$, and that $R$ is the $W$-orbit of cardinality $126$.

\begin{verbatimtab}
	RR:=RootSystem("E7");
	R:=Roots(RR);
	W:=ReflectionGroup(RR);
	DynkinDiagram(RR); // same numbering as Bourbaki!

	// minuscule highest weight from Bourbaki in terms of simple roots
	pi:=Vector([2/2,3/2,4/2,6/2,5/2,4/2,3/2]);	
	Omega:=pi^W; 
	S:={a-b : a,b in Omega | a ne b};	
	while #S ne 0 do
		O:=Rep(S)^W;
		print #O, O eq R;
		S:=S diff O;
	end while;
\end{verbatimtab}

The two $E_6$ cases are computed in the same manner.  Besides replacing ``E7'' with ``E6'' in the above code, one also uses one of the two vectors:
\[
(4/3,3/3,5/3,6/3,4/3,2/3),\quad (2/3,3/3,4/3,6/3,5/3,4/3);
\]
 they are the coordinates of $\varpi_1$ and $\varpi_6$ with respect to the simple roots as chosen in \cite{MR0240238}*{Planche V}.  In both cases, we find that $\calS$ is the union of two $W$-orbits of cardinality $72$ and $270$.   The set $R$ is the orbit with $72$ elements.

\begin{remark}
Due to Proposition~\ref{P:minuscule MT}(\ref{P:minuscule MT ii}), we know that the $E_6$ and $E_7$ cases will not occur in our application to Mumford--Tate groups of abelian varieties.
\end{remark}

\subsubsection{General case} \label{SS:general case}

We now consider the general case.  The goal is to reduce to the special case of \S\ref{SS:special case}.

Let $G_0$ be the neutral component of the center of $G$; it is a torus.    Let $G^\der$ be the derived subgroup of $G$; it is a semisimple group.     Let $G_1,\ldots, G_m$ be the nontrivial normal subgroups of $G^\der$ that are connected and minimal with respect to inclusion; these groups are almost simple, i.e., the center $Z(G_i)$ of $G_i$ is finite and the quotient $G_i/Z(G_i)$ is a simple algebraic group.  The group $G$ is an \defi{almost direct product} of $G_0,\ldots, G_m$, i.e., the morphism
\[
\varphi\colon G_0\times G_1\times \cdots \times G_m \to G
\]
defined by multiplication is a surjective homomorphism whose kernel is finite and contained in the center, cf.~\cite{MR1278263}*{Proposition~2.4}.  

For each $0\leq i \leq m$, define $T_i := G_i \cap T$; it is a maximal torus of $G_i$.   Note that $T_0\times \cdots\times T_m$ is a maximal torus of $G_0\times\cdots\times G_m$ and $\varphi(T_0\times \cdots\times T_m)=T$.  We then have a homomorphism 
\[
X(T) \to X(T_0\times \cdots \times T_m) =  X(T_0)\oplus \cdots \oplus X(T_m),\quad \alpha \mapsto \alpha \circ \varphi.
\]
Tensoring up to $\QQ$ gives an isomorphism
\[
X(T)_\QQ = X(T_0)_\QQ \oplus \cdots \oplus X(T_m)_\QQ,
\]
where $X(T_i)_\QQ:=X(T_i)\otimes_\ZZ \QQ$; we will use this isomorphism  as an identification.  We thus have a disjoint union
\begin{align} \label{E:root decomposition}
R(G,T)= \bigcup_{i=1}^m R(G_i,T_i);
\end{align}
this is the decomposition of $R(G,T)$ into irreducible root systems. 

For each $0\leq i \leq m$, we have a natural map $W(G,T)\to W(G_i,T_i)$.  This induces an isomorphism
\[
W(G,T) = W(G_0,T_0)\times \cdots \times W(G_m,T_m)
\]
that we will also use an identification.   From the action of $W(G,T)$ on $X(T)_\QQ$, we see that every $W(G,T)$-orbit $\OO \subseteq X(T)_\QQ$ is of the form 
\[
\OO_0\oplus \cdots \oplus \OO_m :=\{\alpha_0+\cdots+\alpha_m : \alpha_i \in \OO_i\}
\]
for unique $W(G_i,T_i)$-orbits $\OO_i$ in $X(T_i)_\QQ$.  

\begin{lemma}
\label{L:Lie theory core}
We have $s\leq m$.  After reordering the groups $G_1,\ldots, G_m$, the following will hold for all $1\leq i \leq s$:
\begin{alphenum}
\item  
\label{L:Lie theory core a}
The set $\calS_i$ spans the $\QQ$-vector space $X(T_i)_\QQ$. 
\item
\label{L:Lie theory core b}
There is an irreducible and minuscule representation $V_i$ of $G_i$ such that $\calS_i=\{\alpha-\beta : \alpha,\beta\in \Omega_{V_i},\, \alpha\neq \beta\}$, where $\Omega_{V_i}\subseteq X(T_i)$ is the set of weights of $T_i$ acting on $V_i$.
\end{alphenum}
\end{lemma}
\begin{proof}
Take any $1 \leq i \leq s$.     Assume, after possibly reordering the groups $G_1,\ldots,G_m$,  that $\calS_j$ spans $X(T_j)_\QQ$ for all $1\leq j< i$.   Note that this assumption is vacuous if $i=1$.

We claim that the following hold:
\begin{itemize}
\item $i\leq m$,
\item 
$\calS_i$ spans $X(T_j)_\QQ$ for some $i\leq j\leq m$,
\item
$\calS_i=\{\alpha-\beta : \alpha,\beta\in \Omega_U,\, \alpha\neq \beta\}$ for some irreducible and minuscule representation $U$ of $G_j$.
\end{itemize}    
We could further assume that $j=i$ after swapping $G_i$ and $G_j$.

Note that once the claim has been proved, the lemma will follow by induction on $i$.\\

We now prove the claim.  In the construction of $\calS_i$, we chose a $W(G,T)$-orbit $\OO\subseteq\mathcal{U}_i$ of minimal cardinality and a $W(G,T)$-orbit $\Omega\subseteq \Omega_V$ for which $\OO\subseteq C_\Omega$.    The orbit $\Omega$ is equal to $\Omega_U$ for some irreducible representation $U\subseteq V$ of $G$; this makes use of our minuscule assumption on $\rho$.

Since $U$ is an irreducible representation of $G$, there are irreducible representations $U_j$ of $G_j$ for all $0\leq j \leq m$ such that the representation $U_0\otimes \cdots \otimes U_m$ of $G_0\times \cdots \times G_m$ is isomorphic to the representation $U$ of $G_0\times \cdots \times G_m$ obtained by composing $\varphi$ with the representation of $G$.    

 We have 
\begin{equation} \label{E:omegaU}
\Omega=\Omega_U  = \{\alpha_0+\cdots +\alpha_m : \alpha_j \in \Omega_{U_j}\},
\end{equation}
where $\Omega_{U_j} \subseteq X(T_j)$ is the set of weights of $G_j$ acting on $U_j$ with respect to $T_j$.  The group $W(G,T)$ acts transitively on $\Omega_U$ since $U$ is minuscule by our assumptions on $\rho$.   Therefore, $W(G_j,T_j)$-acts transitively on $\Omega_{U_j}$ for each $j$, i.e., $U_j$ is a minuscule representation of $G_j$.

From (\ref{E:omegaU}), we have
\[
C_\Omega  =  \Big\{ \sum_{j=1}^m (\alpha_j-\beta_j) : \alpha_j,\beta_j \in \Omega_{U_j}\Big\} \setminus \{0\};
\]
we do not need the $j=0$ term since $|\Omega_{U_0}|=1$ ($U_0$ is an irreducible representation of the torus $G_0$ and hence is $1$-dimensional).  
Therefore, the $W(G,T)$-orbit $\OO\subseteq C_\Omega$ is equal to 
\[
\{\alpha_1+\ldots+\alpha_m : \alpha_j \in \OO_j\},
\] 
where $\OO_j$ is a $W(G_j,T_j)$-orbit in $\{\alpha-\beta: \alpha,\beta \in \Omega_{U_j}\}\subseteq X(T_j)_\QQ$.   In particular, $|\OO|=\prod_{j=1}^m |\OO_j|$.

Suppose that $\OO_j = \{0\}$ for all $i\leq j \leq m$.    Then $\OO$ is contained in $\bigoplus_{1\leq j <i } X(T_j)_\QQ$ which is the span of $\bigcup_{1\leq j <i} \calS_j$ in $X(T)_\QQ$ by the assumption at the beginning of the lemma.  This implies that $\OO$ is not a subset of the set $\mathcal{U}_i$ from Algorithm~\ref{algorithm 1}.  However, this contradicts our choice of $\OO$.  Therefore, $\OO_j \neq \{0\}$ for some $i \leq j \leq m$.   In particular, $i\leq m$.

  We may thus assume that $\OO_i \neq \{0\}$  after possibly swapping $G_i$ and $G_j$.  We have $\OO_i\subseteq \mathcal{U}_i$ since $\OO_i$ is not contained in $\bigoplus_{1\leq j <i } X(T_j)_\QQ=\operatorname{span}_{X(T)_\QQ}\big( \bigcup_{1\leq j <i} \calS_j\big)$.   Since $|\OO|=\prod_{j=1}^m |\OO_j|$ and $\OO_i\subseteq \mathcal{U}_i$, the minimality condition in our choice of $W(G,T)$-orbit $\OO$ implies that $|\OO_j|=1$ for all $j \in \{1,\ldots, m\}-\{i\}$.   We have $\OO_j=\{0\}$ for all $j \in \{1,\ldots, m\}-\{i\}$, since otherwise we would have a nonzero element of $X(T_j)_\QQ$ fixed by $W(G_j,T_j)$ which contradicts Lemma~\ref{L:W irreducible}.   Therefore, $\OO=\OO_i$.

Since $\OO=\OO_i\neq \{0\}$ is a $W(G_i,T_i)$-set, it spans $X(T_i)_\QQ$ by Lemma~\ref{L:W irreducible}.   Recall that $\calS_i$ is the set of elements of $C_\Omega$ that lie in the span of $\OO$ in $X(T)_\QQ$, i.e., $X(T_i)_\QQ$.  Therefore, $\calS_i$ equals $\{\alpha-\beta : \alpha,\beta\in \Omega_{U_i},\, \alpha\neq \beta\}$ and spans $X(T_i)_\QQ$.    We have now verified the claim.
\end{proof}

We shall assume that the groups $G_1,\ldots, G_m$ have been reordered so that the conditions of Lemma~\ref{L:Lie theory core} hold.  

\begin{lemma} \label{L:m eq s}
We have $m=s$.
\end{lemma}
\begin{proof}
We have $s\leq m$ by Lemma~\ref{L:Lie theory core}.    Suppose that $s<m$ and hence there is an integer $s<i \leq m$.

Take an irreducible $U\subseteq V$ representation of $G$.  There are irreducible representations $\rho_j\colon G_j\to \GL_{U_j}$ for all $0\leq j \leq m$ such that the representation $U_0\otimes \cdots \otimes U_m$ of $G_0\times \cdots \times G_m$ is isomorphic to the representation $U$ of $G_0\times \cdots \times G_m$ obtained by composing $\varphi$ with the representation $U$ of $G$.   

Since $G_i$ is almost simple, the kernel of $\rho_i$ is finite or $G_i$.  Therefore, $U\subseteq V$ can be chosen so that $\rho_i\colon G_i \to \GL_{U_i}$ has finite kernel; otherwise, this would contradict that the kernel of $\rho$ is in the center of $G$.     In particular, $\Omega_{U_i} \neq \{0\}$ is a $W(G,T)$-orbit in $\Omega_V$ and is a subset of $X(T_i)_\QQ$.   The set $C_{\Omega_{U_i}} \subseteq X(T_i)_\QQ$ is nonempty; otherwise $\Omega_{U_i}$ would consist of one nonzero element stable under the action of $W(G_i,T_i)$-action and this would contradict Lemma~\ref{L:W irreducible}. 

By Lemma~\ref{L:Lie theory core}, the set $\bigcup_{1\leq j \leq s} \calS_j$ spans $\oplus_{1\leq j\leq s} X(T_j)_\QQ$. However, $\oplus_{1\leq j\leq s} X(T_j)_\QQ$ does not contain $C_{\Omega_{U_i}}$ since $i>s$.   In Algorithm~\ref{algorithm 1}, this implies that $C_{\Omega_{U_i}}\subseteq \mathcal{U}_{s+1}$ and in particular $\mathcal{U}_{s+1}\neq \emptyset$.  However, $\mathcal{U}_{s+1}$ is empty by the definition of $s$.   This contradicts our initial assumption that $s<m$.  Since $s\leq m$, we deduce that $s=m$.
\end{proof}

The decomposition of $R(G,T)$ into irreducible root systems is $R(G,T)= \bigcup_{i=1}^s R(G_i,T_i)$ by (\ref{E:root decomposition}) and Lemma~\ref{L:m eq s}.    For each $1\leq i \leq s$, we have $R(G_i,T_i) \subseteq X(T_i)_\QQ$ and $X(T_i)_\QQ$ is spanned by $\calS_i$ by Lemma~\ref{L:Lie theory core}(\ref{L:Lie theory core a}).\\

Take any $1\leq i \leq s$,   By Lemma~\ref{L:Lie theory core}(\ref{L:Lie theory core b}), there is an irreducible and minuscule representation 
\[
\rho_i\colon G_i \to \GL_{V_i} 
\]
such that $\calS_i=\{\alpha-\beta : \alpha,\beta\in \Omega_{V_i},\, \alpha\neq \beta\}$, where $\Omega_{V_i}\subseteq X(T_i)$ is the set of weights of $T_i$ acting on $V_i$.  The representation $\rho_i$ is nontrivial since otherwise $\calS_i=\{\alpha-\beta : \alpha,\beta\in \Omega_{V_i},\, \alpha\neq \beta\}$ is empty which is impossible by the construction of the sets $\calS_i$.    Since $G_i$ is almost simple, we deduce that the kernel of $\rho_i$ is contained in the center of $G_i$.

Now observe that the assumptions in the beginning of \S\ref{S:roots} hold with $(G,T,\rho)$ replaced by $(G_i,T_i, \rho_i)$.    Applying Algorithm~\ref{algorithm 1} to the triple $(G_i,T_i, \rho_i)$, we obtain $s=1$ and the new set ``$\calS_1$'' agrees with our set $\calS_i$.    The triple $(G_i,T_i, \rho_i)$ satisfies the conditions of \S\ref{SS:special case} and hence Propositions~\ref{P:Lie 1}, \ref{P:Lie 2} and \ref{P:Lie 3} hold for it.  

By Proposition~\ref{P:Lie 1} in the setting of \S\ref{SS:special case}, we find that $R(G_i,T_i)\subseteq \calS_i$.    The set $\calS_i$ does not contain $R(G_j,T_j)$ for any $j\neq i$ since it is a subset of $X(T_i)_\QQ$.   This completes the proof of Proposition~\ref{P:Lie 1}

Let $W_i$ be the quotient of $W(G,T)$ that acts faithfully on $\calS_i$.  Since $\calS_i \subseteq X(T_i)_\QQ$, $W_i$ is also the quotient of $W(G_i,T_i)$ that acts faithfully on $\calS_i$.  Let $r_i$ be the dimension of the span of $\calS_i$ in $X(T)_\QQ$; equivalently, in $X(T_i)_\QQ$.  Propositions~\ref{P:Lie 2} and \ref{P:Lie 3} now follow directly from the special cases of the propositions that hold for $(G_i,T_i, \rho_i)$ with $1\leq i \leq s$.

\section{Finding root datum} \label{S:root datum of MT}

Fix a nonzero abelian variety $A$ over a number field $K$. 
Assume that Conjectures~\ref{C:MT} and \ref{C:Frob conj} hold for $A$.  Denote by $r$ the rank of $\MT_A=G_A^\circ$.  The common rank of the $\ell$-adic monodromy groups $G_{A,\ell}^\circ$ is $r$ since we are assuming the Mumford--Tate conjecture.  Let $\calS_A$ be the set of prime ideals of $\OO_K$ from Definition~\ref{D:SA}; it has density $1/[K_A^\conn:K]$ by Lemma~\ref{L:density of SA}(\ref{L:density of SA i}).  \\

The main goal of this section is to explain how one can compute the root datum of $\MT_A$ directly from the Frobenius polynomials of $P_{A,\q}(x)$ and $P_{A,\p}(x)$ for two appropriately chosen prime ideals $\q$ and $\p$ of $\OO_K$.   

More precisely, we will compute the abstract root datum $\Psi(\MT_A)$ along with the homomorphism $\mu_{\MT_A}\colon \Gal_\QQ \to \Out(\Psi(\MT_A))$ from \S\ref{SS:Galois action}.     By Proposition~\ref{P:characterize inner forms}, this information determines $\MT_A$ up to an inner twist.  Along with this  root data, we will also compute the set of weights, with multiplicities, of the natural representation $\MT_A \subseteq \GL_V$.

\subsection{Choice of primes} \label{SS:choice of primes}

We first select our two prime ideals.
\begin{alphenum}
\item \label{I:choice of q}
Let $\q$ be any prime ideal in $\calS_A$.
\item \label{I:choice of p}
Let $\p$ be any prime ideal in $\calS_A$ for which $[L(\calW_{A,\p}):L] = |W(\MT_A)|$, where $L:=\QQ(\calW_{A,\q})$.
\end{alphenum}

There are many possibilities for $\q$ since $\calS_A$ has positive density.  The following lemma shows that for a fixed $\q$, there are also many possibly $\p$.

\begin{lemma} \label{L:p exists}
Fix $\q\in \calS_A$ and set $L=\QQ(\calW_{A,\q})$.  Then the set of prime ideals $\p \in \calS_A$ that do {not} satisfy (\ref{I:choice of p}) has density $0$.
\end{lemma}
\begin{proof}
Let $G$ be the quasi-split inner form of the Mumford--Tate group $\MT_A$.   Let $k_G$ be the subfield of $\Qbar$ fixed by $\ker \mu_G$ with notation as in \S\ref{SS:Galois action}; it agrees with the similarly defined field $k_{\MT_A}$ by Proposition~\ref{P:characterize inner forms}.   

First suppose that $K_A^\conn=K$.       Since $\q\in \calS_A$, we have a maximal torus $T_\q$ of $G$ as in Theorem~\ref{T:Frobenius tori}.   The splitting field of the torus $T_\q$ is $L=\QQ(\calW_{A,\q})$ and hence  $k_G \subseteq L$.  Corollary~\ref{C:Frobenius Galois groups} implies that $\Gal(L(\calW_{A,\p})/L) \cong W(\MT_A)$ for all nonzero prime ideals $\p$ of $\OO_K$ away from a set of density $0$; note that this uses our assumption $K_A^\conn=K$.  This proves the lemma in the special case $K_A^\conn=K$.

We now consider the general case.     Let $A'$ be the base change of $A$ to $K_A^\conn$.  Note that $A$ and $A'$ have the same Mumford--Tate group and hence the same $G$.  Take any prime ideal $\p \in \calS_{A}$ and choose a prime ideal $\P$ of $\OO_{K_A^\conn}$ dividing $\p$.   The prime $\p$ splits completely in $K_A^\conn$ by Lemma~\ref{L:common rank}(\ref{L:common rank i}) and hence $N(\P)=N(\p)$.    By Lemma~\ref{L:density of SA}(\ref{L:density of SA ii}), we have $\P \in \calS_{A'}$ and $P_{A',\P}(x)=P_{A,\p}(x)$.    In particular, $L(\calW_{A',\P})=L(\calW_{A,\p}) $.  Therefore, we have an inequality
\begin{align} \label{E:density conn}
&|\{ \p \in \calS_A : N(\p)\leq x,\, [L(\calW_{A,\p}):L] \neq |W(\MT_A)| \}| \\
\notag
\leq & |\{ \P \in \calS_{A'} : N(\P)\leq x,\, [L(\calW_{A',\P}):L] \neq |W(\MT_{A'})|\}|.
\end{align}
A similar argument as above shows that $L=\QQ(\calW_{A',\mathfrak{Q}})$ for any prime ideal $\mathfrak{Q}$ of $\OO_{K_A^\conn}$ that divides $\q$.    By the connected case of the lemma that we have already proved, we have
\[
|\{ \P \in \calS_{A'} : N(\P)\leq x,\, [L(\calW_{A',\P}):L] \neq  | W(\MT_{A'})|\}| = o(x/\log x).
\]
The general case of the lemma then follows from the inequality (\ref{E:density conn}).
\end{proof}

For the rest of \S\ref{S:root datum of MT}, we fix prime ideals $\p$ and $\q$ as in (\ref{I:choice of q}) and (\ref{I:choice of p}) above.  For our later algorithmic considerations, we assume that the Frobenius polynomials $P_{A,\q}(x)$ and $P_{A,\p}(x)$ have been computed.

\subsection{Frobenius torus at \texorpdfstring{$\p$}{p}} \label{SS:Frob torus at p}

Let $G$ be the quasi-split inner form of the Mumford--Tate group $\MT_A$ and let
\[
\rho\colon G_{\Qbar}\xrightarrow{\sim} (\MT_A)_{\Qbar} \subseteq \GL_{V_A,\Qbar},
\] 
be a representation  where the isomorphism is one arising from $G$ being an inner twist of $\MT_A$.     \commentr{$V_A$}
  
By Theorem~\ref{T:Frobenius tori}, there is a maximal torus $T_\p$ of $G$ for which we have an isomorphism $X(T_\p)=\Phi_{A,\p}$ of $\Gal_\QQ$-modules which we will use as an identification.     Explicitly, there is a $t_\p \in T_\p(\QQ)$ as in Theorem~\ref{T:Frobenius tori} such that the isomorphism is given by $\alpha\mapsto \alpha(t_\p)$.

\begin{remark}
We will construct the root datum of $G$ with respect to the maximal torus $T_\p$.   Our identification $X(T_\p)=\Phi_{A,\p}$ gives a tangible place to start since the group $\Phi_{A,\p}$ with its Galois action are computable directly from the polynomial $P_{A,\p}(x)$, cf.~\S\ref{SS:computing PhiAp}.  
\end{remark}

 Moreover, under this identification, we can assume by Theorem~\ref{T:Frobenius tori} that $\Omega_\p:=\calW_{A,\p}$ is the set of weights in $X(T_\p)$ of the representation $\rho$.   The multiplicity of a weight $\alpha$ agrees with the multiplicity of it viewed as a root of $P_{A,\p}(x)$.

\subsection{Computing the Weyl group}

Define the field $L:=\QQ(\calW_{A,\q})$ and the Galois group 
\[
W:=\Gal(L(\calW_{A,\p})/L).
\]  
The group $W$ acts faithfully on $\calW_{A,\p}$ and hence also faithfully on the group $\Phi_{A,\p}$.   Using our identification $X(T_\p)=\Phi_{A,\p}$, we can thus view $W$ as a subgroup of $\Aut_\ZZ(X(T_\p))$.   We can also identify $W(G,T_\p)$ with a subgroup of $\Aut_\ZZ(X(T_\p))$.   The following lemma shows that $W$ recovers the Weyl group of $G$.

\begin{lemma}
We have $W=W(G,T_\p)$.
\end{lemma}
\begin{proof}
  Since $\q\in \calS_A$, we have a maximal torus $T_\q$ of $G$ as in Theorem~\ref{T:Frobenius tori}.   The splitting field of the torus $T_\q$ is the field $L$.  With notation as in \S\ref{S:Galois groups of Frob}, we have $k_G \subseteq L$ and hence $\varphi_\p(\Gal_L) \subseteq W(G,T_\p)$.  Using our identification of $W$ with a subgroup of $\Aut_\ZZ(X(T_\p))$, we have $W \subseteq W(G,T_\p)$.   By our choice of $\p$, we find that $|W|=[L(\calW_{A,\p}):L]$ equals $|W(\MT_A)|=|W(G,T_\p)|$.    Therefore, $W=W(G,T_\p)$.
\end{proof}

\subsection{Computing the roots} \label{SS:computing roots}

So far we have described how to compute the following from our polynomials $P_{A,\q}(x)$ and $P_{A,\p}(x)$:
\begin{itemize}
\item
$X(T_\p)$ as a $\Gal_\QQ$-module,  
\item
the Weyl group $W(G,T_\p)$ and its action on $X(T_\p)$, 
\item 
the set of weights $\Omega_\p \subseteq X(T_\p)$, with multiplicities, of the faithful representation $\rho\colon G_{\Qbar}\hookrightarrow (\GL_{V_A})_\Qbar= \GL_{V_A\otimes_\QQ \Qbar}$.   
\end{itemize} 
We now explain how to compute the set $R:=R(G,T_\p)\subseteq X(T_\p)$ of roots of $G$ with respect to $T_\p$. \\

By Proposition~\ref{P:minuscule MT}, every irreducible representation $U \subseteq V_A\otimes_\QQ \Qbar$ of $G_{\Qbar}$ is minuscule.     We are thus in the setting of \S\ref{S:roots} with the representation $\rho\colon G_{\Qbar} \hookrightarrow \GL_{V_A\otimes_\QQ \Qbar}$.   Applying Algorithm~\ref{algorithm 1}, we obtain nonempty finite subsets $\calS_1,\ldots, \calS_s$ of $X(T_\p)$ for some integer $s\geq 0$.   Note this algorithm only makes use of the set $\Omega_\p$ and the action of $W(G,T_\p)$ on $X(T_\p)$.  

By Proposition~\ref{P:Lie 1}, there is a unique irreducible component $R_i$ of the root system $R(G,T_\p)\subseteq X(T_\p)\otimes_{\ZZ} \RR$ with $R_i\subseteq \calS_i$ for each $1\leq i \leq s$.  Moreover, 
\[
R(G,T_\p)= \bigcup_{1\leq i \leq s} R_i
\]
is a disjoint union and is the decomposition into irreducible root systems.   By Proposition~\ref{P:Lie 2} and \ref{P:Lie 3}, we can compute the Lie type of each $R_i$ and then compute $R_i$ itself; note that these propositions only requires the sets $\calS_i \subseteq X(T_\p)$ and the action of $W(G,T_\p)$.  We can thus compute $R(G,T_\p)$!

\subsection{Computing the root datum of the Mumford--Tate group} \label{SS:computing MT}

Define the group $Y:=\Hom_\ZZ(X(T_\p),\ZZ)$; it has a natural pairing with $X(T_\p)$ and we can identify it with the group of cocharacters $X^\vee(T_\p)$.   Lemma~\ref{L:root datum from W} implies that the root datum $\Psi(G,T_\p)$ can be determined from the action of $W(G,T_\p)$ on $X(T_\p)$ and the set of roots $R(G,T_\p)$; moreover, the proof shows how to construct the bijective map $\alpha\mapsto \alpha^\vee$ from $R(G,T_\p)$ to the set of coroots $R^\vee(G,T_\p)$.   We thus know how to compute the root datum $\Psi(G,T_\p)$ since from \S\ref{SS:computing roots} we have already found $X(T_\p)$, $R(G,T_\p)$ and the action of $W(G,T_\p)$ on $X(T_\p)$.    Since we now know the root datum $\Psi(G,T_\p)$ and the action of $\Gal_\QQ$ on $X(T_\p)$, we can also compute $\mu_G \colon \Gal_\QQ \to \Out(\Psi(G,T_\p))$.  

Since $\MT_A$ is an inner form of $G$, we may take $\Psi(\MT_A)$ to be $\Psi(G,T_\p)$.  The homomorphism $\mu_{\MT_A}\colon \Gal_\QQ\to \Out(\Psi(\MT_A))$ agrees with $\mu_G$ by Lemma~\ref{P:characterize inner forms}.   

Finally note that $\Omega_\p \subseteq X(T_\p)$ gives the weights of the representation $\MT_A \subseteq \GL_V$.  The multiplicity of a weight $\alpha\in \Omega_\p$ of $\rho$ agrees with the multiplicity of $\alpha$ as a root of $P_{A,\p}(x)$ (recall we are identifying $\Omega_\p$ with $\calW_{A,\p}$).

\subsection{Computing the root datum of the Hodge group} \label{SS:root data for Hg}

The Hodge group $\Hg_A$ of $A$ was defined in \S\ref{SS:MT and Hodge groups}.   Recall that $\MT_A$ contains  the group of scalars in $\GL_V$ and that $\MT_A=\GG_m \cdot \Hg_A$. 

We first describe an algebraic subgroup $H$ of $G$ that is a quasi-split inner form of $\Hg_A$.   We have $\MT_A\subseteq \GSp_{V,E}$ and $\Hg_A = \MT_A \cap \Sp_{V,E}$, where $E$ is an appropriate nondegenerate alternating pairing on $V$.    Let $\nu\colon \GSp_{V,E}\to (\GG_m)_\QQ$ be the similitude character. Define the homomorphism 
\[
\mu:= \nu\circ \rho \colon G_{\Qbar}\to (\GG_m)_{\Qbar}.
\]    
For each $\sigma\in \Gal_\QQ$, we have
\[
\sigma(\mu)=\nu \circ \sigma(\rho) \circ \rho^{-1} \circ \rho= \nu\circ \rho=\mu,
\]
where we have used that $\sigma(\rho)\circ \rho^{-1}$ is an inner automorphism of $(\MT_A)_{\Qbar}$ and hence $\nu\circ (\sigma(\rho)\circ \rho^{-1})= \nu|_{(\MT_A)_{\Qbar}}$.    Therefore,  $\mu$ arises by base change from a unique homomorphism $G\to \GG_m$ defined over $\QQ$ that we also denote by $\mu$.    Let $H \subseteq G$ be the kernel of $\mu$.   The homomorphism $\rho$ induces an isomorphism between $H_{\Qbar}$ and $(\Hg_A)_{\Qbar}$.  In particular, $H$ is a connected and reductive group defined over $\QQ$.  The group $H$ is an inner form of $\Hg_A$ via $\rho$; note that each inner automorphism of $(\MT_A)_{\Qbar}$ arises from conjugation by some $\Hg_A(\Qbar)$ since $\MT_A=\GG_m\cdot \Hg_A$.   The centers of $\MT_A$ and $G$ are naturally isomorphic since they are inner forms of each other, so we can view the group $\GG_m$ of homotheties as a subgroup of $G$ and hence $G=\GG_m\cdot H$.   Since $G$ is quasi-split, we find that $H$ is a quasi-split inner form of $\Hg_A$ and is given by the isomorphism $\rho|_{H_\Qbar} \colon H_\Qbar \to (\Hg_A)_\Qbar$.

Thus to compute the abstract root datum $\Psi(\Hg_A)$ and the homomorphism $\mu_{\Hg_A}\colon \Gal_\QQ \to \Out(\Psi(\Hg_A))$, it suffices to compute them for the group $H$.    Let $T$ be the unique maximal torus of $H$ contained in $T_\p$.    Restriction to $T_{\Qbar}$ defines a surjective homomorphism
\begin{align} \label{E:X quotient}
X(T_\p) \to X(T)
\end{align}	
that respects the $\Gal_\QQ$-actions.   Using that $G=\GG_m\cdot H$, we find that (\ref{E:X quotient}) gives a bijection between $R(G,T_\p)$ and $R(H,T)$.      The Weyl group $W(G,T_\p)$ fixes the kernel of (\ref{E:X quotient}) and the induced action on $X(T)$ is faithfully.   The action of the group $W(G,T_\p)$ on $X(T)$ agrees with $W(H,T)$.    The set of roots $R(H,T) \subseteq X(T)$ and the action of $W(H,T)$ on $X(T)$ determine the root datum $\Psi(H,T)$ by Lemma~\ref{L:root datum from W} (and the proof shows how to compute it).  The homomorphism (\ref{E:X quotient}) respects the $\Gal_\QQ$-actions, so from $\Psi(H,T)$ and the $\Gal_\QQ$-action on $X(T_\p)$, we can compute $\mu_H\colon \Gal_\QQ \to \Out(\Psi(H))$.\\

We have already described how to compute $W(G,T_\p)$ and $R(G,T_\p)$.   So to make everything computable, it remains to describe the homomorphism (\ref{E:X quotient}) with respect to our identification $X(T_\p)=\Phi_{A,\p}$.   The following lemma shows that we may take $X(T)=\Phi_{A,\p}/\langle N(\p) \rangle$.
 
\begin{lemma} \label{L:kernel Np}
Under our identification $X(T_\p)=\Phi_{A,\p}$, the restriction homomorphism $X(T_\p)\to X(T)$, $\alpha\mapsto \alpha|_{T_\Qbar}$ has kernel $\langle N(\p) \rangle$.
\end{lemma}
\begin{proof}
Let $\beta\colon T_\p \to \GG_m$ be the character defined by $\beta=\mu|_{T_\p}$.  The kernel of $\beta$ is a commutative algebraic subgroup of $H$ containing $T$.  Since maximal tori of reductive groups are their own centralizers, we find that $T=\ker\beta$.   So the kernel of $X(T_\p)\to X(T)$ is $\langle \beta \rangle$.   With our identification $X(T_\p)=\Phi_{A,\p}$ given by $\alpha\mapsto \alpha(t_\p)$, it thus suffices to show that $\beta(t_\p)=N(\p)$.

The eigenvalues of an element in $\Sp_{V,E}(\Qbar)$ occur in inverse pairs.  Since the characteristic polynomial of $\rho(t_\p)$ is $P_{A,\p}(x)$ and the roots of $P_{A,\p}(x)$ have absolute value $N(\p)$ (with respect to any embedding into $\CC$), we deduce that $\pm \rho(t_\p) /\sqrt{N(\p)}$ are precisely the elements of $\Hg_A(\Qbar)$ that are scalar multiples of $\rho(t_\p)$.   Therefore, $\nu(\pm \rho(t_\p)/\sqrt{N(\p)})=1$ and hence $\nu(\rho(t_\p))=N(\p)$.   In particular, $\beta(t_\p)=\mu(t_\p)=\nu(\rho(t_\p))=N(\p)$. 
\end{proof}

Finally, the image $\Omega \subseteq X(T)$ of $\Omega_\p$ under the restriction map $X(T_\p)\to X(T)$ is the set of weights of the representation $\rho|_{H_{\Qbar}}$ with respect to $T$.   

We claim that the map $\Omega_\p\to \Omega$ is a bijection that respects multiplicities.   If the claim fails, then by Lemma~\ref{L:kernel Np} and our identification $\Omega_\p=\calW_{A,\p}$ there are distinct $z$ and $z' \in \calW_{A,\p}$ such that  $z' z^{-1} \in \langle N(\p) \rangle$. Since  $z$ and $z'$ are roots of $P_{A,\p}(x)$, they have absolute value $\sqrt{N(\p)}$ under any embedding into $\CC$.   Therefore, $z=z'$ which contradicts that they are distinct and proves the claim.

\subsection{Proof of Theorem~\ref{T:main}} \label{SS:main proof}
Recall that $G_A^\circ=\MT_A$.  Let $\Sigma$ be the set of nonzero prime ideals $\p$ of $\OO_K$ that split completely in $K_A^\conn$.  By Lemma~\ref{L:common rank}, we have $\calS_A \subseteq \Sigma$ and the set $\Sigma - \calS_A$ has density $0$.   So it suffices to proof the theorem with any $\q\in  \calS_A$.    Let $S_2$ be the set of $\p\in \calS_A$ for which condition (\ref{I:choice of p}) in \S\ref{SS:choice of primes} does not hold for our fixed $q$.   The set $S_2$ has density 0 by Lemma~\ref{L:p exists}.  So it suffices to proof the theorem with any choice of $\p \in \calS_A-S_2$ since $\Sigma - (\calS_A-S_2)$ has density $0$.

Our primes $\q$ and $\p$ satisfies condition  (\ref{I:choice of q}) and  (\ref{I:choice of p}) in \S\ref{SS:choice of primes}.  In \S\S\ref{SS:Frob torus at p}--\ref{SS:computing MT}, we have explained how given the polynomials $P_{A,\p}(x)$ and $P_{A,\q}(x)$, one can compute the abstract root datum $\Psi(\MT_A)$,  the homomorphism $\varphi_G\colon \Gal_\QQ\to\Out(\Psi(\MT_A))$, and
 the set of weights $\Omega$ of the representation $\MT_A\subseteq \GL_{V_A}$ and their multiplicities.  This proves Theorem~\ref{T:main}.
 
\subsection{Proof of Theorem~\ref{T:ST}}
The theorem is an immediate consequence of Theorem~\ref{T:main} and the construction of \S\ref{SS:root data for Hg}.  Note that the part concerning the group $\ST_A^\circ$ follows since it is a maximal compact subgroup of $\Hg_A(\CC)$.

\section{Some computational remarks}  \label{S:computational remarks}

Fix a nonzero abelian variety $A$ over a number field $K$.   When trying to implement the algorithm described in \S\ref{SS:algorithms}, the most computationally intensive part is computing the splitting field of the polynomials $P_{A,\p}(x)$.  

\subsection{Computing \texorpdfstring{$\Phi_{A,\p}$}{Phi{A,p}}, revisited} \label{SS:computing Phi again}
  
Let $\p$ be a nonzero prime ideal of $\OO_K$ for which $A$ has good reduction.  In \S\ref{SS:computing PhiAp}, we described how to compute the group $\Phi_{A,\p}$ from $P_{A,\p}(x)$.    The method outlined involves computing a splitting field $L/\QQ$ of $P_{A,\p}(x)\in \ZZ[x]$ and working with $\Gal(L/\QQ)$.    Unfortunately, computing splitting fields can be extremely time consuming.   

For simplicity, we now assume further that $A$ has ordinary reduction at $\p$; this can be verified by checking that the middle coefficient of $P_{A,\p}(x)$ is $0$ modulo $\p$.   We will now explain how to compute $\Phi_{A,\p}$ without computing a splitting field $L$ of $P_{A,\p}(x)$ over $\QQ$.\\

The function \texttt{GaloisGroup} implemented in \texttt{Magma} is suitable for our purposes, see \cite{MR3838226} for an overview of the algorithm.  The function produces the distinct roots $\pi_1,\ldots, \pi_n$ of $P_{A,\p}(x)$ in a splitting field $L'/\QQ_\ell$ for some prime $\ell$; more precisely, approximations are produced for the $\pi_i$ which  can be computed to arbitrary accuracy.  Using the numbering of these roots, there is an injective homomorphism $\Gal(\QQ(\pi_1,\ldots, \pi_n)/\QQ)\hookrightarrow S_n$.   The function \texttt{GaloisGroup} also produces the image $\Gamma \subseteq S_n$ of this homomorphism.  Note that the splitting field $L'$ of $P_{A,\p}(x)$ over $\QQ_\ell$ is easy to compute. This Galois group algorithm also has the advantage that it  does not impose any restrictions on the degree of $P_{A,\p}(x)$.\\

For each root $\alpha$ of $P_{A,\p}(x)$, $N(\p)/\alpha$ is also a root.   We have $\alpha^2\neq N(\p)$ for each root $\alpha$ of $P_{A,\p}(x)$ since otherwise $A$ would not have ordinary reduction at $\p$.   So the integer $n$ is even and we can assume that the roots are chosen so that $\alpha_i \alpha_{i+n/2}=N(\p)$ for all $1\leq i \leq n/2$.  Let $\OO_{L'}$ be the valuation ring of $L'$.  There is enough flexibility in the above algorithm to ensure that $\ell$ is odd, $L'/\QQ_\ell$ is unramified, and that all the $\pi_i$ lie in $\OO_{L'}^\times$. \\

Given $\pi_1,\ldots, \pi_n \in L'$ and $\Gamma$ as above, the rest of \S\ref{SS:computing Phi again} is devoted to computing $\Phi_{A,\p}$.  

\begin{lemma} \label{L:set S from ordinary}
There is a subset $S \subseteq \{1,\ldots, n\}$ satisfying the following conditions:
\begin{alphenum} 
\item \label{I:set S from ordinary a}
For each $1\leq i \leq n/2$, exactly one of the two values $i$ and $i+n/2$ are in $S$.   
\item \label{I:set S from ordinary b}
For $e\in \ZZ^n$, $\prod_{i=1}^n \pi^{e_i}$ is a root of unity if and only if $\sum_{i\in \sigma(S)} e_i=0$ for all $\sigma\in \Gamma$.
\end{alphenum}
\end{lemma}
 \begin{proof}
 This is an immediate consequence of Lemma~\ref{L:root of unity linear} and our assumption that $A$ is ordinary at $\p$.   Note that for any prime ideal $\lambda$ of $\OO_L$ dividing $N(\p)$, with $L=\QQ(\pi_1,\ldots, \pi_n)$, we have $\{ v_\lambda(\pi_i), v_\lambda(\pi_{i+n/2})\} = \{0, v_\lambda(N(\p))\}$.
 \end{proof}
 
For a tuple $e\in \ZZ^n$, the following lemma will allow us to determine when $\prod_{i=1}^n \pi_i^{e_i}$ equals $1$ by using our approximations of the $\pi_i$ in the $\ell$-adic field $L'$. 
 
\begin{lemma} \label{L:computing products}
Take any $e\in \ZZ^n$.   Then $\prod_{i=1}^n \pi_i^{e_i}=1$ if and only if $\prod_{i=1}^n \pi_i^{e_i} - 1 \in \ell^m \OO_{L'}$ holds for some integer
\begin{align} \label{E:how big m}
m> |\Gamma| \,\frac{ \log 2 + \max\{\sum_{i=1}^n a_i,\sum_{i=1}^n b_i\}/2 \cdot \log N(\p) }{f \log \ell},
\end{align}
where $a_i := \max\{ e_i, 0 \}$, $b_i :=\max\{ -e_i,0\}$ and $f:=[L':\QQ_\ell]$.
\end{lemma}
\begin{proof}
Define the algebraic integer $\alpha:= \prod_{i=1}^n \pi^{a_i}- \prod_{i=1}^n \pi^{b_i}$ in the number field $L:=\QQ(\pi_1,\ldots, \pi_n)$.  By Weil, we know that each $\pi_i$ has absolute value $N(\p)^{1/2}$ under any embedding $L\hookrightarrow \CC$.   Therefore,
\begin{align} \label{E:norm of alpha}
|N_{L/\QQ}(\alpha)| \leq (2 N(\p)^{\max\{\sum_i a_i, \sum_i b_i\}/2})^{|\Gamma|}.
\end{align}
Using the inclusion $L\subseteq L'$, we can identify $L'$ with the completion $L_\lambda$ of $L$ at some prime ideal $\lambda$ of $\OO_L$ dividing $\ell$.

First suppose that $\prod_{i=1}^n \pi_i^{e_i} - 1 \in \ell^m \OO_{L'}$ for an integer $m$ satisfying (\ref{E:how big m}).    Since the $\pi_i$ are units in $\OO_{L'}$ and $L'/\QQ_\ell$ is unramified, this implies that $\alpha \in \lambda^m$.  Therefore, $N(\lambda^m)=\ell^{fm}$ divides the integer $|N_{L/\QQ}(\alpha)|$.  So if $\alpha\neq0$, then $\ell^{fm} \leq |N_{L/\QQ}(\alpha)|$.  However, by (\ref{E:how big m}) and (\ref{E:norm of alpha}), we find that $\ell^{fm}>|N_{L/\QQ}(\alpha)|$.  Therefore, $\alpha=0$ and hence $\prod_i \pi_i^{e_i}=1$.

The other implication in the lemma is trivial.
\end{proof}

Observe that the above lemma gives a way to determine if $\prod_{i=1}^n \pi_i^{e_i} =1$ holds for a fixed $e\in \ZZ$.   By increasing the accuracy of the approximations of $\pi_i$, one can assume that $\pi_i + \ell^m \OO_{L'}$ is known for the smallest integer $m$ satisfying (\ref{E:how big m}).   Since the $\pi_i$ are units in $\OO_{L'}$, one can then compute $\prod_{i=1}^n \pi_i^{e_i} + \ell^m \OO_{L'}$.  Using Lemma~\ref{L:computing products}, we can then verify whether or not we have $\prod_{i=1}^n \pi_i^{e_i} =1$.  
\\
 
Take any subset $S \subseteq \{1,\ldots, n\}$ satisfying condition (\ref{I:set S from ordinary a}) of Lemma~\ref{L:set S from ordinary}.       We now explain how to verify if it satisfies condition (\ref{I:set S from ordinary b}).    Define the finite field $\FF=\OO_{L'}/\ell \OO_{L'}$.  Let $M_S$ be the group consisting of $e\in \ZZ^n$ satisfying $\sum_{i\in \sigma(S)} e_i=0$ for all $\sigma\in \Gamma$.  Let 
\[
\varphi\colon M_S \to \FF^\times 
\]
be the homomorphism that sends a tuple $e\in \ZZ^n$ to the image of $\prod_{i=1}^n \pi_i^{e_i} \in \OO_{L'}^\times$ in $\FF^\times$.  Let $M$ be the kernel of $\varphi$.  By our choice of extension $L'/\QQ_\ell$, the roots of unity in $L'$ can be distinguished by their images in $\FF^\times$.   So condition (\ref{I:set S from ordinary b}) holds for the set $S$ if and only if $\prod_{i=1}^n \pi_i^{e_i} = 1$ for a set of $e$ that generate the group $M$.  Using the remarks following Lemma~\ref{L:computing products}, one can thus verify whether condition (\ref{I:set S from ordinary b}) holds or not.
\\

One can thus compute the set $\scrS$ of subsets $S\subseteq \{1,\ldots, n\}$ that satisfy conditions (\ref{I:set S from ordinary a}) and (\ref{I:set S from ordinary b}) of Lemma~\ref{L:set S from ordinary}.  For making this computation practical, note that for a set $S\subseteq \{1,\ldots, n\}$, we have $S \in \scrS$ if and only if $\sigma(S) \in \scrS$ for all $\sigma\in \Gamma$.  Let $M'$ be the subgroup of $\ZZ^n$ generated by the sets $M_S$ with $S\in \scrS$.     For $e\in \ZZ^n$, Lemma~\ref{L:set S from ordinary} implies that $\prod_{i=1}^n \pi_i^{e_i}$ is a root of unity if and only if $e\in M'$.   Arguing as above, one can compute the subgroup $M\subseteq M'$ such that $\prod_{i=1}^n \pi_i^{e_i}=1$  if and only if $e\in M$.   Therefore, the map $\ZZ^n \to \Phi_{A,\p}$, $e\mapsto \prod_{i=1}^n \pi_i^{e_i}$ induces an isomorphism
\begin{align} \label{E:revisted Phi}
\ZZ^n/M \xrightarrow{\sim} \Phi_{A,\p}.
\end{align}
This is our explicit description of the group $\Phi_{A,\p}$.  

Observe that the action of $\Gamma\subseteq S_n$ on $\ZZ^n$ induces an action on $\ZZ^n/M$.   With respect to the embedding $\Gal(\QQ(\pi_1,\ldots, \pi_n)/\QQ)\xrightarrow{\sim} \Gamma \subseteq S_n$ and the isomorphism (\ref{E:revisted Phi}), this describes the Galois action on $\Phi_{A,\p}$.  

\subsection{Remarks on computing root data} \label{SS:computing Psi}

With notation and assumptions as in \S\ref{S:root datum of MT}, fix primes $\p$ and $\q$ as in \S\ref{SS:choice of primes}.     In \S\ref{SS:computing Phi again}, we described how to compute the group $\Phi_{A,\p}$ and the group $\Gamma=\Gal(\QQ(\calW_{A,\p})/\QQ)$ acting on it.   Note that with the method described, the splitting field $\QQ(\calW_{A,\p})$ of $P_{A,\p}(x)$ is not actually computed; the roots of $P_{A,\p}(x)$ are given in some nonarchimedean field and the action of $\Gamma$ is given as a permutation of its roots.

We now describe how to compute the subgroup $W:=\Gal(L(\calW_{A,\p})/L)$ of $\Gamma$, where $L:=\QQ(\calW_{A,\q})$.  We may assume that $P_{A,\p}(x)$ and $P_{A,\q}(x)$ are relatively prime; this is automatic if the residue fields of $\p$ and $\q$ have different characteristics.     Let $\Gamma_0$ be the Galois group of $P_{A,\p}(x) \cdot P_{A,\q}(x)$ over $\QQ$.   Let $m$ and $n$ be the number of distinct roots of $P_{A,\p}(x)$ and $P_{A,\q}(x)$, respectively.  As in \S\ref{SS:computing Phi again}, we can compute $\Gamma_0$ as a subgroup of $S_m \times S_n$, where the first and second factors describe the Galois action on the distinct roots of $P_{A,\p}(x)$ and $P_{A,\q}(x)$, respectively, with roots given explicitly in a suitable local field.   We can identify $\Gamma$ with the image of $\Gamma_0$ under the projection $\varphi\colon S_n\times S_m \to S_n$, $(a,b)\mapsto a$.   The subgroup $W$ of $\Gamma$ is then the image under $\varphi$ of the group $\Gamma_0 \cap (S_n\times \{1\})$.

We can then  proceed as in \S\ref{SS:algorithms}, to compute the root datum $\Psi(G_A^\circ)$ up to isomorphism, 
 the \emph{image} of the homomorphism $\mu_{G_A^\circ} \colon \Gal_\QQ \to \Out(\Psi(G_A^\circ))$, and 
the set of weights of the representation $G_A^\circ \subseteq \GL_{V_A}$ and their multiplicities.  Note that the image of $\mu_{G_A^\circ}$ is given by $\Gamma/W$.

\section{Proof of Proposition~\ref{P:endomorphism}} \label{S:endomorphism proof}
We have fixed an embedding $\Kbar\subseteq \CC$ and $\End(A_{\Kbar})=\End(A_\CC)$.   So there is no harm in replacing $A$ by its base extension by $\CC$; this does not change the group $\MT_A$ or its representation $V_A$. After replacing $A$ by an isogenous abelian variety, we may assume that $A = \prod_{i=1}^s A_i$ with $A_i=C_i^{e_i}$, where $e_i\geq 1$ and the $C_i$ are simple abelian varieties over $\CC$ that are pairwise nonisogenous.   We have
\[
\End(A)\otimes_\ZZ \QQ \cong B_1\times \cdots \times B_s,
\]
where $B_i:= \End(A_i)\otimes_\ZZ \QQ$.  Note that $B_i$ is a central simple algebra over its center $L_i$.  The field $L_i$ is a number field and let $m_i$ be the positive integer satisfying $\dim_{L_i} B_i = m_i^2$.\\

We first assume that $A$ is a simple abelian variety over $\CC$, i.e., $s=1$ and $e_1=1$.   In particular, $B_1= \End(A)\otimes_\ZZ \QQ$ is a division algebra with center $L_1$. Define $r:=[L_1:\QQ]$.  Recall that we can identify $B_1$ with the subring of $\End_\QQ(V_A)$ that commutes with the action of $\MT_A$.  Since $B_1$ is a division algebra, we find that $V_A$ is an irreducible representation of $\MT_A$ .   

Observe that the algebraic group $\MT_A$ and its faithful representation $V_A$ satisfy the assumptions of \S3.2 of \cite{MR563476}.    From \S3.2 of \cite{MR563476} (especially the irreducible case considered before Proposition~8),  we find that the following hold:
\begin{itemize}
\item 
$\Gamma$ acts transitively on $\Omega$, 
\item
the representation $V_A \otimes_\QQ \Qbar$ of $(\MT_A)_{\Qbar}$ decomposes as the direct sum of $r$ irreducible and minuscule representations that each occur with multiplicity $m_1$,
\item
each weight in $\Omega$ of $V_A$ has multiplicity $m_1$.  
\end{itemize}
From the above minuscule property, we find that $\Omega$ consists of $r$ distinct $W$-orbits.    Fix a $W$-orbit $\OO\subseteq \Omega$ and let $H$ be the stabilizer of $\OO\subseteq \Omega$.   We have $[\Gamma:H]=r$ since $\Gamma$ acts transitively on $\Omega$ and hence also transitively on the $W$-orbits in $\Omega$.  Using the natural isomorphism $\Gal(k/\QQ)=\Gamma/W$, we let $L'$ be the subfield of $k$ corresponding to the group $H/W$; it is a number field of degree $r$.  Using the transitivity of the action of $\Gamma$ again, we find that the number field $L'$, up to isomorphism, does not depend on the choice of $\OO$.

In the setting of the proposition, we have $s=s'=1$, $\Omega_1=\Omega$, $\OO_1:=\OO$, $L_1'=L'$ and $m_1'=m_1$.   So to complete the proof in the simple case, it remains to show that the fields $L_1$ and $L'$ are isomorphic.  

Since $\Gal_{L'}$ stabilizes the $W$-orbit $\OO$, we find that there is a representation $U\subseteq V_A \otimes_\QQ L'$ of $(\MT_A)_{L'}$ whose set of weights is $\OO$ and each occurs with multiplicity $m_1$ (in particular, $U\otimes_{L'} \Qbar \subseteq V_A \otimes_\QQ \Qbar$ is an isotypic representation of $(\MT_A)_{\Qbar}$).    So $L\otimes_\QQ L'$ acts on the $L'$-vector space $U$ by homotheties.   In particular, the homomorphism $L\otimes_\QQ L' \to L'$ induced by this action gives an embedding $L\hookrightarrow L'$ of fields.   Since $L$ and $L'$ are both number fields of  degree $r$, we deduce that they are isomorphic.
\\

Now consider the case where $s=1$, i.e., $A = C_1^{e_1}$ with $e_1\geq 1$ and $C_1$ a simple abelian variety over $\CC$.   We have $V_A = V_{C_1}^{\oplus e_1}$ and the action of $\MT_{C_1}$ on $V_A$ induces an isomorphism $\MT_{C_1}=\MT_A$.   In particular, $\MT_{C_1}$ and $\MT_A$ have the same root datum and the same set of weights $\Omega$.  For $\alpha\in \Omega$, the multiplicity of $\alpha$ as a weight of $\MT_A$ acting on $V_A$ is $e_1$ times the multiplicity as a weight of $\MT_{C_1}$ acting on $V_{C_1}$.    So using the previous case, it suffices to show that the rings $B_1=\End(A)\otimes_\ZZ \QQ$ and $B':=\End(C_1)\otimes_\ZZ \QQ$ have isomorphic centers and that $\dim_\QQ B_1 = e_1^2 \dim_\QQ B'$.  This is clear since $B_1\cong M_{e_1}(B')$.\\

Finally, we consider the general case.   We have $V_A = \bigoplus_{i=1}^s V_{A_i}$ which induces a homomorphism $\MT_A \hookrightarrow \prod_{i=1}^s \MT_{A_i}$ such that each projection $\MT_A\to \MT_{A_i}$ is surjective.         Note that $V_{A_i}$ is a representation of $\MT_A$ and agrees with the action via the natural homomorphism $\MT_A\to \MT_{A_i}$.  This induces surjective homomorphisms \begin{align*}
W\to W(\MT_{A_i},T_i)\quad \text{ and }\quad \Gamma/W \to \Gamma(\MT_{A_i},T_i)/W(\MT_{A_i},T_i), 
\end{align*}
where $T_i$ is the image of $T$.    

Let $\Omega_i \subseteq X(T)$ be the set of weights of the representation $V_{A_i}$; it is stable under the $\Gamma$-action.   Choose any $W$-orbit $\OO_i \subseteq \Omega_i$ and let $H_i$ be the subgroup of $\Gamma$ that stabilizes $\OO_i$.  From the previous cases, we find that $\Gamma$ acts transitively on $\Omega_i$, each weight $\alpha\in \Omega_i$ of the representation $V_{A_i}$ has multiplicity $m_i$, and $L_i$ is isomorphic to the subfield of $k$ fixed by $H_i/W \subseteq \Gamma/W=\Gal(k/\QQ)$.

We claim that the sets $\Omega_i$ are pairwise disjoint.   Assuming the claim, we find that sets $\Omega_1,\ldots, \Omega_s$ are the $\Gamma$-orbits of $\Omega$ and each $\alpha\in \Omega_1$ has multiplicity $m_i$ as a weight of $V_A$.   By reordering the $A_i$, we may assume that these are the same $\Gamma$-orbits as in the setup of the proposition and hence $s'=s$.      We now have $m_i'=m_i$ and the field $L_i'$ is isomorphic to $L_i$.
 
It remains to prove the claim.  Suppose that there are nondisjoint $\Omega_i$ and $\Omega_j$ with $1\leq i < j \leq s$. We have $\Omega_i=\Omega_j$ since they are transitive $\Gamma$-sets.   The representations $V_{A_i}^{\oplus m_j}$ and $V_{A_j}^{\oplus m_i}$ of $\MT_A$ thus have the same weights and multiplicities, and so are isomorphic.   In particular, this implies that there is a nonzero linear map $V_{A_i}\to V_{A_j}$ that respects the $\MT_A$-actions.   However, since $\End(A)\otimes_\ZZ \QQ$ agrees with the subring of $\End(V_A)$ that commutes with the $\MT_A$-action, we deduce that there is a nonzero homomorphism $A_i\to A_j$ of abelian varieties.  This is impossible since $A_i$ and $A_j$ are powers of nonisogenous simple abelian varieties.  This contradiction proves the claim.

\bibliographystyle{plain}

\end{document}